\documentclass{amsart}
\usepackage{amsmath,amssymb,mathrsfs,bbm,yhmath,longtable,mathtools}
\usepackage{amsthm,rotating,xcolor,epsfig,hyperref,url}

\usepackage{graphicx}

\date{}

\def\prtor{\qopname\relax o{pr}_{\mathbb T^2}}

\theoremstyle{theorem}
\newtheorem{theo}{Theorem}[section]
\newtheorem{lemm}{Lemma}[section]
\newtheorem{prop}{Proposition}[section]

\theoremstyle{remark}

\newtheorem{exam}{Example}[section]

\theoremstyle{definition}
\newtheorem{defi}{Definition}[section]

\numberwithin{equation}{section}
\numberwithin{figure}{section}

\author {Ivan Dynnikov and Maxim Prasolov}
\address{\noindent V.A. Steklov Mathematical Institute of Russian Academy of Science, 8 Gubkina Str., Moscow 119991, Russia}
\address{\noindent St.\ Petersburg State University, Line 14th (Vasilyevsky Island), 29, Saint Petersburg, 199178, Russia}
\email{dynnikov@mech.math.msu.su}
\address{Department of Mechanics and Mathematics of Moscow State University, 1 Leninskije gory, Moscow 119991, Russia}
\address{\noindent St.\ Petersburg State University, Line 14th (Vasilyevsky Island), 29, Saint Petersburg, 199178, Russia}
\email{0x00002a@gmail.com}
\thanks{The work is supported by the Russian Science Foundation under grant~19-11-00151.}

\title{Rectangular diagrams of surfaces: the basic moves}

\begin{document}
\maketitle

\begin{abstract}
In earlier papers we introduced a 
representation of isotopy classes of compact surfaces embedded in the three-sphere~$\mathbb S^3$
by so called rectangular diagrams.
The formalism proved useful for comparing Legendrian knots.
The aim of this paper is to prove a Reidemeister type theorem for rectangular diagrams of surfaces.
\end{abstract}

\section{Introduction}
We work in the piecewise smooth category. Unless otherwise specified, all homeomorphisms
and isotopies are assumed to be piecewise smooth. 

Throughout the paper, \emph{a surface}~$F\subset\mathbb S^3$
means a compact smooth surface with corners embedded in the three-sphere~$\mathbb S^3$.
`With corners' means that the boundary of~$F$ is a union of cusp-free piecewise smooth curves,
and that~$F$ can be
extended beyond the boundary to become a smooth submanifold of~$\mathbb S^3$ with smooth boundary.
Surfaces are not assumed to be orientable or to have a non-empty boundary.

If the tangent plane~$T_pF$ to a surface~$F\subset\mathbb S^3$
at a point~$p\in F$
is said to be preserved by a self-homeomorphism~$\phi$ of~$\mathbb S^3$ this means
that~$\phi(p)=p$ and~$\phi(F)$ has a well defined tangent plane at~$p$ coinciding with~$T_pF$.

If~$F_1,F_2\subset\mathbb S^3$ are two surfaces, then
by \emph{a morphism from~$F_1$ to~$F_2$} we mean a connected component of
orientation preserving self-homeomorphisms~$\phi$ of~$\mathbb S^3$ such that~$\phi(F_1)=F_2$.
The morphism represented by a homeomorphism~$\phi$ will be denoted by~$[\phi]$.

In~\cite{dp17} we introduced rectangular diagrams of surfaces (see definitions below), and with every rectangular diagram of a surface~$\Pi$
we associated a surface~$\widehat\Pi\subset\mathbb S^3$. We also showed~\cite[Theorem~1]{dp17}
that every isotopy class of surfaces can be represented by a rectangular diagram of a surface.

In~\cite{distinguishing} we defined basic moves for rectangular diagrams of surfaces,
which are transformations of rectangular diagrams such that the associated surface
is transformed by an isotopy.
So, every basic move~$\Pi\mapsto\Pi'$ comes with a well defined morphism from~$\widehat\Pi$ to~$\widehat\Pi'$.

The main result of the present paper (which is announced in~\cite{distinguishing}
in a slightly weaker form) is Theorem~\ref{main-theo}, which states that
any morphism between surfaces represented by rectangular diagrams can be decomposed
into basic moves, and if the given morphism preserves a sublink of the surface boundary, then so does
each basic move in the decomposition.

The paper is organized as follows. In Section~\ref{prelim-sec} we introduce
rectangular diagrams, their moves, and formulate the main result of the paper.
In Sections~\ref{bubble-sec}--\ref{movie-moves-sec}
we introduce auxiliary tools and prove intermediate results.
Sections~\ref{fixed-b-sec} and~\ref{boundary-sec} contain the proof of the main result.

\section{Preliminaries. Rectangular diagrams and basic moves}\label{prelim-sec}

We recall some definitions and notation from~\cite{dp17,distinguishing}.

For two distinct points~$x,y$ of the circle~$S^1$ we denote by~$[x;y]$ a unique arc of~$\mathbb S^1$ such that,
with respect to the standard orientation of~$\mathbb S^1$, it has starting point~$x$,
and end point~$y$.

By~$\mathbb T^2$ we denote the two-torus~$\mathbb S^1\times\mathbb S^1$ and by~$\theta$ and~$\varphi$
the angular coordinates on the first and the second $\mathbb S^1$-factor, respectively.

We identify the three-sphere~$\mathbb S^3$ with the join of two circles:
$\mathbb S^3=\mathbb S^1\times\mathbb S^1\times[0;1]/{\sim}$, where~$\sim$ stands
for the following equivalence relation:
$$(\theta,\varphi,0)\sim(\theta',\varphi,0),\
(\theta,\varphi,1)\sim(\theta,\varphi',1)\quad\forall\,\theta,\theta',\varphi,\varphi'\in\mathbb S^1,$$
and use~$\theta,\varphi,\tau$ for the corresponding coordinate system. 

\emph{The torus projection} $\prtor$ is defined as the following map from~$\mathbb S^3\setminus\bigl(\mathbb S^1_{\tau=0}\cup
\mathbb S^1_{\tau=1}\bigr)$ to~$\mathbb T^2$:
$$\prtor(\theta,\varphi,\tau)=(\theta,\varphi).$$
For a point~$v\in\mathbb T^2$, we denote by~$\widehat v$ the closed arc
$$\widehat v=\overline{\prtor^{-1}(v)}.$$
For a finite subset~$X\subset\mathbb T^2$, we define~$\widehat X$ to be the union
$$\widehat X=\bigcup\limits_{v\in X}\widehat v.$$

For~$\theta,\varphi\in\mathbb S^1$ we also denote by $m_\theta$
the \emph{meridian}~$\{\theta\}\times\mathbb S^1\subset\mathbb T^2$, and by~$\ell_\varphi$
the \emph{longitude}~$\mathbb S^1\times\{\varphi\}$.
By $\widehat m_\theta$ and $\widehat\ell_\varphi$ we denote the endpoints
of the arc~$\widehat{(\theta,\varphi)}$, lying on~$\mathbb S^1_{\tau=1}$
and~$\mathbb S^1_{\tau=0}$, respectively, that is,
$$\widehat m_\theta=(\theta,*,1)/{\sim}\in\mathbb S^1_{\tau=1},\quad
\widehat\ell_\varphi=(*,\varphi,0)/{\sim}\in\mathbb S^1_{\tau=0}.$$

By a \emph{rectangle} we mean a subset~$r\subset\mathbb T^2$ of the form~$[\theta_1;\theta_2]\times[\varphi_1;\varphi_2]$.
By~$V(r)$ we denote the set of vertices of~$r$: $V(r)=\{\theta_1,\theta_2\}\times\{\varphi_1,\varphi_2\}$. We also set
$$V_\diagup(r)=\{(\theta_1,\varphi_2),(\theta_2,\varphi_1)\},\quad
V_\diagdown(r)=\{(\theta_1,\varphi_1),(\theta_2,\varphi_2)\}.$$

Two rectangles $r_1$, $r_2$ are said to be \emph{compatible}
if their intersection satisfies one of the following:
\begin{enumerate}
\item $r_1\cap r_2$ is empty;
\item $r_1\cap r_2$ is a subset of vertices of $r_1$ (equivalently: of~$r_2$);
\item $r_1\cap r_2$ is a rectangle disjoint from the vertices of both rectangles $r_1$ and $r_2$.
\end{enumerate}

\begin{defi}\label{rect-diagr-def}
\emph{A rectangular diagram of a surface} is a collection $\Pi=\{r_1,\ldots,r_k\}$
of pairwise compatible rectangles in~$\mathbb T^2$ such
that every meridian $\{\theta\}\times\mathbb S^1$ and every longitude $\mathbb S^1\times\{\varphi\}$
of the torus contains at most two free vertices, where by \emph{a free vertex}
we mean a point that is a vertex of exactly one rectangle in~$\Pi$.

The set of all free vertices of $\Pi$ is called \emph{the boundary of $\Pi$} and
denoted by $\partial\Pi$.

All elements of the union~$\bigcup_{r\in\Pi}V(r)$ are called \emph{vertices} of~$\Pi$.
The elements of the union~$\bigcup_{r\in\Pi}V_\diagup(r)$ (respectively, of~$\bigcup_{r\in\Pi}V_\diagdown(r)$)
are called \emph{$\diagup$-vertices} (respectively, \emph{$\diagdown$-vertices}) of~$\Pi$,
and said to be \emph{of type~$\diagup$} (respectively, \emph{of type~$\diagdown$}).

All meridians
and longitudes of~$\mathbb T^2$ passing through vertices of~$\Pi$ are called
\emph{occupied levels} of~$\Pi$.
\end{defi}

With every rectangle~$r\in\mathbb T^2$ one can associate a surface~$\widehat r$ homeomorphic to a two-disc
so that the following holds:
\begin{enumerate}
\item
for any rectangle~$r\subset\mathbb T^2$, the torus projection~$\prtor$ takes the interior of~$\widehat r$
to the interior of~$r$ homeomorphically;
\item
$\partial\widehat r=\widehat{V(r)}$;
\item
whenever~$r_1$ and~$r_2$ are compatible rectangles, the interiors of~$\widehat r_1$ and~$\widehat r_2$
are disjoint;
\item
whenever~$r_1$ and~$r_2$ are compatible and share a vertex, the union~$\widehat r_1\cup\widehat r_2$
is a surface (that is to say, there are no singularities at~$\widehat r_1\cap\widehat r_2$).
\end{enumerate}

In~\cite{dp17}, the discs~$\widehat r$ satisfying these properties are defined explicitly, but
any other choice will work as well. The choice is supposed to be fixed from now on.

\begin{defi}
Let~$\Pi$ be a rectangular diagram of a surface. \emph{The surface~$\widehat\Pi$ associated with~$\Pi$}
is defined as
$$\widehat\Pi=\bigcup\limits_{r\in\Pi}\widehat r.$$
\end{defi}

One can see that the boundary~$\partial\Pi$ of a rectangular diagram of a surface is a rectangular
diagram of a link in the sense of~\cite{dyn06,dp17}, and we have
$$\widehat{\partial\Pi}=\partial\widehat\Pi.$$

By \emph{a basic move} of rectangular diagrams of surfaces we mean
any of the transformations defined below in this section. These include: (half-)wrinkle creation and reduction moves, (de)stabilization
moves, exchange moves, and flypes.

`Transformations' means pairs~$(\Pi,\Pi')$ of diagrams endowed with a morphism~$\widehat\Pi\rightarrow\widehat\Pi'$.
In each definition of a basic move~$\Pi\mapsto\Pi'$ defined below, the description
of the pair~$(\Pi,\Pi')$ naturally suggests what the morphism~$\widehat\Pi\rightarrow\widehat\Pi'$ should be.
So, we omit the description of this morphism in Definitions~\ref{wrinkledef}--\ref{flype-def}
and provide the necessary hints afterwards.

\begin{defi}\label{wrinkledef}
Let $\Pi$ be a rectangular diagram of a surface, and let $v_1=(\theta_0,\varphi_1)$ and $v_2=(\theta_0,\varphi_2)$
be a $\diagdown$-vertex and a $\diagup$-vertex of $\Pi$, respectively, lying on the same meridian $m_{\theta_0}$.
Choose an $\varepsilon>0$ so that no meridian in $[\theta_0-2\varepsilon;\theta_0+2\varepsilon]\times
\mathbb S^1\subset\mathbb T^2$  other than $m_{\theta_0}$ is an occupied level of $\Pi$.
Also choose an orientation-preserving self-homeomorphism~$\psi$ of the interval~$[\theta_0-2\varepsilon;\theta_0+2\varepsilon]$.

Let $\Pi'$ be the rectangular diagram of a surface obtained from $\Pi$ by making the following modifications:
\begin{enumerate}
\item
every rectangle of the form $[\theta_0;\theta_1]\times[\varphi';\varphi'']$ (respectively, $[\theta_1;\theta_0]\times[\varphi';\varphi'']$)
with $[\varphi';\varphi'']\subset[\varphi_1;\varphi_2]$ is replaced by
$[\psi(\theta_0+\varepsilon);\theta_1]\times[\varphi';\varphi'']$ (respectively, by $[\theta_1;\psi(\theta_0+\varepsilon)]\times[\varphi';\varphi'']$);
\item
every rectangle of the form $[\theta_0;\theta_1]\times[\varphi';\varphi'']$ (respectively, $[\theta_1;\theta_0]\times[\varphi';\varphi'']$)
with $[\varphi';\varphi'']\subset[\varphi_2;\varphi_1]$ is replaced by
$[\psi(\theta_0-\varepsilon);\theta_1]\times[\varphi';\varphi'']$ (respectively, by $[\theta_1;\psi(\theta_0-\varepsilon)]\times[\varphi';\varphi'']$);
\item
the following two new rectangles are added: $[\psi(\theta_0-\varepsilon);\psi(\theta_0)]\times[\varphi_1;\varphi_2]$ and
$[\psi(\theta_0);\psi(\theta_0+\varepsilon)]\times[\varphi_2;\varphi_1]$.
\end{enumerate}
Then we say that the passage from $\Pi$ to $\Pi'$ is \emph{a vertical wrinkle creation move}.
The inverse operation is referred to as \emph{a vertical wrinkle reduction move}.

\emph{Horizontal wrinkle creation} and \emph{reduction moves} are defined similarly with the roles of $\theta$ and $\varphi$
exchanged.
\end{defi}

A vertical wrinkle move is illustrated in Figure~\ref{wrinklemovefig}. The left pair of pictures shows how the rectangular
diagram changes,
\begin{figure}[ht]
\centerline{\includegraphics{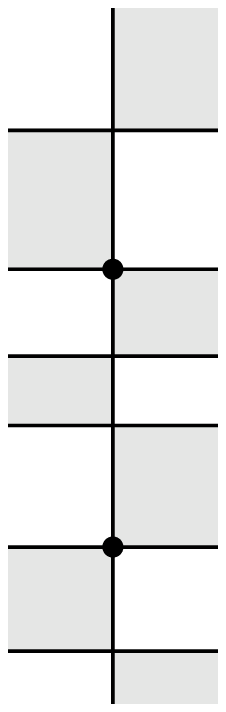}\put(-68,0){$\Pi$}
\put(-37,47){$v_1$}\put(-37,140){$v_2$}%
\put(-45,0){$m_{\theta_0}$}\put(-85,53){$\ell_{\varphi_1}$}\put(-85,133){$\ell_{\varphi_2}$}%
\raisebox{3.5cm}{$\longleftrightarrow$}\includegraphics{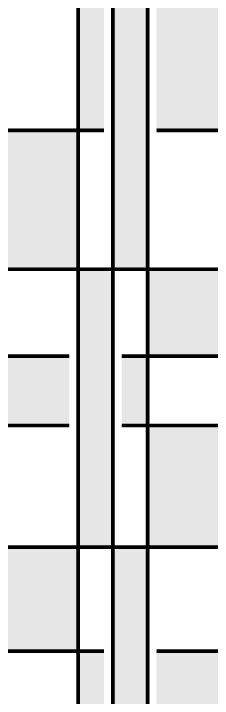}\put(-68,0){$\Pi'$}
\raisebox{1.5cm}{\includegraphics{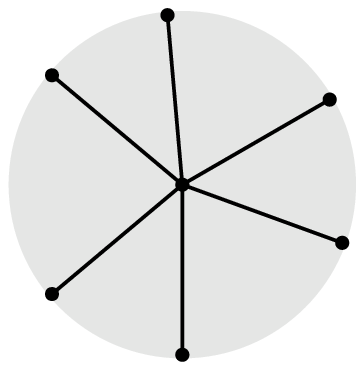}\put(-58,30){$\widehat v_1$}\put(-60,85){$\widehat v_2$}\put(-80,60){$\widehat m_{\theta_0}$}
\raisebox{2cm}{$\longleftrightarrow$}\includegraphics{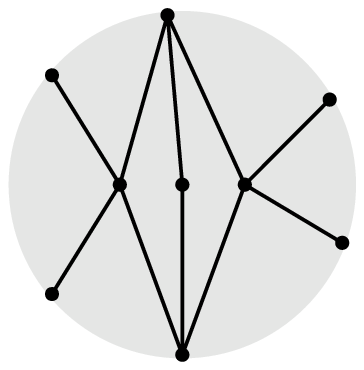}}}
\caption{A vertical wrinkle move}\label{wrinklemovefig}
\end{figure}
and the right pair of pictures shows the change in the corresponding tiling of $\widehat\Pi$.

\begin{defi}\label{half-wrinkle-def}
Let $\Pi$, $v_1$, $v_2$, and~$\psi$ be as in Definition~\ref{wrinkledef} and suppose
additionally that we have $v_1,v_2\in\partial\Pi$.
Let $\Pi'$ be obtained from $\Pi$ as described in Definition~\ref{wrinkledef} with the following one distinction:
\begin{itemize}
\item
if $\Pi$ has no rectangle of the form $[\theta_0;\theta_1]\times[\varphi';\varphi'']$ or
$[\theta_1;\theta_0]\times[\varphi';\varphi'']$ with $[\varphi';\varphi'']\subset[\varphi_1;\varphi_2]$,
we do not add the rectangle $[\psi(\theta_0);\psi(\theta_0+\varepsilon)]\times[\varphi_2;\varphi_1]$;
\item
if $\Pi$ has no rectangle of the form $[\theta_0;\theta_1]\times[\varphi';\varphi'']$ or
$[\theta_1;\theta_0]\times[\varphi';\varphi'']$ with $[\varphi';\varphi'']\subset[\varphi_2;\varphi_1]$,
we do not add the rectangle~$[\psi(\theta_0-\varepsilon);\psi(\theta_0)]\times[\varphi_1;\varphi_2]$.
\end{itemize}
One of these two cases must occur.

Then we say that the passage from $\Pi$ to $\Pi'$ is \emph{a vertical half-wrinkle creation move}, and the inverse operation is \emph{a vertical half-wrinkle reduction move}.

\emph{Horizontal half-wrinkle moves} are defined similarly with the roles of $\theta$ and $\varphi$
exchanged.
\end{defi}

A vertical half-wrinkle move is illustrated in Figure~\ref{halfwrinklemovefig}.
\begin{figure}[ht]
\centerline{
\includegraphics{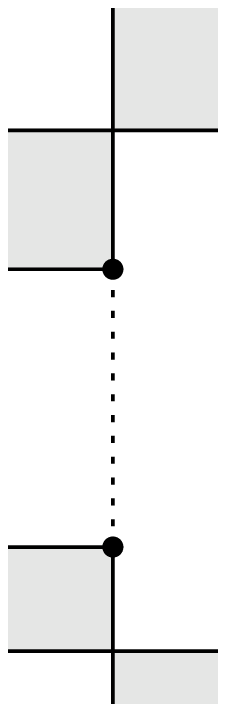}\put(-68,0){$\Pi$}\put(-45,0){$m_{\theta_0}$}\put(-37,60){$v_1$}\put(-37,127){$v_2$}%
\raisebox{3.5cm}{$\longleftrightarrow$}\includegraphics{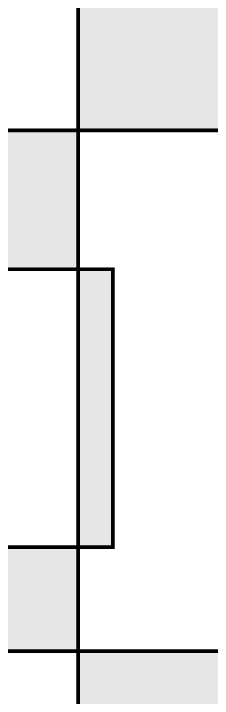}\put(-68,0){$\Pi'$}
\raisebox{1.5cm}{\includegraphics{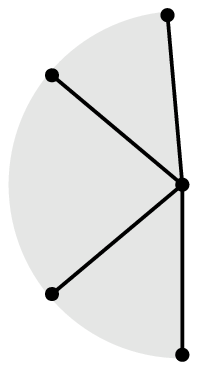}\put(-8,30){$\widehat v_1$}\put(-10,85){$\widehat v_2$}\put(-7,52){$\widehat m_{\theta_0}$}
\hspace{0.2cm}
\raisebox{2cm}{$\longleftrightarrow$}
\includegraphics{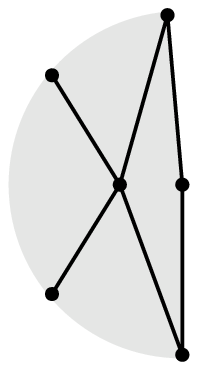}}
}
\caption{A vertical half-wrinkle move}\label{halfwrinklemovefig}
\end{figure}
In the case $v_1,v_2\in\partial\Pi$ the respective wrinkle creation move can be decomposed into two half-wrinkle creation moves, which justifies the term `half-wrinkle'.

\begin{defi}\label{stab-move-def}
Let~$\Pi$, $v_1$, $v_2$, and~$\psi$ be as in Definition~\ref{wrinkledef} except that $v_1\in m_{\theta_0}$ is not a vertex of $\Pi$ and,
moreover, $v_1$ does not belong to any rectangle and any occupied longitude of~$\Pi$. Let $\Pi'$ be obtained from $\Pi$
by exactly the same modification as the one described in Definition~\ref{wrinkledef}, and let~$T$ be
a transformation of rectangular diagrams induced by any of the following symmetries or any decomposition of them (including the identity
transformation):
$$(\theta,\varphi)\mapsto(\varphi,\theta),\quad
(\theta,\varphi)\mapsto(-\theta,\varphi).$$
Then the passage
from~$T(\Pi)$ to~$T(\Pi')$ is called \emph{a stabilization move} and the inverse one
\emph{a destabilization move}.

\end{defi}

An example of a (de)stabilization move is shown in Figure~\ref{boundarystabmovefig}.
\begin{figure}[ht]
\centerline{\includegraphics{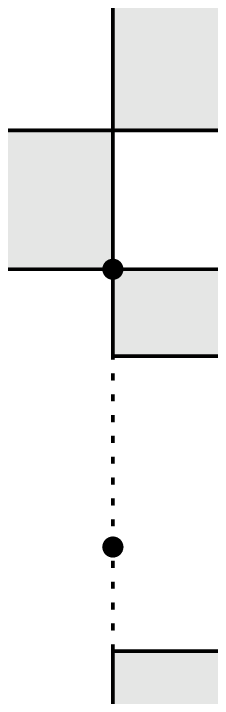}\put(-68,0){$\Pi$}\put(-45,0){$m_{\theta_0}$}\put(-37,47){$v_1$}\put(-37,140){$v_2$}%
\raisebox{3.5cm}{$\longleftrightarrow$}\includegraphics{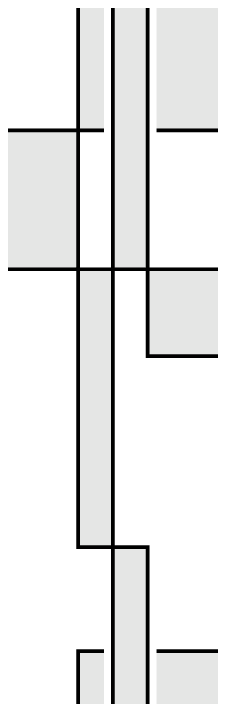}\put(-68,0){$\Pi'$}
\raisebox{1.5cm}{
\includegraphics{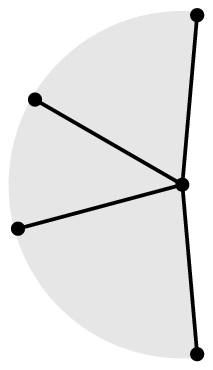}\put(-35,75){$\widehat v_2$}\put(-10,54){$\widehat m_{\theta_0}$}
\hspace{0.1cm}
\raisebox{2cm}{$\longleftrightarrow$}\includegraphics{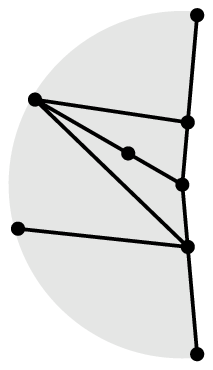}
}}
\caption{A stabilization/destabilization moves}\label{boundarystabmovefig}
\end{figure}

Note also that if~$\Pi\mapsto\Pi'$ is a 
stabilization of a rectangular diagram of a surface,
then~$\partial\Pi\mapsto\partial\Pi'$ is a stabilization of a rectangular diagram of a link
(in the generalized sense of~\cite{dyn06}).

\begin{defi}\label{exchange-def}
Let $\Pi$ be a rectangular diagram of a surface, and let $\theta_1,\theta_2,\theta_3,\varphi_1,\varphi_2\in\mathbb S^1$ be such
that:
\begin{enumerate}
\item
we have $\theta_2\in(\theta_1;\theta_3)$;
\item
the rectangles $r_1=[\theta_1;\theta_2]\times[\varphi_1;\varphi_2]$ and $r_2=[\theta_2;\theta_3]\times[\varphi_2;\varphi_1]$
contain no vertices of $\Pi$;
\item
the vertices of $r_1$ and $r_2$ are disjoint from the rectangles of~$\Pi$.
\end{enumerate}
Let $f:\mathbb S^1\rightarrow\mathbb S^1$ be a map that is identical on $[\theta_3;\theta_1]$, and exchanges
the intervals $(\theta_1;\theta_2)$ and $(\theta_2;\theta_3)$:
$$f(\theta)=\left\{\begin{aligned}\theta-\theta_2+\theta_3,&\text{ if }\theta\in(\theta_1,\theta_2],\\
\theta-\theta_2+\theta_1,&\text{ if }\theta\in(\theta_2,\theta_3).
\end{aligned}\right.$$
Choose a self-homeomorphism~$\psi$ of~$\mathbb S^1$ identical on~$[\theta_3;\theta_1]$, and let
$$\Pi'=\bigl\{[\\\psi(f(\theta'));\psi(f(\theta''))]\times[\varphi';\varphi'']:[\theta';\theta'']\times[\varphi';\varphi'']\in\Pi\bigr\}.$$
Then we say that the passage from $\Pi$ to $\Pi'$, or the other way, is \emph{a vertical exchange move}.

\emph{A horizontal exchange move} is defined similarly with the roles of $\theta$ and $\varphi$ exchanged.
\end{defi}

An example of a vertical exchange move is shown in Figure~\ref{exchangemovefig}.
\begin{figure}[ht]
\includegraphics{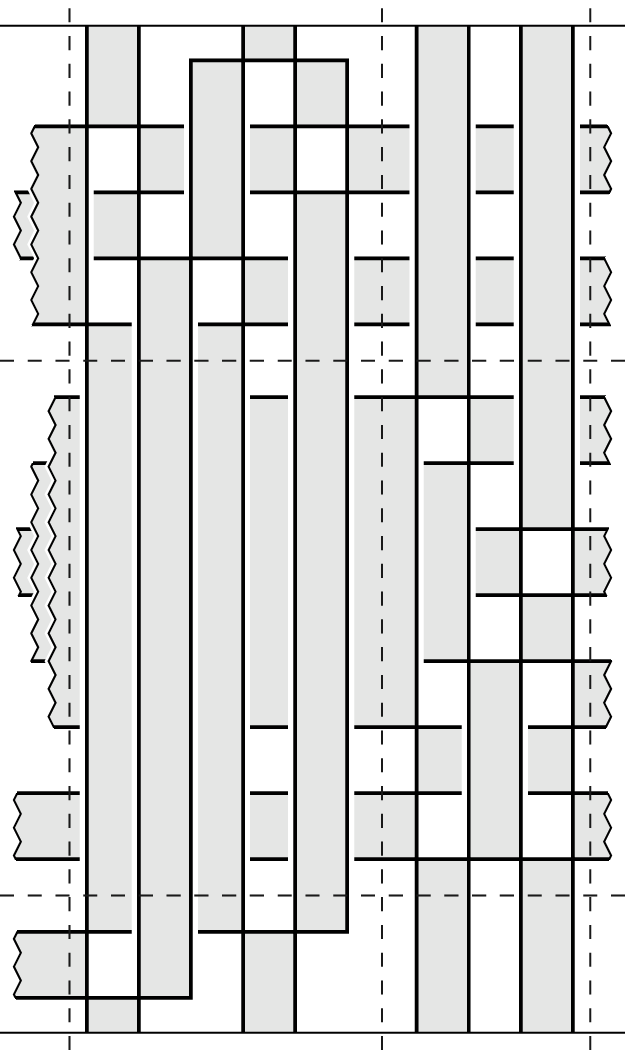}\put(-164,-10){$\theta_1$}\put(-73,-10){$\theta_2$}\put(-13,-10){$\theta_3$}%
\put(-195,42){$\varphi_1$}\put(-195,197){$\varphi_2$}
\hskip1cm\raisebox{145pt}{$\longleftrightarrow$}\hskip1cm
\includegraphics{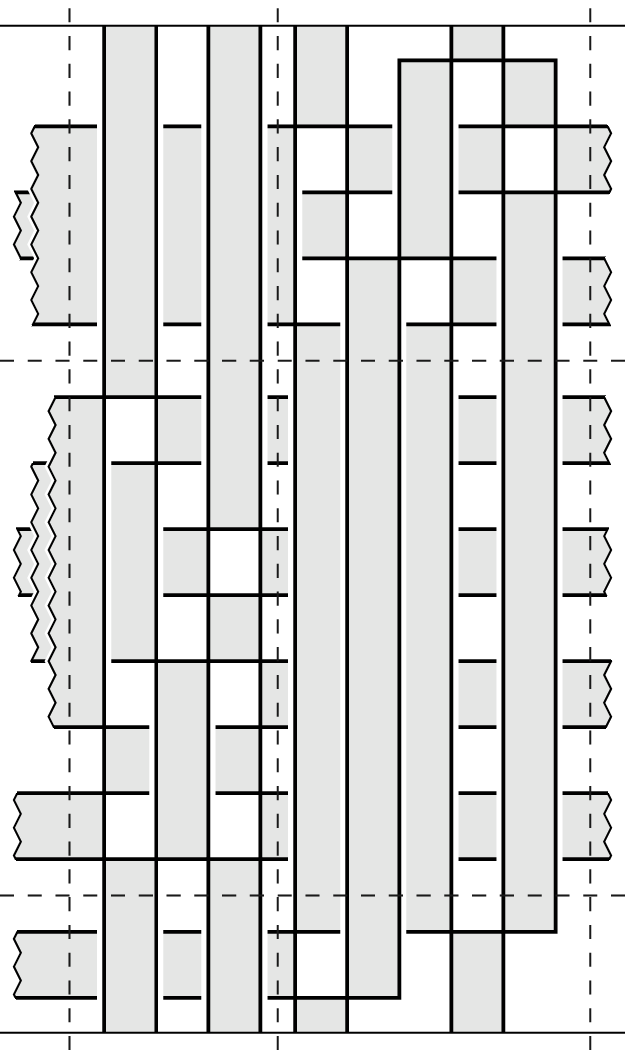}\put(-164,-10){$\theta_1$}\put(-130,-10){$\psi(\theta_1-\theta_2+\theta_3)$}\put(-13,-10){$\theta_3$}%
\put(-195,42){$\varphi_1$}\put(-195,197){$\varphi_2$}
\caption{An exchange move}\label{exchangemovefig}
\end{figure}

\begin{defi}\label{flype-def}
Let $\Pi$ be a rectangular diagram of a surface, and let $\theta_1,\theta_2,\theta_3,\varphi_1,\varphi_2,\varphi_3\in\mathbb S^1$
be such that:
\begin{enumerate}
\item
$\theta_2\in(\theta_1;\theta_3)$, $\varphi_2\in(\varphi_1;\varphi_3)$;
\item
points $v_1,v_2,v_3,v_4,v_5\in\mathbb T^2$ having coordinates
$(\theta_1,\varphi_3)$, $(\theta_2,\varphi_3)$, $(\theta_3,\varphi_3)$, $(\theta_3,\varphi_2)$,
$(\theta_3,\varphi_1)$, respectively, are vertices of $\Pi$, and, moreover, $v_1,v_3,v_5$ are $\diagup$-vertices,
and $v_2,v_4$ are $\diagdown$-vertices;
\item
$v_2$, $v_3$, and~$v_4$ do not belong to $\partial\Pi$;
\item
there are no more vertices of $\Pi$ in $[\theta_1;\theta_3]\times[\varphi_1;\varphi_3]$.
\end{enumerate}

These assumptions imply that $\Pi$ contains, among others, four rectangles of the following form:
$$r_1=[\theta_1;\theta_2]\times[\varphi';\varphi_3],\quad
r_2=[\theta_2;\theta_3]\times[\varphi_3;\varphi''],\quad
r_3=[\theta_3;\theta'']\times[\varphi_2;\varphi_3],\quad
r_4=[\theta';\theta_3]\times[\varphi_1;\varphi_2]$$
with some $\theta',\theta''\in(\theta_3;\theta_1)$, $\varphi',\varphi''\in(\varphi_3;\varphi_1)$.

Let $\Pi'$ be obtained from $\Pi$ by replacing these four rectangles with the following ones:
$$r_1'=[\theta_1;\theta_2]\times[\varphi';\varphi_1],\quad
r_2'=[\theta_2;\theta_3]\times[\varphi_1;\varphi''],\quad
r_3'=[\theta_1;\theta'']\times[\varphi_2;\varphi_3],\quad
r_4'=[\theta';\theta_1]\times[\varphi_1;\varphi_2];
$$
see Figure~\ref{flypefig}.
\begin{figure}[ht]
\includegraphics{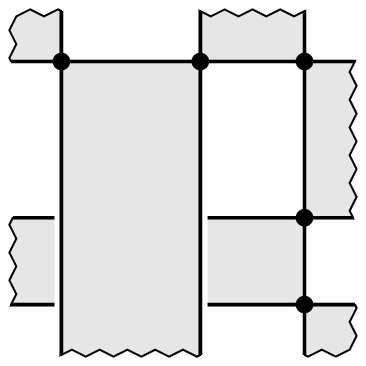}\put(-88,93){$v_1$}\put(-62,93){$v_2$}\put(-18,93){$v_3$}\put(-18,38){$v_4$}\put(-18,23){$v_5$}%
\put(-73,64){$r_1$}\put(-38,95){$r_2$}\put(-17,64){$r_3$}\put(-38,31){$r_4$}\put(-95,-2){$m_{\theta_1}$}\put(-55,-2){$m_{\theta_2}$}%
\put(-25,-2){$m_{\theta_3}$}\put(-120,18){$\ell_{\varphi_1}$}\put(-120,43){$\ell_{\varphi_2}$}\put(-120,88){$\ell_{\varphi_3}$}
\hskip1cm\raisebox{50pt}{$\longleftrightarrow$}\hskip1cm
\includegraphics{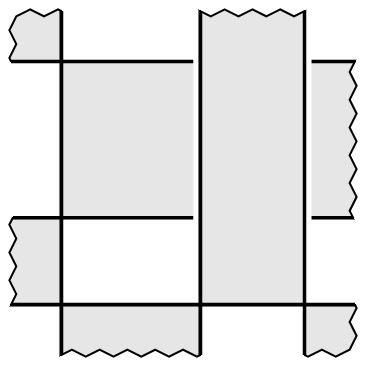}%
\put(-73,10){$r_1'$}\put(-38,31){$r_2'$}\put(-73,64){$r_3'$}\put(-101,31){$r_4'$}
\\[5pt]
\includegraphics{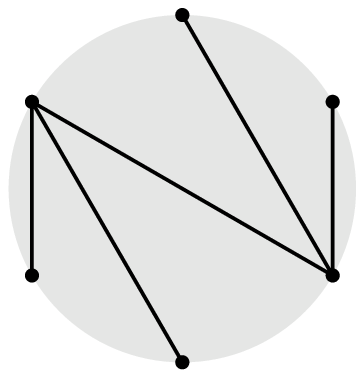}\put(-114,57){$\widehat v_1$}\put(-88,37){$\widehat v_2$}\put(-63,50){$\widehat v_3$}%
\put(-50,70){$\widehat v_4$}\put(-13,55){$\widehat v_5$}\put(-90,21){$\widehat r_1$}\put(-46,21){$\widehat r_2$}%
\put(-80,92){$\widehat r_3$}\put(-36,92){$\widehat r_4$}\put(-65,117){$\widehat\ell_{\varphi_2}$}%
\put(-13,87){$\widehat\ell_{\varphi_1}$}\put(-120,87){$\widehat\ell_{\varphi_3}$}\put(-13,30){$\widehat m_{\theta_3}$}%
\put(-65,-2){$\widehat m_{\theta_2}$}\put(-120,30){$\widehat m_{\theta_1}$}%
\hskip0.5cm\raisebox{58pt}{$\longleftrightarrow$}\hskip0.5cm\includegraphics{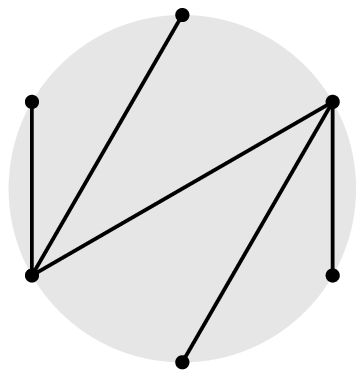}\put(-85,21){$\widehat r_1'$}\put(-41,21){$\widehat r_2'$}%
\put(-87,92){$\widehat r_3'$}\put(-43,92){$\widehat r_4'$}\\
\caption{A flype}\label{flypefig}
\end{figure}
Then we say that the passage from $\Pi$ to $\Pi'$, or the other way, is \emph{a flype}.
Note that the other rectangles of $\Pi$ and $\Pi'$, not
shown in Figure~\ref{flypefig}, are allowed to pass through
$[\theta_1;\theta_3]\times[\varphi_1;\varphi_3]$, the region where the modification occurs.

If~$\Pi\mapsto\Pi'$ is a flype and~$T$ is the transformation of rectangular diagrams induced by
the map~$(\theta,\varphi)\mapsto(-\theta,\varphi)$, then~$T(\Pi)\mapsto T(\Pi')$ is also called a flype.
\end{defi}

To complete Definitions~\ref{wrinkledef}--\ref{flype-def} we need to define the respective morphisms~$\widehat\Pi\rightarrow\widehat\Pi'$
in each case. In~\cite{distinguishing} we described an isotopy from~$\widehat\Pi$ to~$\widehat\Pi'$, which can
be continued to the whole of~$\mathbb S^3$ thus producing the required morphism. Here we outline the
general principles of the construction.

First, let~$\Pi\mapsto\Pi'$ be any basic move other than
an exchange move. Then there are open discs or open half-discs~$d\subset\widehat\Pi$ and~$d'\subset\widehat\Pi'$
such that
$$\overline d=\bigcup_{r\in\Pi\setminus\Pi'}\widehat r,\quad\overline d'=\bigcup_{r\in\Pi'\setminus\Pi}\widehat r$$
(by an open \emph{half-disc} in a surface~$F$ we mean an open subset~$d\subset F$ such that~$d$ is homeomorphic
to~$\{(x,y)\in\mathbb R^2:x^2+y^2<1,y\geqslant0\}$, and~$\overline d\setminus d$ is not a single point).
There is also an open $3$-ball~$B\subset\mathbb S^3$ such that~$B\cap\widehat\Pi=d$ and~$B\cap\widehat\Pi'=d'$.
The morphism~$\Pi\mapsto\Pi'$ is represented by any
homeomorphism~$(\mathbb S^3,\widehat\Pi)\rightarrow(\mathbb S^3,\widehat\Pi')$
which is identical outside~$B$.

Now let~$\Pi\mapsto\Pi'$ be a vertical exchange move. We use the notation from Definition~\ref{exchange-def}.
Pick an~$\varepsilon>0$ so small that no occupied meridian of~$\Pi$ is contained in~$\bigl((\theta_1;\theta_1+\varepsilon]\cup
[\theta_2-\varepsilon;\theta_2+\varepsilon]\cup[\theta_3-\varepsilon;\theta_3)\bigr)\times\mathbb S^1$.
Choose self-homeomorphisms~$\psi_1,\psi_2$ of~$\mathbb S^1$ such that:
\begin{enumerate}
\item
$\psi_1(\theta)=\psi_2(\theta)=\theta$ if~$\theta\in[\theta_3;\theta_1]$;
\item
$\psi_1(\theta)=\psi(f(\theta))$ if~$\theta\in[\theta_1+\varepsilon;\theta_2-\varepsilon]$;
\item
$\psi_2(\theta)=\psi(f(\theta))$ if~$\theta\in[\theta_2+\varepsilon;\theta_3-\varepsilon]$.
\end{enumerate}

Here and in the sequel we use \emph{the open book decomposition} of~$\mathbb S^3$
with the \emph{binding}~$\mathbb S^1_{\tau=0}$ and \emph{pages} of the form
$$\mathscr P_{\theta_0}=\{\theta_0\}*\mathbb S^1_{\tau=0},\quad\theta_0\in\mathbb S^1.$$
In other words, for any~$\theta_0\in\mathbb S^1$, the page~$\mathscr P_{\theta_0}$ is defined
by the equation~$\theta=\theta_0$.

The union~$\widehat r_1\cup\widehat r_2$ is a $2$-disc. There is a small open neighborhood~$U$ of this disc such that~$U$
is disjoint from~$\widehat\Pi$, $\mathbb S^3\setminus U$ is homeomorphic to a closed $3$-ball, and, for any~$\theta_0\in[\theta_1;\theta_3]$,
the intersection of~$\overline U$ with the page~$\mathscr P_{\theta_0}$ is a regular
neighborhood of~$(\widehat r_1\cup\widehat r_2)\cap\mathscr P_{\theta_0}$ in~$\mathscr P_{\theta_0}$.

For any~$\theta\in[\theta_1;\theta_3]$, the complement~$\mathscr P_\theta\setminus U$ is a union of two closed
$2$-discs, which we denote by~$\mathscr P_\theta'$ and~$\mathscr P_\theta''$ using the following rule:
$$\mathscr P_\theta'\cap\mathbb S^1_{\tau=0}\subset[\varphi_2;\varphi_1],\quad
\mathscr P_\theta''\cap\mathbb S^1_{\tau=0}\subset[\varphi_1;\varphi_2].$$

There is then an embedding~$\sigma:\mathbb S^3\setminus U\rightarrow\mathbb S^3$
such that:
\begin{enumerate}
\item
$\sigma(\widehat\Pi)=\widehat\Pi'$;
\item
$\sigma$ is identical on~$\mathbb S^1_{\tau=0}\setminus U$;
\item
if~$\theta\in[\theta_3;\theta_1]$, then~$\sigma(\mathscr P_\theta\setminus U)\subset\mathscr P_\theta$;
\item
if~$\theta\in[\theta_1;\theta_3]$, then~$\sigma(\mathscr P_\theta')\subset\mathscr P_{\psi_1(\theta)}$
and~$\sigma(\mathscr P_\theta'')\subset\mathscr P_{\psi_2(\theta)}$.
\end{enumerate}
This embedding can be extended to a homeomorphism~$\mathbb S^3\rightarrow\mathbb S^3$,
which represents the morphism~$\widehat\Pi\rightarrow\widehat\Pi'$ assigned to the exchange move~$\Pi\mapsto\Pi'$.

\begin{defi}\label{bound-fixed-def}
Let~$\Pi$ be a rectangular diagram of a surface, and let~$X$ be a subset of~$\partial\Pi$.
A basic move~$\Pi\mapsto\Pi'$ is said to be \emph{fixed} on~$X$ if~$X\subset\partial\Pi'$
and the types of any~$v\in X$ as a vertex of~$\Pi$ and as a vertex of~$\Pi'$ (which are~`$\diagup$' or~`$\diagdown$') coincide.
\end{defi}

\begin{theo}\label{main-theo}
Let~$\Pi$ and~$\Pi'$ be rectangular diagrams of surfaces, and let~$R$ be a rectangular
diagram of a link such that~$L=\widehat R$
is a \emph(possibly empty\emph) sublink of~$\partial\widehat\Pi$.
Let also~$\phi$ be a self-homeomorphism of~$\mathbb S^3$
that takes~$\widehat\Pi$ to~$\widehat\Pi'$.
Then the following two conditions are equivalent.
\begin{enumerate}
\item[(i)]
There exists an isotopy from~$\mathrm{id}|_{\mathbb S^3}$ to~$\phi$
fixed on~$L$ and such that the
tangent plane to the surface at any point~$p\in L$ is preserved during the isotopy.
\item[(ii)]
There exists a sequence of basic moves fixed on~$R$
$$\Pi=\Pi_0\mapsto\Pi_1\mapsto\ldots\mapsto\Pi_N=\Pi'$$
such that the composition of the corresponding morphisms is~$[\phi]$.
\end{enumerate}
\end{theo}

The proof of Theorem~\ref{main-theo} will be given in full generality in Section~\ref{boundary-sec}.
The hardest part of the proof
is the subject of Section~\ref{fixed-b-sec} (Proposition~\ref{fixed-b-prop}),
where we restrict ourselves to the case~$R=\partial\Pi$.

\section{Bubble moves}\label{bubble-sec}

\begin{defi}\label{bubble-def}
Let~$r=[\theta_1;\theta_2]\times[\varphi_1;\varphi_2]$ be a rectangle of a rectangular diagram
of a surface~$\Pi$, and let~$\theta_3,\theta_4\in(\theta_1;\theta_2)$ be
such that no occupied meridian of~$\Pi$ lies in the domain~$[\theta_3;\theta_4]\times\mathbb S^1$,
and the rectangle~$[\theta_3;\theta_4]\times[\varphi_1;\varphi_2]$ is not contained
in the interior of any rectangle in~$\Pi$.
Then the passage from~$\Pi$ to~$\Pi'=\Pi\cup\{r_1,r_2,r_3\}\setminus\{r\}$ (assigned with the morphism~$\widehat\Pi\rightarrow
\widehat\Pi'$ described below), where
$$r_1=[\theta_1;\theta_3]\times[\varphi_1;\varphi_2],\quad
r_2=[\theta_3;\theta_4]\times[\varphi_2;\varphi_1],\quad
r_3=[\theta_4;\theta_2]\times[\varphi_1;\varphi_2],$$
is called \emph{a vertical bubble creation move}, and the inverse passage~$\Pi'\mapsto\Pi$
\emph{a vertical bubble reduction move}.

One can see that the transition from~$\widehat\Pi$ to~$\widehat\Pi'$ is a
replacement of the interior of the two-disc~$\widehat r$
with the interior of the two-disc~$d=\widehat r_1\cup\widehat r_2\cup\widehat r_3$,
and we have~$\widehat r\cap d=\partial\widehat r=\partial d$. 
The interior of both discs lies in the domain~$\theta\in(\theta_1;\theta_2)$.
For any~$\theta\in(\theta_1;\theta_2)$, the intersection of each of the discs
with~$\mathscr P_\theta$ is an arc, and the two arcs~$\widehat r\cap\mathscr P_\theta$
and~$d\cap\mathscr P_\theta$ enclose a $2$-disc whose interior is disjoint from~$\widehat\Pi$.
This implies that there exists a self-homeomorphism~$\phi$ of~$\mathbb S^3$
preserving each page~$\mathscr P_\theta$ and taking~$\widehat\Pi$ to~$\widehat\Pi'$.
The morphism~$[\phi]$ is the one assigned to the bubble creation move~$\Pi\mapsto\Pi'$.

\emph{Horizontal bubble creation} and \emph{reduction moves} are defined similarly
with the roles of~$\theta$ and~$\varphi$ exchanged.
\end{defi}

Bubble moves are illustrated in Figure~\ref{bubble-fig}
\begin{figure}[ht]
\includegraphics{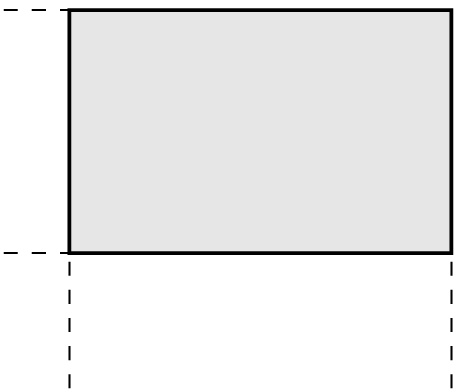}\put(-163,38){$\varphi_1$}\put(-163,108){$\varphi_2$}
\put(-133,-10){$\theta_1$}\put(-23,-10){$\theta_2$}\put(-77,73){$r$}
\hskip1cm\raisebox{70pt}{$\longleftrightarrow$}\hskip1cm
\includegraphics{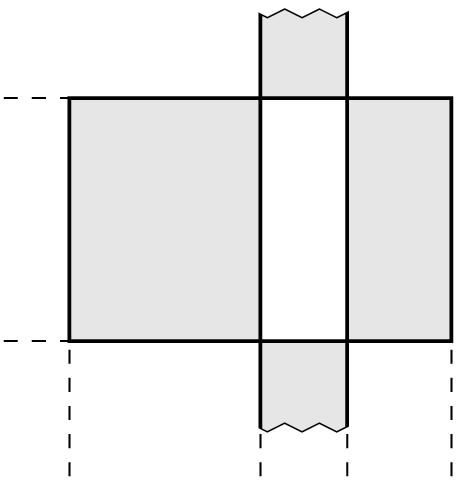}\put(-163,38){$\varphi_1$}\put(-163,108){$\varphi_2$}
\put(-133,-10){$\theta_1$}\put(-23,-10){$\theta_2$}
\put(-78,-10){$\theta_3$}\put(-53,-10){$\theta_4$}
\put(-105,73){$r_1$}\put(-39,73){$r_3$}\put(-67,25){$r_2$}\put(-67,120){$r_2$}

\vskip5mm
\includegraphics{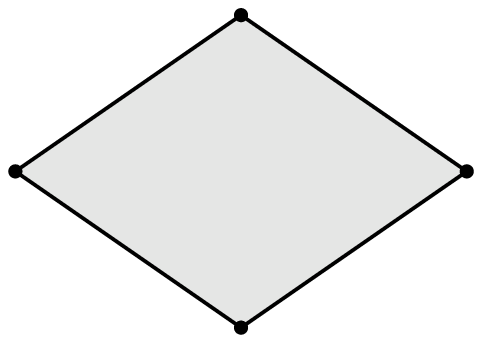}
\put(-155,48){$\widehat m_{\theta_1}$}\put(0,48){$\widehat m_{\theta_2}$}
\put(-75,-9){$\widehat\ell_{\varphi_1}$}\put(-75,102){$\widehat\ell_{\varphi_2}$}
\put(-73,48){$\widehat r$}
\hskip1cm\raisebox{48pt}{$\longleftrightarrow$}\hskip1cm
\includegraphics{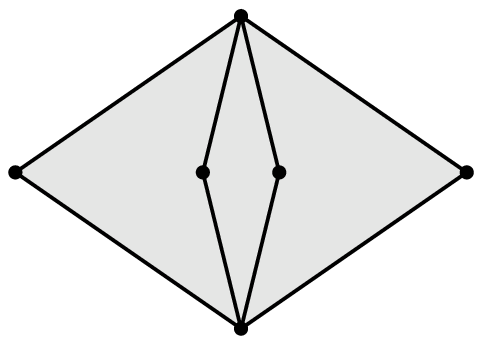}
\put(-155,48){$\widehat m_{\theta_1}$}\put(0,48){$\widehat m_{\theta_2}$}
\put(-75,-9){$\widehat\ell_{\varphi_1}$}\put(-75,102){$\widehat\ell_{\varphi_2}$}
\put(-74,48){$\widehat r_2$}\put(-115,48){$\widehat r_1$}\put(-33,48){$\widehat r_3$}
\put(-55,48){$\widehat m_{\theta_4}$}\put(-101,48){$\widehat m_{\theta_3}$}
\caption{Bubble moves}\label{bubble-fig}
\end{figure}

\begin{lemm}
Any bubble move of marked rectangular diagrams of surfaces can be decomposed into basic moves.
\end{lemm}

\begin{proof}
Let~$\Pi,\Pi'$ 
be as in Definition~\ref{bubble-def}. Pick~$(\theta_5;\theta_6)\subset(\theta_1;\theta_3)$
so close to~$\theta_1$ that no occupied meridian of~$\Pi$ lies in the domain~$(\theta_1;\theta_6]$.
Define~$\Pi''$ to be~$\Pi\cup\{r_1',r_2',r_3'\}$, where
$$r_1'=[\theta_1;\theta_5]\times[\varphi_1;\varphi_2],\quad
r_2'=[\theta_5;\theta_6]\times[\varphi_2;\varphi_1],\quad
r_3'=[\theta_6;\theta_2]\times[\varphi_1;\varphi_2].$$
Then~$\Pi\mapsto\Pi''$ is a wrinkle creation move, and~$\Pi''\mapsto\Pi'$ is an exchange move.
It is elementary to verify that the composition of the respective morphisms yields the morphism
assigned to the bubble creation move~$\Pi\mapsto\Pi'$.
\end{proof}

\section{Movie diagrams}

\begin{defi}
By \emph{a non-crossing chord diagram} we mean a finite set
$$\bigl\{\{\varphi_1',\varphi_1''\},\{\varphi_2',\varphi_2''\},\ldots,\{\varphi_m',\varphi_m''\}\bigr\}$$
of unordered pairs (which are referred as \emph{chords}) of points of~$\mathbb S^1$ such that, for any~$i,j=1,\ldots,m$, $i\ne j$, we have
either~$\{\varphi_i',\varphi_i''\}\subset(\varphi_j';\varphi_j'')$ or~$\{\varphi_i',\varphi_i''\}\subset(\varphi_j'';\varphi_j')$.
The set of all non-crossing chord diagrams is denoted by~$\mathscr C$ and endowed with a discrete topology.
\end{defi}

\begin{defi}
Let~$C_1$ and~$C_2$ be two non-crossing chord diagrams. We say that the passage from~$C_1$ to~$C_2$
is \emph{an admissible event} if, for some~$\varphi_1,\varphi_2,\ldots,\varphi_k\in\mathbb S^1$, $k\geqslant 2$,
following on~$\mathbb S^1$ in the same circular order as listed, one of the diagrams (whichever of the two) can be obtained from the other
in one of the following ways:
\begin{enumerate}
\item
replacing~$\{\varphi_1,\varphi_2\},\{\varphi_3,\varphi_4\},\ldots,\{\varphi_{k-1},\varphi_k\}$
with~$\{\varphi_2,\varphi_3\},\{\varphi_4,\varphi_5\},\ldots,\{\varphi_{k-2},\varphi_{k-1}\}$ (provided~$k$ is even);
\item
replacing~$\{\varphi_1,\varphi_2\},\{\varphi_3,\varphi_4\},\ldots,\{\varphi_{k-2},\varphi_{k-1}\}$
with~$\{\varphi_2,\varphi_3\},\{\varphi_4,\varphi_5\},\ldots,\{\varphi_{k-1},\varphi_k\}$ (provided~$k$ is odd);
\item
replacing~$\{\varphi_1,\varphi_2\},\{\varphi_3,\varphi_4\},\ldots,\{\varphi_{k-1},\varphi_k\}$
with~$\{\varphi_2,\varphi_3\},\{\varphi_4,\varphi_5\},\ldots,\{\varphi_{k-2},\varphi_{k-1}\},\{\varphi_k,\varphi_1\}$ (provided~$k\geqslant4$ is even).
\end{enumerate}
The points~$\varphi_1,\ldots,\varphi_k\in\mathbb S^1$ are then said to be \emph{involved} in the event.

To specify an admissible event we indicate which chords are removed and which are added,
separating the two lists by a `$\leadsto$' sign. The two
extreme cases, when no chord is removed and just one added, or no one is added and just one removed, are allowed.
For instance, the following are valid specifications
of admissible events (provided that~$\varphi_1,\varphi_2,\varphi_3,\varphi_4$ follow
on~$\mathbb S^1$ in the same circular order as listed):
$$\{\varphi_1,\varphi_2\}\leadsto\varnothing,\quad
\varnothing\leadsto\{\varphi_1,\varphi_2\},\quad
\{\varphi_1,\varphi_2\}\leadsto\{\varphi_2,\varphi_3\},\quad
\{\varphi_1,\varphi_2\},\{\varphi_3,\varphi_4\}\leadsto\{\varphi_2,\varphi_3\},\{\varphi_4,\varphi_1\}.$$

\end{defi}

\begin{defi}
\emph{A movie diagram} is a left continuous map~$\Psi:\mathbb S^1\rightarrow\mathscr C$ with finitely many 
discontinuities, such that,
for any discontinuity point~$\theta_0\in\mathbb S^1$, the passage from~$C_0=\Psi(\theta_0)$ to~$C_1=\lim_{\theta\rightarrow\theta_0+0}\Psi(\theta)$
is an admissible event.
\end{defi}

When~$\Psi$ is a movie diagram, and~$\theta_0\in\mathbb S^1$,
we denote $\lim_{\theta\rightarrow\theta_0+0}\Psi(\theta)$ by~$\Psi(\theta_0+0)$ for brevity.
We also denote by~$\Phi(\Psi)$ the set of all~$\varphi_0\in\mathbb S^1$ such that for some~$\theta_0,\varphi_1\in\mathbb S^1$
the non-crossing chord diagram~$\Psi(\theta_0)$ contains the chord~$\{\varphi_0,\varphi_1\}$.

\begin{defi}
Let~$\Psi$ be a movie diagram.
We say that a surface~$F\subset\mathbb S^3$ \emph{represents~$\Psi$} if the following conditions hold:
\begin{enumerate}
\item
$F\cap\mathbb S^1_{\tau=0}=\Phi(\Psi)$;
\item
for any~$\theta\in\mathbb S^1$, each connected component of the intersection~$\mathscr P_\theta\cap F$
is either an arc or a star-like graph whose edges join a few points in~$\mathbb S^1_{\tau=0}$
with a vertex located in the interior of~$\mathscr P_\theta$;
\item
for any~$\theta\in\mathbb S^1$, the points~$\varphi',\varphi''\in\mathbb S^1_{\tau=0}$
are in the same connected component of~$\mathscr P_\theta\cap F$ if and only if one
of the following occurs:
\begin{itemize}
\item
$\{\varphi',\varphi''\}\in\Psi(\theta)$;
\item
$\Psi$ has an event at~$\theta$ and~$\varphi',\varphi''$ are involved in it;
\end{itemize}
\item
there are only finitely many points at which~$F$ is tangent to some page~$\mathscr P_\theta$.
\end{enumerate}
\end{defi}

\begin{defi}
For every rectangular diagram of a surface~$\widehat\Pi$ we define \emph{the associated movie diagram}~$\Psi_\Pi$
by requesting that, whenever~$m_{\theta_0}$ is not an occupied meridian of~$\Pi$, we
have~$\{\varphi_1,\varphi_2\}\in\Psi_\Pi(\theta_0)$ if and only if the following holds:
$$\exists\ \theta_1,\theta_2\in\mathbb S^1\text{ such that }
\theta_0\in(\theta_1;\theta_2)\text{ and either }
[\theta_1;\theta_2]\times[\varphi_1;\varphi_2]\in\Pi\text{ or }[\theta_1;\theta_2]\times[\varphi_2;\varphi_1]\in\Pi.$$
\end{defi}

Proposition~\ref{another-char-prop} below, which follows easily from definitions, gives another characterization of~$\Psi_\Pi$.

\begin{prop}\label{another-char-prop}
Let~$\Pi$ be a rectangular diagram of a surface, and let~$\Psi$ be a movie diagram.
Then the surface~$\widehat\Pi$ represents~$\Psi$ if and only if~$\Psi=\Psi_\Pi$.
\end{prop}

We omit the easy proof. For the reason which is clear from this statement, we say that a rectangular
diagram of a surface~$\Pi$ represents a movie diagram~$\Psi$ if~$\Psi=\Psi_\Pi$ or, equivalently,
if~$\widehat\Pi$ represents~$\Psi$.

\begin{exam}
Shown in Figure~\ref{movieexam-fig} are a rectangular diagram of a surface~$\Pi$
and the intersections of the pages~$\mathscr P_\theta$ with the surface~$\widehat\Pi$.
The topology of these intersections is the information encoded in the movie
diagram~$\Psi_\Pi$.
\begin{figure}[p]
\includegraphics[scale=.4]{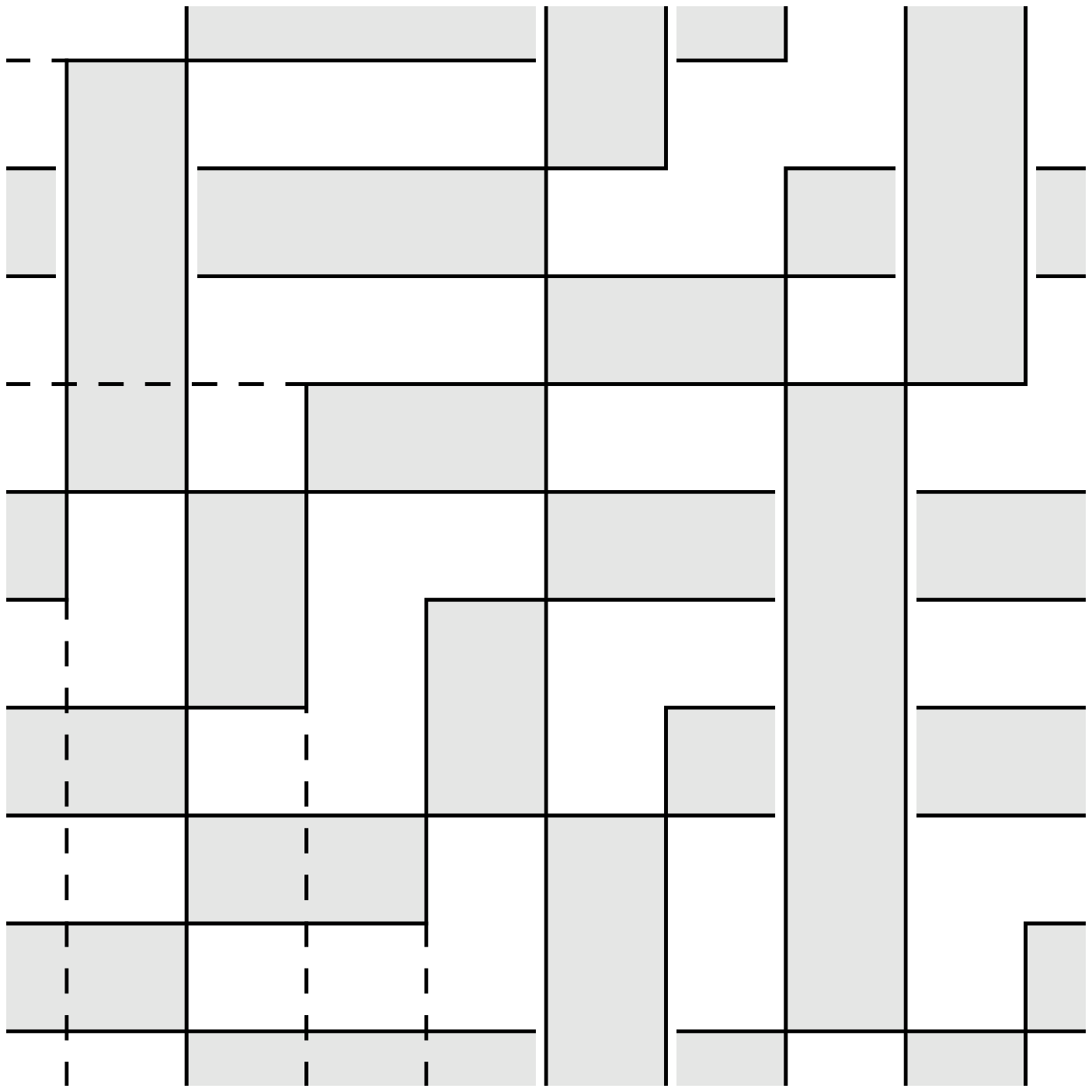}\put(-159,-8){$\theta_1$}\put(-141.125,-8){$\theta_2$}
\put(-123.25,-8){$\theta_3$}\put(-105.375,-8){$\theta_4$}\put(-87.5,-8){$\theta_5$}\put(-69.625,-8){$\theta_6$}
\put(-51.75,-8){$\theta_7$}\put(-16,-8){$\theta_8$}
\put(-180,10){$\varphi_1$}\put(-180,26.1){$\varphi_2$}\put(-180,42.2){$\varphi_3$}\put(-180,58.3){$\varphi_4$}
\put(-180,74.4){$\varphi_5$}\put(-180,90.5){$\varphi_6$}\put(-180,106.6){$\varphi_7$}
\put(-180,122.7){$\varphi_8$}\put(-180,138.8){$\varphi_9$}\put(-180,154.9){$\varphi_{10}$}

\bigskip
\includegraphics[scale=0.4]{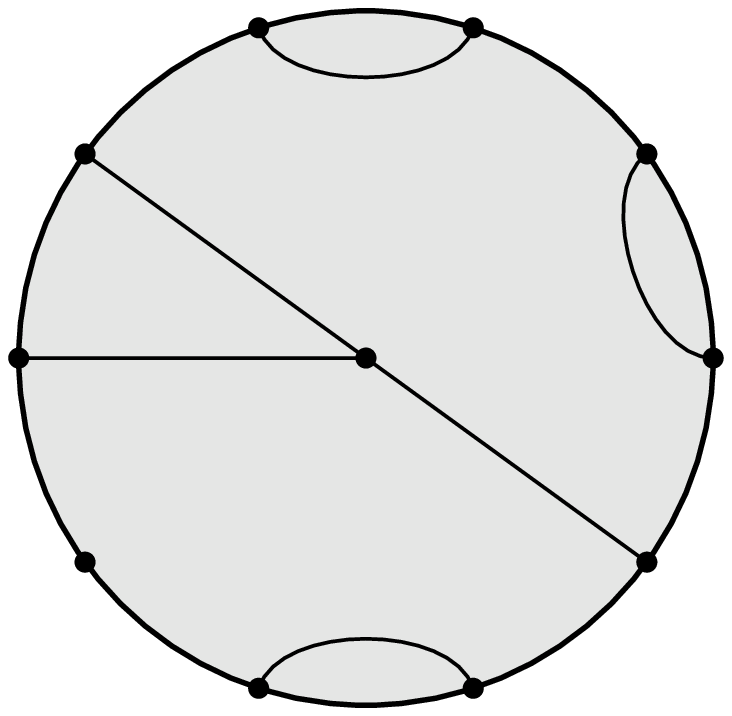}\put(0,43){$\varphi_1$}\put(-8,69){$\varphi_2$}
\put(-34,88){$\varphi_3$}\put(-63,88){$\varphi_4$}\put(-90,69){$\varphi_5$}
\put(-98,43){$\varphi_6$}\put(-90,17){$\varphi_7$}\put(-63,-2){$\varphi_8$}
\put(-34,-2){$\varphi_9$}\put(-8,17){$\varphi_{10}$}
\put(-50,50){$\theta=\theta_1$}
\hskip1cm
\includegraphics[scale=0.4]{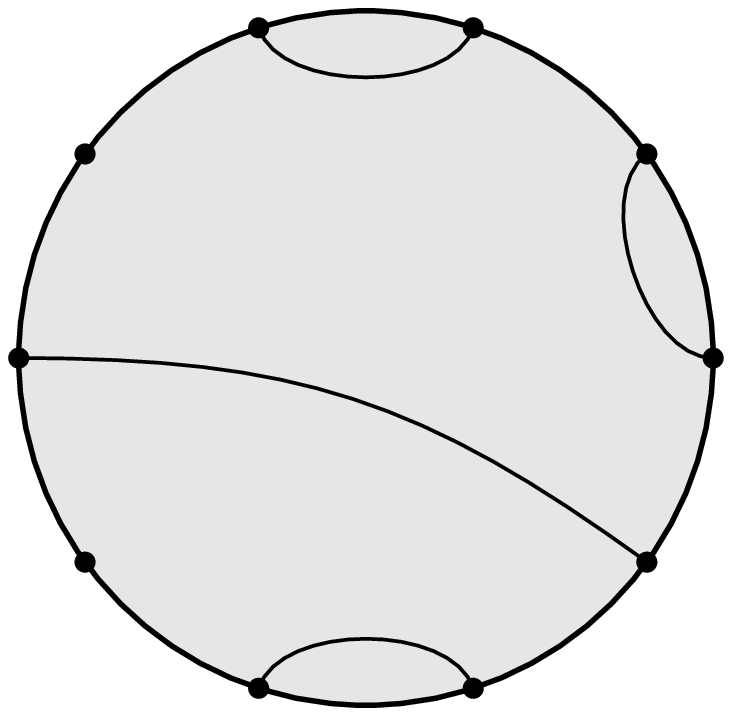}\put(0,43){$\varphi_1$}\put(-8,69){$\varphi_2$}
\put(-34,88){$\varphi_3$}\put(-63,88){$\varphi_4$}\put(-90,69){$\varphi_5$}
\put(-98,43){$\varphi_6$}\put(-90,17){$\varphi_7$}\put(-63,-2){$\varphi_8$}
\put(-34,-2){$\varphi_9$}\put(-8,17){$\varphi_{10}$}
\put(-70,50){$\theta\in(\theta_1;\theta_2)$}
\hskip1cm
\includegraphics[scale=0.4]{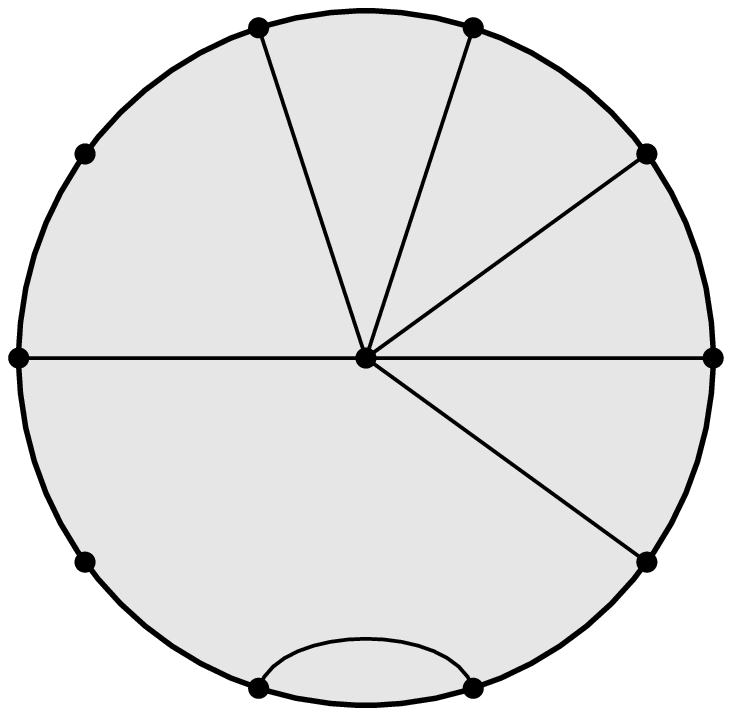}\put(0,43){$\varphi_1$}\put(-8,69){$\varphi_2$}
\put(-34,88){$\varphi_3$}\put(-63,88){$\varphi_4$}\put(-90,69){$\varphi_5$}
\put(-98,43){$\varphi_6$}\put(-90,17){$\varphi_7$}\put(-63,-2){$\varphi_8$}
\put(-34,-2){$\varphi_9$}\put(-8,17){$\varphi_{10}$}
\put(-75,50){$\theta=\theta_2$}
\hskip1cm
\includegraphics[scale=0.4]{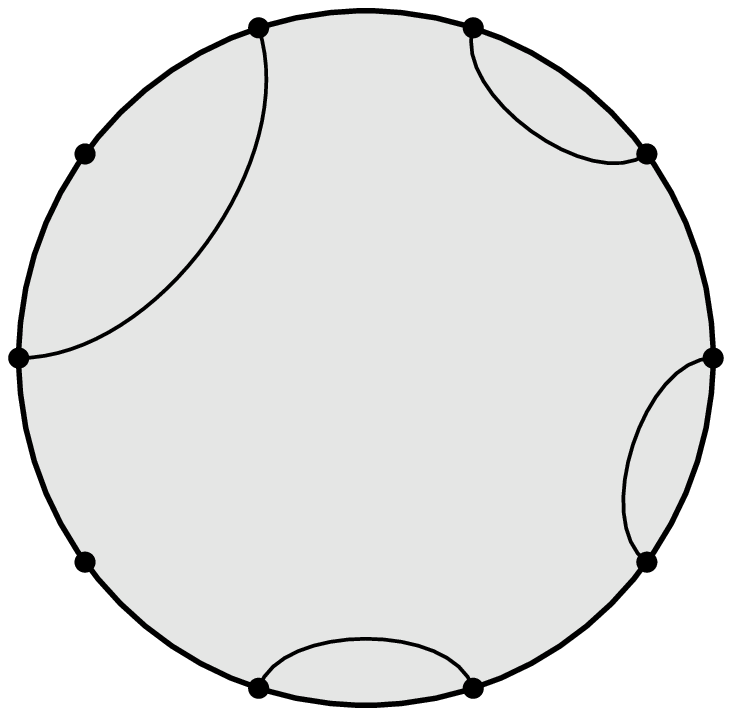}\put(0,43){$\varphi_1$}\put(-8,69){$\varphi_2$}
\put(-34,88){$\varphi_3$}\put(-63,88){$\varphi_4$}\put(-90,69){$\varphi_5$}
\put(-98,43){$\varphi_6$}\put(-90,17){$\varphi_7$}\put(-63,-2){$\varphi_8$}
\put(-34,-2){$\varphi_9$}\put(-8,17){$\varphi_{10}$}
\put(-60,50){$\theta\in(\theta_2;\theta_3)$}
\bigskip

\includegraphics[scale=0.4]{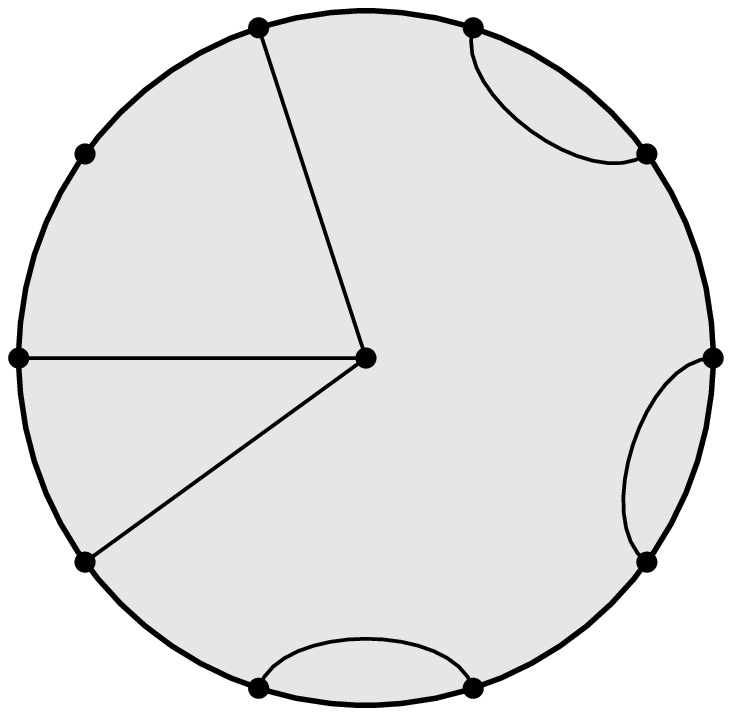}\put(0,43){$\varphi_1$}\put(-8,69){$\varphi_2$}
\put(-34,88){$\varphi_3$}\put(-63,88){$\varphi_4$}\put(-90,69){$\varphi_5$}
\put(-98,43){$\varphi_6$}\put(-90,17){$\varphi_7$}\put(-63,-2){$\varphi_8$}
\put(-34,-2){$\varphi_9$}\put(-8,17){$\varphi_{10}$}
\put(-45,50){$\theta=\theta_3$}
\hskip1cm
\includegraphics[scale=0.4]{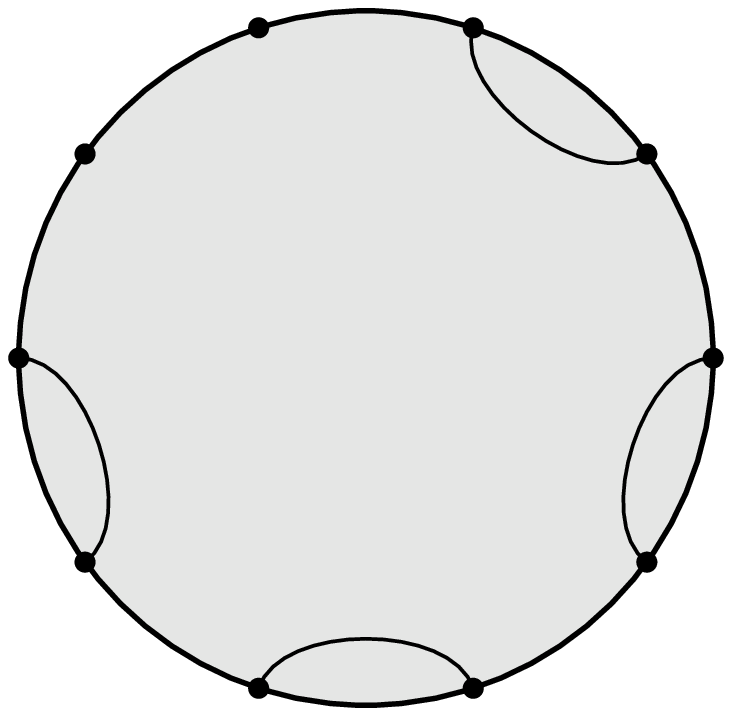}\put(0,43){$\varphi_1$}\put(-8,69){$\varphi_2$}
\put(-34,88){$\varphi_3$}\put(-63,88){$\varphi_4$}\put(-90,69){$\varphi_5$}
\put(-98,43){$\varphi_6$}\put(-90,17){$\varphi_7$}\put(-63,-2){$\varphi_8$}
\put(-34,-2){$\varphi_9$}\put(-8,17){$\varphi_{10}$}
\put(-67,50){$\theta\in(\theta_3;\theta_4)$}
\hskip1cm
\includegraphics[scale=0.4]{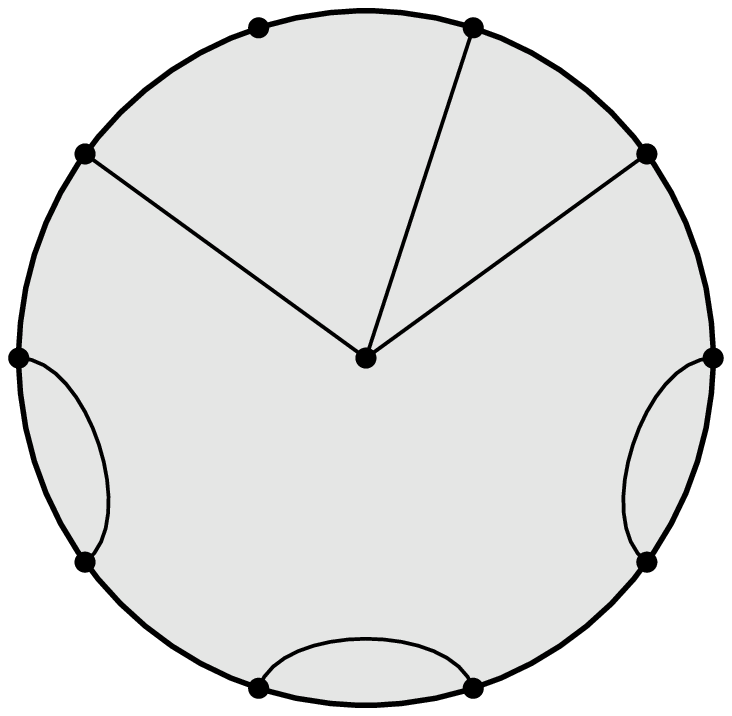}\put(0,43){$\varphi_1$}\put(-8,69){$\varphi_2$}
\put(-34,88){$\varphi_3$}\put(-63,88){$\varphi_4$}\put(-90,69){$\varphi_5$}
\put(-98,43){$\varphi_6$}\put(-90,17){$\varphi_7$}\put(-63,-2){$\varphi_8$}
\put(-34,-2){$\varphi_9$}\put(-8,17){$\varphi_{10}$}
\put(-55,35){$\theta=\theta_4$}
\hskip1cm
\includegraphics[scale=0.4]{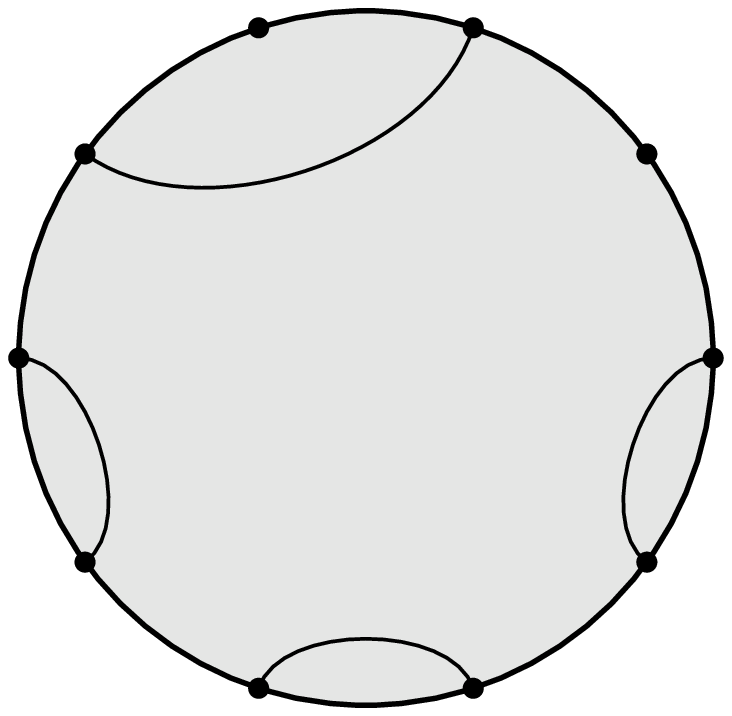}\put(0,43){$\varphi_1$}\put(-8,69){$\varphi_2$}
\put(-34,88){$\varphi_3$}\put(-63,88){$\varphi_4$}\put(-90,69){$\varphi_5$}
\put(-98,43){$\varphi_6$}\put(-90,17){$\varphi_7$}\put(-63,-2){$\varphi_8$}
\put(-34,-2){$\varphi_9$}\put(-8,17){$\varphi_{10}$}
\put(-67,50){$\theta\in(\theta_4;\theta_5)$}

\bigskip
\includegraphics[scale=0.4]{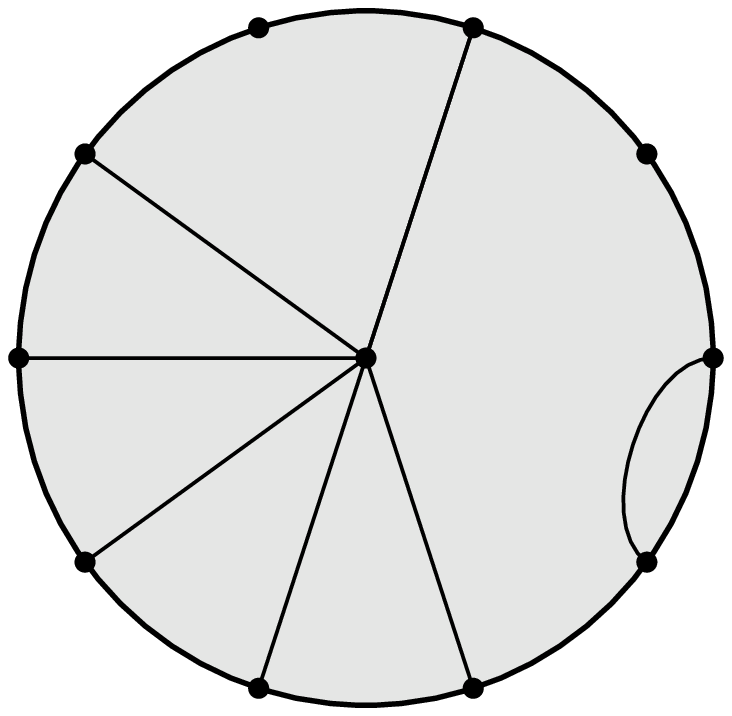}\put(0,43){$\varphi_1$}\put(-8,69){$\varphi_2$}
\put(-34,88){$\varphi_3$}\put(-63,88){$\varphi_4$}\put(-90,69){$\varphi_5$}
\put(-98,43){$\varphi_6$}\put(-90,17){$\varphi_7$}\put(-63,-2){$\varphi_8$}
\put(-34,-2){$\varphi_9$}\put(-8,17){$\varphi_{10}$}
\put(-40,50){$\theta=\theta_5$}
\hskip1cm
\includegraphics[scale=0.4]{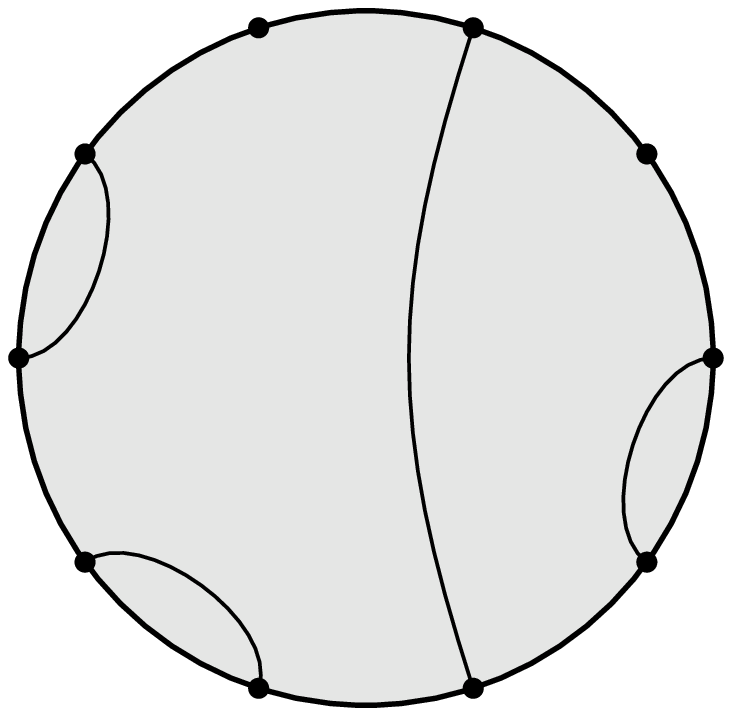}\put(0,43){$\varphi_1$}\put(-8,69){$\varphi_2$}
\put(-34,88){$\varphi_3$}\put(-63,88){$\varphi_4$}\put(-90,69){$\varphi_5$}
\put(-98,43){$\varphi_6$}\put(-90,17){$\varphi_7$}\put(-63,-2){$\varphi_8$}
\put(-34,-2){$\varphi_9$}\put(-8,17){$\varphi_{10}$}
\put(-69,50){$\theta\in(\theta_5;\theta_6)$}
\hskip1cm
\includegraphics[scale=0.4]{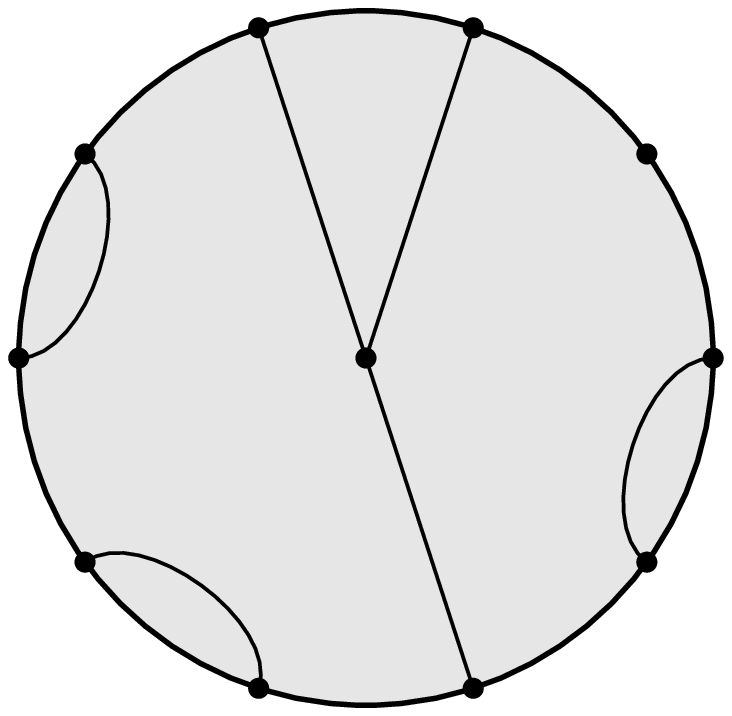}\put(0,43){$\varphi_1$}\put(-8,69){$\varphi_2$}
\put(-34,88){$\varphi_3$}\put(-63,88){$\varphi_4$}\put(-90,69){$\varphi_5$}
\put(-98,43){$\varphi_6$}\put(-90,17){$\varphi_7$}\put(-63,-2){$\varphi_8$}
\put(-34,-2){$\varphi_9$}\put(-8,17){$\varphi_{10}$}
\put(-40,50){$\theta=\theta_6$}
\hskip1cm
\includegraphics[scale=0.4]{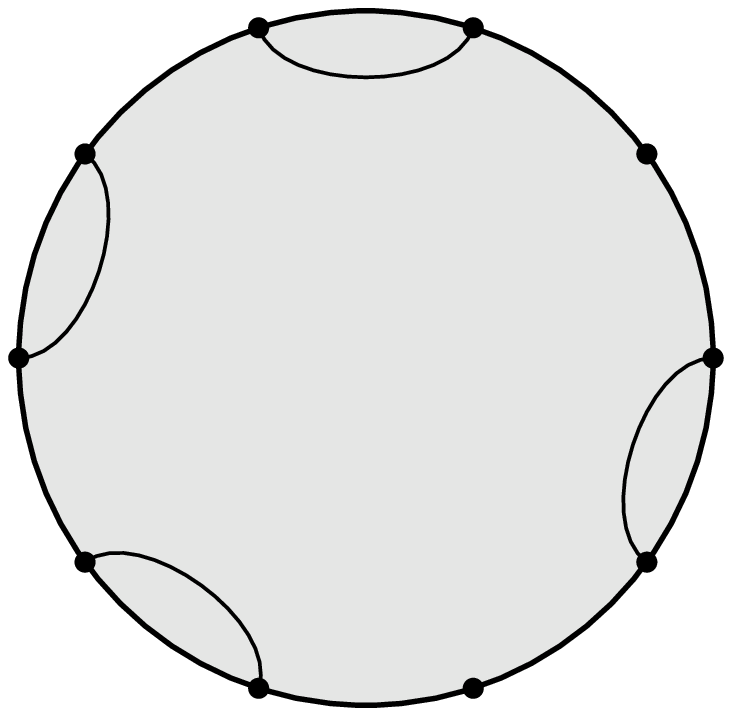}\put(0,43){$\varphi_1$}\put(-8,69){$\varphi_2$}
\put(-34,88){$\varphi_3$}\put(-63,88){$\varphi_4$}\put(-90,69){$\varphi_5$}
\put(-98,43){$\varphi_6$}\put(-90,17){$\varphi_7$}\put(-63,-2){$\varphi_8$}
\put(-34,-2){$\varphi_9$}\put(-8,17){$\varphi_{10}$}
\put(-67,50){$\theta\in(\theta_6;\theta_7)$}
\hskip1cm

\bigskip
\includegraphics[scale=0.4]{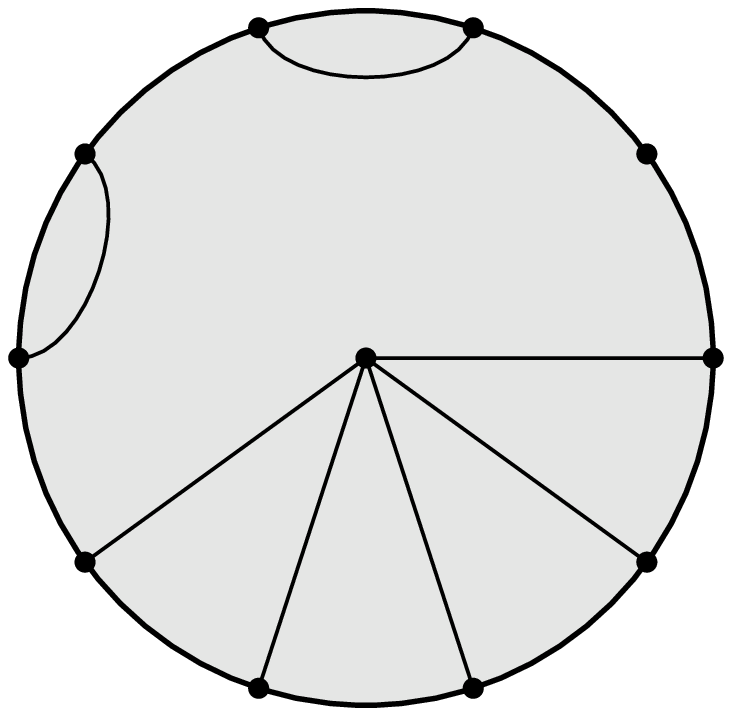}\put(0,43){$\varphi_1$}\put(-8,69){$\varphi_2$}
\put(-34,88){$\varphi_3$}\put(-63,88){$\varphi_4$}\put(-90,69){$\varphi_5$}
\put(-98,43){$\varphi_6$}\put(-90,17){$\varphi_7$}\put(-63,-2){$\varphi_8$}
\put(-34,-2){$\varphi_9$}\put(-8,17){$\varphi_{10}$}
\put(-55,50){$\theta=\theta_7$}
\hskip1cm
\includegraphics[scale=0.4]{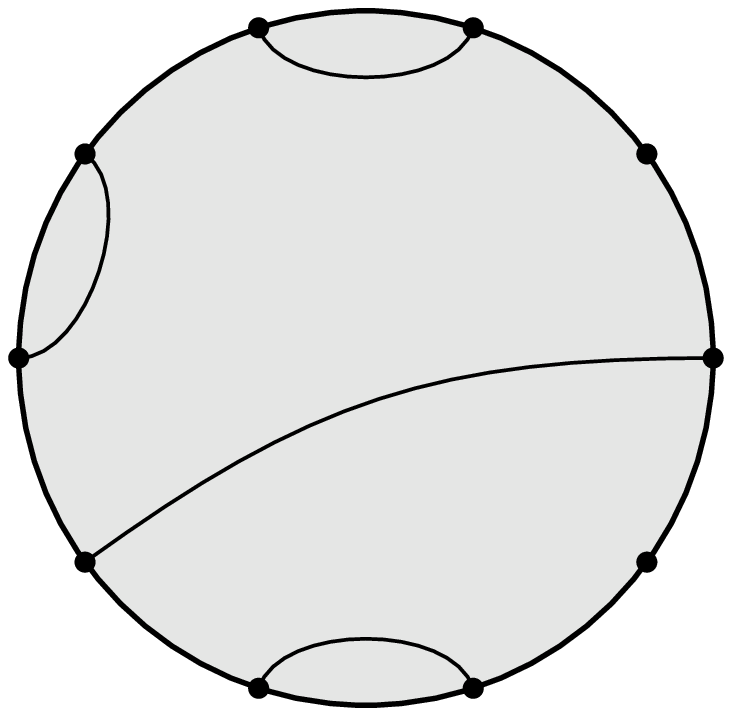}\put(0,43){$\varphi_1$}\put(-8,69){$\varphi_2$}
\put(-34,88){$\varphi_3$}\put(-63,88){$\varphi_4$}\put(-90,69){$\varphi_5$}
\put(-98,43){$\varphi_6$}\put(-90,17){$\varphi_7$}\put(-63,-2){$\varphi_8$}
\put(-34,-2){$\varphi_9$}\put(-8,17){$\varphi_{10}$}
\put(-67,50){$\theta\in(\theta_7;\theta_8)$}
\hskip1cm
\includegraphics[scale=0.4]{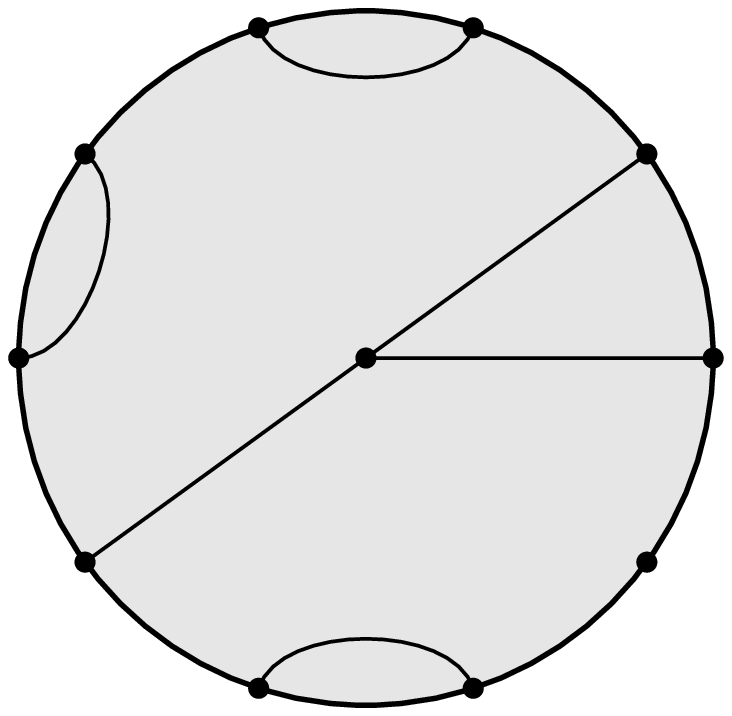}\put(0,43){$\varphi_1$}\put(-8,69){$\varphi_2$}
\put(-34,88){$\varphi_3$}\put(-63,88){$\varphi_4$}\put(-90,69){$\varphi_5$}
\put(-98,43){$\varphi_6$}\put(-90,17){$\varphi_7$}\put(-63,-2){$\varphi_8$}
\put(-34,-2){$\varphi_9$}\put(-8,17){$\varphi_{10}$}
\put(-68,50){$\theta=\theta_8$}
\hskip1cm
\includegraphics[scale=0.4]{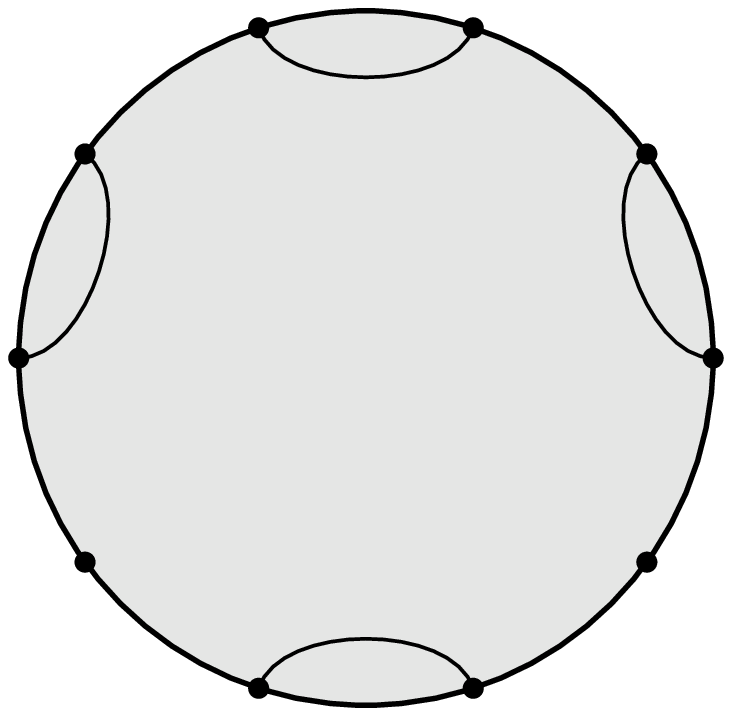}\put(0,43){$\varphi_1$}\put(-8,69){$\varphi_2$}
\put(-34,88){$\varphi_3$}\put(-63,88){$\varphi_4$}\put(-90,69){$\varphi_5$}
\put(-98,43){$\varphi_6$}\put(-90,17){$\varphi_7$}\put(-63,-2){$\varphi_8$}
\put(-34,-2){$\varphi_9$}\put(-8,17){$\varphi_{10}$}
\put(-67,50){$\theta\in(\theta_8;\theta_1)$}

\caption{A rectangular diagram of a surface~$\Pi$
and the topology of the intersections~$\widehat\Pi\cap\mathscr P_\theta$
for various~$\theta$}\label{movieexam-fig}
\end{figure}

\end{exam}

Note that if~$m_{\theta_0}$ is an occupied meridian of~$\Pi$ that contains exactly two vertices of~$\Pi$ and no
vertices from~$\partial\Pi$, then, from the topological point of view, the intersection~$\mathscr P_{\theta_0}\cap\widehat\Pi$ has no singularity,
and there is no event of~$\Psi_\Pi$ at the moment~$\theta_0$. (For instance, such a situation
occurs in the case shown in Figure~\ref{movieexam-fig} for the midpoint of
the interval~$(\theta_7;\theta_8)$. What does change in~$\mathscr P_\theta\cap\widehat\Pi$
at this moment is the relative position of the arc joining~$\varphi_1$ with~$\varphi_7$
and the point~$\mathscr P_\theta\cap\mathbb S^1_{\tau=1}$.)
In particular, a vertical bubble move
applied to~$\Pi$ does not change~$\Psi_\Pi$.

\begin{defi}
Let~$F_1$ and~$F_2$ be two surfaces representing the same movie diagram. A morphism from~$F_1$ to~$F_2$
is said to be \emph{canonical} if it can be represented by a self-homeomorphism of~$\mathbb S^3$ that
fixes~$\mathbb S^1_{\tau=0}$ pointwise and preserve every page~$\mathscr P_\theta$.
Such a morphism is clearly unique, so it will be refereed to as \emph{the} canonical morphism from~$F_1$ to~$F_2$.
\end{defi}

\begin{prop}\label{representing-movie-prop}
\emph{(i)}
Any movie diagram has the form~$\Psi_\Pi$ for some rectangular diagram of a surface~$\Pi$.

\emph{(ii)}
We have~$\Psi_\Pi=\Psi_{\Pi'}$ if and only if the diagrams~$\Pi$ and~$\Pi'$ are related by a finite
sequence of vertical bubble moves and vertical half-wrinkle moves preserving the boundary of the diagram.
Any such sequence induces the canonical morphism~$\widehat\Pi\rightarrow\widehat\Pi'$.
\end{prop}

\begin{proof}
The fact that vertical bubble and half-wrinkle moves induce the canonical morphisms between
the respective surfaces follows easily the definitions of the moves (Definitions~\ref{bubble-def} and~\ref{wrinkledef}).
So, we need only to explain how to recover a rectangular diagram of a surface~$\Pi$ from~$\Psi_\Pi$, and to show
that the arbitrariness
in this procedure amounts to an application of a sequence of vertical bubble and half-wrinkle moves to~$\Pi$.
We provide a sketch only, since the details are pretty elementary.

We need some preparation before describing the procedure.
For a non-crossing chord diagram~$C$, by \emph{a region of~$C$} we mean
a union of intervals
$$R=(\varphi_1;\varphi_2)\cup\ldots\cup(\varphi_{2k-1};\varphi_{2k})\subset\mathbb S^1$$ such that:
\begin{enumerate}
\item
$R$ is disjoint from any chord of~$C$;
\item
$\{\varphi_2,\varphi_3\},\ldots,\{\varphi_{2k-2};\varphi_{2k-1}\},\{\varphi_{2k},\varphi_1\}$ are chords of~$C$.
\end{enumerate}
The subset~$\{\varphi_1,\varphi_2,\ldots,\varphi_{2k}\}\subset\mathbb S^1$ will be referred to as \emph{the boundary of~$R$}
and denoted~$\partial R$.

In other words, a region of~$C$ is a maximal subset~$R$ of~$\mathbb S^1$ such that, for any two distinct points~$\varphi'$, $\varphi''$ of~$R$,
the addition of~$\{\varphi',\varphi''\}$ to~$C$ yields a non-crossing chord diagram.

Clearly, any two distinct regions of~$C$ are disjoint, and their boundaries are either disjoint or have exactly two points in common,
which form a chord of~$C$. In the latter case the regions are called \emph{neighboring}.

Denote by~$T(C)$ the graph whose vertices are regions of~$C$, and edges are pairs of neighboring regions.
This graph is obviously a tree. The set of all subsets of~$\mathbb S^1$ that have the form of a union
of finitely many pairwise disjoint open intervals is denoted by~$\mathscr R$.

Let~$\Pi$ be a rectangular diagram of a surface. For every~$\theta_0\in\mathbb S^1$ such that~$m_{\theta_0}$ is not
an occupied meridian of~$\Pi$, define~$R_\Pi(\theta_0)$ to be the region~$\Omega(\theta_0)\cap\mathbb S^1_{\tau=0}$
of~$\Psi_\Pi(\theta_0)$,
where~$\Omega(\theta_0)$ is the connected component of~$\mathscr P_{\theta_0}\setminus\widehat\Pi$
that contains the point~$\mathscr P_{\theta_0}\cap\mathbb S^1_{\tau=1}$. One can see that the following equality
is an equivalent characterization of~$R_\Pi(\theta_0)$:
\begin{equation}\label{intersection-eq}
\{\theta_0\}\times R_\Pi(\theta_0)=m_{\theta_0}\setminus\bigcup_{r\in\Pi}r.
\end{equation}
The function~$R_\Pi$ is clearly constant on every open interval between two occupied meridians of~$\Pi$.
We extend~$R_\Pi$ to be a left continuous map~$\mathbb S^1\rightarrow\mathscr R$ (the set~$\mathscr R$
is endowed with a discrete topology).

The diagram~$\Pi$ can be uniquely recovered from the union~$\bigcup_{r\in\Pi}r$. Indeed,
$\Pi$ is the set of all maximal rectangles (with respect to inclusion) contained in~$\bigcup_{r\in\Pi}r$.
Therefore, it can also be uniquely recovered from~$R_\Pi$. Indeed, due to~\eqref{intersection-eq},
the union~$\bigcup_{r\in\Pi}r$ is the closure
of~$\mathbb T^2\setminus\bigcup_{\theta\in\mathbb S^1}\bigl(\{\theta\}\times R_\Pi(\theta)\bigr)$.

Suppose that~$m_{\theta_0}$ is an occupied meridian of~$\Pi$ containing at least three vertices of~$\Pi$.
Then~$\Psi_\Pi(\theta_0)\mapsto\Psi_\Pi(\theta_0+0)$ is an admissible event involving at least three
points of~$\mathbb S^1$. One can see that there are unique regions of~$\Psi_\Pi(\theta_0)$
and~$\Psi_\Pi(\theta_0+0)$ whose closure contains all the points of~$\mathbb S^1$ involved in the event,
and these regions are~$R_\Pi(\theta_0)$ and~$R_\Pi(\theta_0+0)$, respectively.

Now suppose that a meridian~$m_{\theta_0}$ contains exactly two vertices of~$\Pi$, and they
belong to~$\partial\Pi$.
This means that the event~$\Psi_\Pi(\theta_0)\mapsto\Psi_\Pi(\theta_0+0)$ either has the form
$\{\varphi',\varphi''\}\leadsto\varnothing$ or~$\varnothing\leadsto\{\varphi',\varphi''\}$.
In the former case, there are two regions of~$\Psi_\Pi(\theta_0)$ whose
boundary contain~$\varphi'$ and~$\varphi''$, and one of them is~$R_\Pi(\theta_0)$. By applying
a vertical half-wrinkle creation move keeping~$\Psi_\Pi$ unaltered we can change~$R_\Pi(\theta_0)$ to
the other one. The region~$R_\Pi(\theta_0+0)$ is prescribed by~$\Psi_\Pi$, and this is the unique
region of~$\Psi_\Pi(\theta_0+0)$ that contains~$\varphi'$ and~$\varphi''$.\

Similarly, if the event~$\Psi_\Pi(\theta_0)\mapsto\Psi_\Pi(\theta_0+0)$ has the form~$\varnothing\leadsto
\{\varphi',\varphi''\}$, then~$R_\Pi(\theta_0)$ is prescribed by~$\Psi_\Pi$, whereas
there are two regions of~$\Psi_R(\theta_0+0)$ that are eligible for~$R_\Pi(\theta_0+0)$,
and the choice can be changed by a vertical half-wrinkle creation move applied to~$\Pi$.

We are ready to give a recipe for constructing a rectangular diagram of a surface representing a given movie diagram.
Let~$\Psi$ be an arbitrary movie diagram. Pick a function~$R:\mathbb S^1\rightarrow\mathscr R$
having the following properties:
\begin{enumerate}
\item
$R$ is left continuous and has finitely many discontinuity points;
\item
$R(\theta)$ is a region of~$\Psi(\theta)$ for all~$\theta\in\mathbb S^1$;
\item
if no admissible event of~$\Psi$ occurs at~$\theta_0\in\mathbb S^1$
and~$R(\theta_0)\ne R(\theta_0+0)$, then~$R(\theta_0)$ and~$R(\theta_0+0)$
are neighboring regions of~$\Psi(\theta_0)$;
\item
if an admissible event of~$\Psi$ occurs at~$\theta_0$, $\theta_0\in\mathbb S^1$,
then~$\overline{R(\theta_0)}$
and~$\overline{R(\theta_0+0)}$ contain all points involved
in the event;
\item
we have~$\overline{\bigcup_{\theta\in\mathbb S^1}R(\theta)}=\mathbb S^1$.
\end{enumerate}
To see that such~$R$ does exist, let~$\theta',\theta''$ be the moments of two successive events of~$\Psi$.
We start by defining~$R(\theta'+0)$ and~$R(\theta'')$ to comply with Condition~(4) above.
These are two regions of~$\Psi(\theta'')=\Psi(\theta'+0)$, and~$R$ can be defined
on the whole interval~$(\theta';\theta'']$ to satisfy Conditions~$(1)-(3)$ due
to the fact that the graph~$T(\Psi(\theta''))$ is connected, and~$\Psi(\theta)$ is
constant for~$\theta\in(\theta';\theta'']$. This is done independently
for all intervals between any two successive events.

To see that Condition~(5) can also be met by~$R$, note that the graph~$T(C)$ is
connected for any non-crossing chord diagram~$C$, and the closure of the union
of all regions of~$C$ is the entire circle~$\mathbb S^1$. So, if~$R$ satisfies Conditions~(1)--(4), but not~(5),
we choose an interval~$[\theta_1;\theta_2)$ on which~$\Psi$ is constant and modify the values of~$R$ in it
so as to let~$R(\theta)$, $\theta\in[\theta_1;\theta_2)$ visit all vertices of the tree~$T(\Psi(\theta_1))$.
This can clearly be done without violating Conditions~(1)--(4).

Now we let~$\Pi$ be the collection of maximal rectangles contained in the closure of~$\mathbb T^2\setminus\bigcup_{\theta\in\mathbb S^1}\bigl(\{\theta\}\times R_\Pi(\theta)\bigr)$. We leave it to the reader
to verify that~$\Pi$ is a rectangular diagram of a surface such that~$\Psi=\Psi_\Pi$.

The arbitrariness in the construction of~$\Pi$ has the following two sources. First, if at some moment~$\theta_0$,
an admissible event of the form~$\{\varphi',\varphi''\}\leadsto\varnothing$ (or, respectively,
$\varnothing\leadsto\{\varphi',\varphi''\}$) occurs in~$\Psi$, then~$R(\theta_0)$ (respectively,
$R(\theta_0+0)$) can be chosen in two different ways. As mentioned above, a vertical half-wrinkle move can be used
to change the choice.

Second, there is a large freedom in
defining~$R$ on an interval~$(\theta';\theta'']$ between two successive events of~$\Psi$,
provided that~$R(\theta'+0)$ and~$R(\theta'')$ are already fixed. Different choices
correspond to different paths from~$R(\theta'+0)$ to~$R(\theta'')$ in the graph~$T(\Psi(\theta''))$.
Since this graph is a tree, different choices are related
by a sequence of the following operations and their inverses.

Let~$\Psi$ and~$R$ be constant on an interval~$[\theta_1;\theta_2]$, and let~$[\theta_3;\theta_4]$
be a subinterval of~$(\theta_1;\theta_2)$. Change the value of~$R$ on~$(\theta_3;\theta_4]$
to any region of~$\Psi(\theta_1)$ which is neighboring to~$R(\theta_1)=R(\theta_2)$.

One can see that such an operation results in a vertical bubble creation move performed on~$\Pi$.
This completes the proof of the proposition.
\end{proof}

It is often useful to look at a movie diagram~$\Psi$ `from inside a surface representing~$\Psi$'.
Let~$F$ be such a surface. We denote by~$\mathscr F_F$ the singular foliation induced on~$F$ by the open
book decomposition~$\{\mathscr P_\theta\}_{\theta\in\mathbb S^1}$. More precisely,
the foliation~$\mathscr F_F$ is defined on~$F\setminus\mathbb S^1_{\tau=0}$.
The leaves of~$\mathscr F_F$ are connected components of the intersections~$F\cap\mathscr(P_\theta\setminus\mathbb S^1)$,
$\theta\in\mathbb S^1$.

The points in~$F\cap\mathbb S^1_{\tau=0}$, where~$\mathscr F_F$ is not defined,
are referred to as \emph{vertices of~$\mathscr F_F$}.
Unless otherwise specified, we adopt a purely topological point of view on~$\mathscr F_F$ and regard all points~$p\in F\setminus\mathbb S^1_{\tau=0}$
such that a sufficiently small neighborhood of~$p$ is foliated by open arcs
as \emph{regular}. So, \emph{singularities of~$\mathscr F_F$} are such points in~$F\setminus\mathbb S^1_{\tau=0}$
that do not have an open neighborhood foliated by open arcs. Leaves containing a singularity
are called \emph{singular} (and otherwise \emph{regular}). Connected components of the set of regular
points of a singular leaf are called \emph{separatrices}. The behavior of~$\mathscr F_F$
near vertices and some singularities is shown in Figure~\ref{sing-pic}
\begin{figure}[ht]
\includegraphics{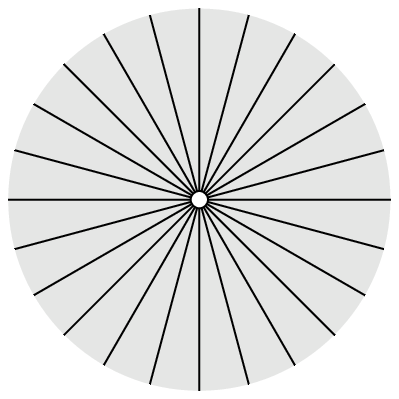}
\includegraphics{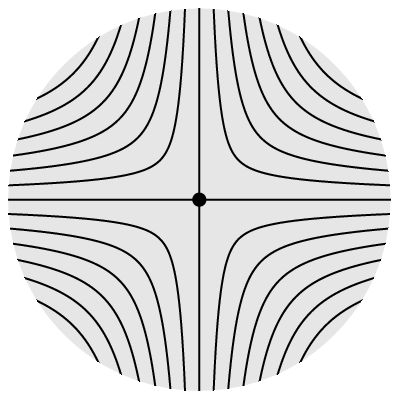}
\includegraphics{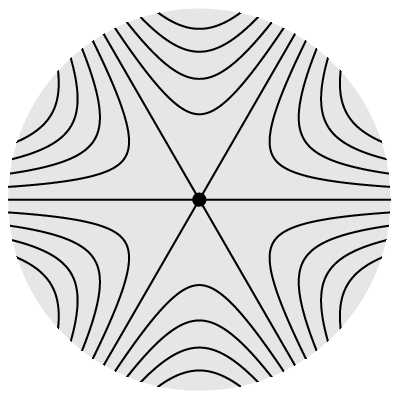}

\includegraphics{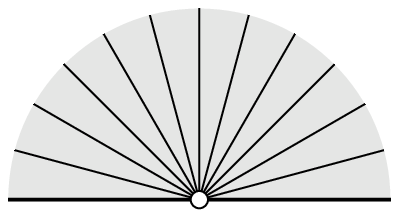}
\includegraphics{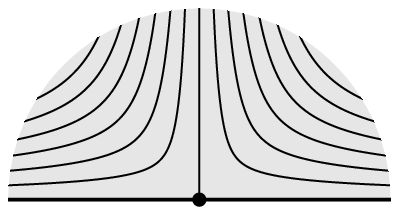}
\includegraphics{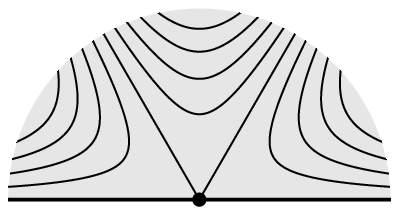}
\caption{Foliation~$\mathscr F_F$ around vertices and singularities}\label{sing-pic}
\end{figure}

If~$F$ and~$F'$ are two surfaces representing the same movie diagram~$\Psi$, then
there is a homeomorphism representing the canonical morphism~$F\rightarrow F'$
that takes~$\mathscr F_F$ to~$\mathscr F_{F'}$. In other words,
the topological structure of~$\mathscr F_F$ does not depend on the concrete choice of~$F$.
For this reason we will use the notation~$\mathscr F_\Psi$ for any foliation~$\mathscr F_F$
with~$F$ representing~$\Psi$.

Let~$\Psi$ be a movie diagram.
Quite clearly, the topological type of~$\mathscr F_\Psi$ with a little more data allows
to completely recover~$\Psi$. The additional data include the value of~$\varphi$
at vertices of~$\mathscr F_\Psi$, the value of~$\theta$ at the singularities of~$\mathscr F_\Psi$,
and the coorientation of~$\mathscr F_\Psi$ defined by~$d\theta$. So, to
illustrate an alteration of~$\Psi$ it is sometimes more convenient to
show the alteration of~$\mathscr F_\Psi$ (and of the additional data).

\section{Moves of movie diagrams}\label{movie-moves-sec}

\begin{defi}
Let~$\Psi$ and~$\Psi'$ be movie diagrams.
By \emph{a morphism} from~$\Psi$ to~$\Psi'$ we mean a maximal collection~$\chi$ of
morphisms of surfaces~$F\rightarrow F'$ representing~$\Psi$ and~$\Psi'$, respectively,
such that if~$\mu:F\rightarrow F'$ and~$\mu_1:F_1\rightarrow F_1'$ are
two morphisms belonging to~$\chi$, then~$\sigma'\circ\mu=\mu_1\circ\sigma$,
where~$\sigma:F\rightarrow F_1$ and~$\sigma':F'\rightarrow F_1'$
are the canonical morphisms.

By \emph{a move} of movie diagrams we mean a pair~$(\Psi,\Psi')$
of movie diagrams assigned with a morphism~$\chi$ from~$\Psi$ to~$\Psi'$.
For such a move we use the notation~$\Psi\mapsto\Psi'$ or~$\Psi\xmapsto\chi\Psi'$
or~$\Psi\xmapsto\mu\Psi'$ if~$\mu\in\chi$.
\end{defi}

Let~$F,F'\subset\mathbb S^3$ be two surfaces, and let~$\phi$ be a self-homeomorphism of~$\mathbb S^3$
taking~$F$ to~$F'$. If~$\phi$ is identical outside of a $3$-ball~$B$ intersecting
each of~$F$ and~$F'$ in an open disc or an open half-disc we call~$\phi$ \emph{an elementary isotopy
from~$F$ to~$F'$ supported on~$B$}.

We define below several types of moves of movie diagrams.
In each case except the one of a rescaling move, for which the corresponding morphism is
described explicitly, the morphism assigned to the move is induced either by an elementary isotopy or by
a composition of two elementary
isotopies supported on two disjoint $3$-balls. So, to define the morphisms assigned to the moves it suffices
to specify the altered part of the respective surfaces.

\begin{defi}\label{finger-def}
Let~$\Psi$ and~$\Psi'$ be two movie diagrams such that there exist pairwise distinct~$\varphi_1,\varphi_2,\varphi_3,\varphi_4\in\mathbb S^1$
and distinct~$\theta_1,\theta_2\in\mathbb S^1$ satisfying the following conditions:
\begin{enumerate}
\item
the intersection of~$[\varphi_1;\varphi_2]$ with~$\Phi(\Psi)$ is empty;
\item
there are no events of~$\Psi$ in the interval~$[\theta_1;\theta_2]$
\item
if~$\theta\in(\theta_1;\theta_2]$, then
$\Psi'(\theta)$ is obtained from~$\Psi(\theta)$ by replacing~$\{\varphi_3,\varphi_4\}$ with~$\{\varphi_1,\varphi_4\}$ and~$\{\varphi_2,\varphi_3\}$;
\item
if~$\theta\in(\theta_2;\theta_1]$, then
$\Psi'(\theta)$ is obtained from~$\Psi(\theta)$ by adding the chord~$\{\varphi_1,\varphi_2\}$.
\end{enumerate}
Then we say that~$\Psi\mapsto\Psi'$ is \emph{a finger move}, and~$\Psi'\mapsto\Psi$ is \emph{an inverse finger move}.
Surfaces representing~$\Psi$ and~$\Psi'$ can be chosen to be related by an elementary isotopy;
encircled in dashed line in Figure~\ref{finger-foli-fig} are the intersections of the surfaces with the $3$-ball on which
the elementary isotopy is supported.
\end{defi}
\begin{figure}[ht]
\includegraphics[scale=.6]{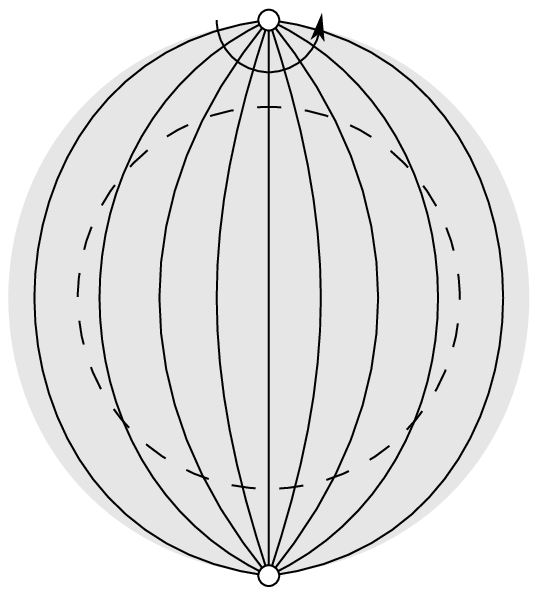}\put(-53,-5){$\varphi_4$}
\put(-53,105){$\varphi_3$}
\hskip1cm
\raisebox{50pt}{$\longrightarrow$}
\hskip1cm
\includegraphics[scale=.6]{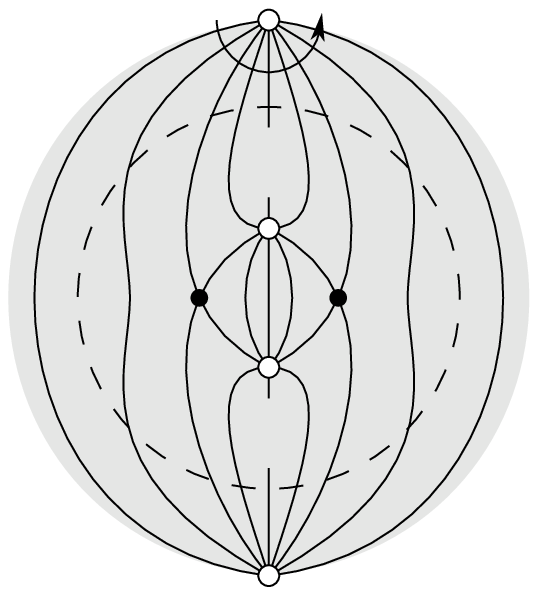}\put(-53,-5){$\varphi_4$}
\put(-53,105){$\varphi_3$}\put(-53,27){$\varphi_1$}\put(-53,73){$\varphi_2$}
\put(-70,48){$\theta_1$}\put(-34,49){$\theta_2$}
\caption{Change of the foliation~$\mathscr F_\Psi$ under a finger move. The arrows
show the coorientation of the foliation}\label{finger-foli-fig}
\end{figure}

\begin{defi}\label{splitting-def}
Let~$\Psi$ and~$\Psi'$ be two movie diagrams such that~$\Phi(\Psi)=\Phi(\Psi')$ and, for some distinct~$\theta_1,\theta_2\in\mathbb S^1$, the following conditions hold:
\begin{enumerate}
\item
for any~$\theta\in(\theta_2;\theta_1]$ we have~$\Psi(\theta)=\Psi'(\theta)$;
\item
there is exactly one event of~$\Psi$ in the interval~$[\theta_1;\theta_2]$;
\item
there are two events~$e_1,e_2$ of~$\Psi'$ at the moments~$\theta_1$ and~$\theta_2$, respectively,
and no events in~$(\theta_1;\theta_2)$;
\item
the non-crossing chord diagram~$\Psi'(\theta_2)$ contains an arc~$\{\varphi_1,\varphi_2\}$
not present in~$\Psi'(\theta_1)$ and~$\Psi'(\theta_2+0)$;
\item
all points of~$\mathbb S^1$ involved in~$e_1$ (respectively, $e_2$)
are contained in~$[\varphi_2;\varphi_1]$ (respectively, $[\varphi_1;\varphi_2]$).
\end{enumerate}
Then we say that~$\Psi\mapsto\Psi'$ is \emph{a splitting of an event},
and~$\Psi'\mapsto\Psi$ is \emph{a merging of events}.
The morphisms assigned to these moves are again induced by an elementary isotopy supported
on a $3$-ball~$B$ that contains no vertices and only those singularities of~$\mathscr F_\Psi,\mathscr F_{\Psi'}$
that correspond to the events in which~$\Psi$ and~$\Psi'$ differ.

A few examples of how the foliation~$\mathscr F_\Psi$ changes under a splitting of an event
is shown in Figure~\ref{sing-split-fig}.
Encircled in dashed line are the intersections of the surfaces with~$B$.
\begin{figure}[ht]
\includegraphics[scale=.6]{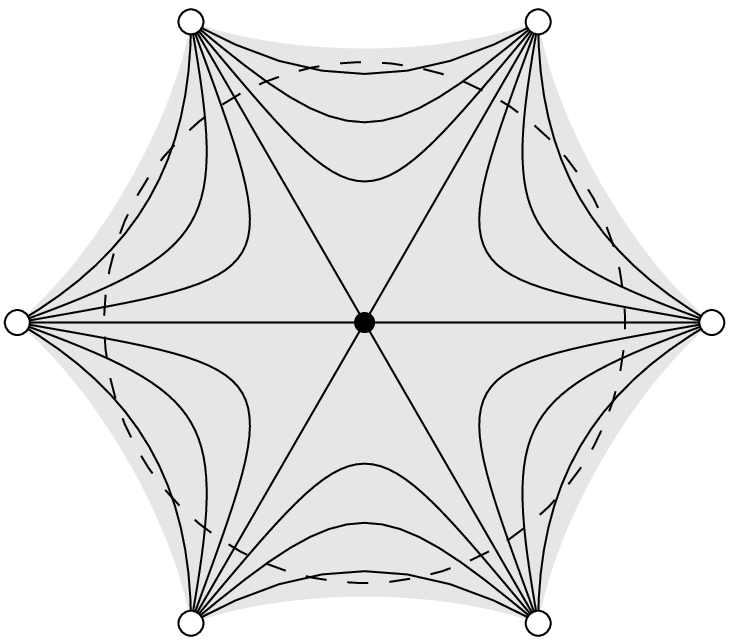}\put(2,53){$\varphi_2$}\put(-137,53){$\varphi_1$}
\hskip1cm
\raisebox{53pt}{$\longrightarrow$}
\hskip1cm
\includegraphics[scale=.6]{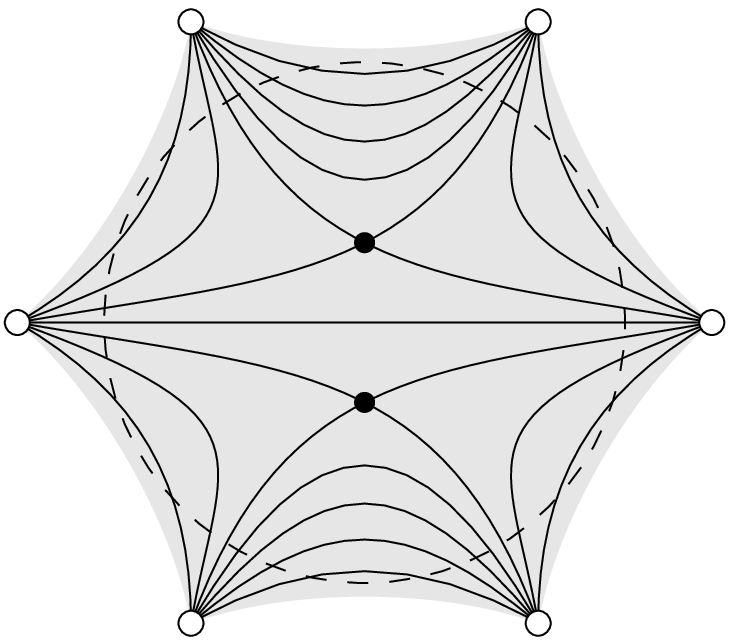}\put(2,53){$\varphi_2$}\put(-137,53){$\varphi_1$}

\vskip.5cm

\includegraphics[scale=.6]{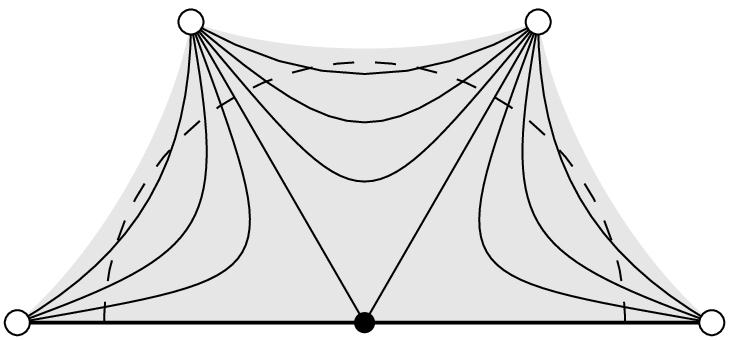}\put(2,1){$\varphi_2$}\put(-137,1){$\varphi_1$}
\hskip1cm
\raisebox{25pt}{$\longrightarrow$}
\hskip1cm
\includegraphics[scale=.6]{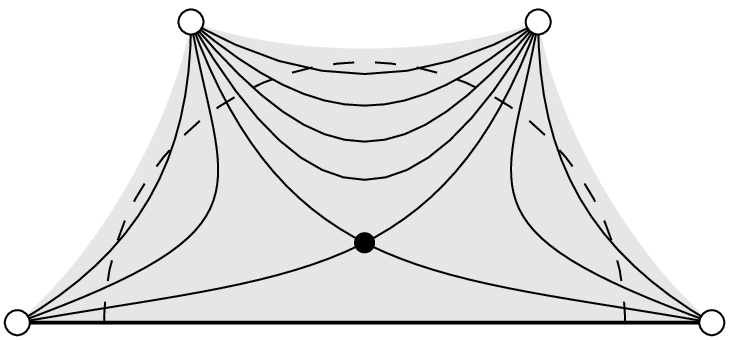}\put(2,1){$\varphi_2$}\put(-137,1){$\varphi_1$}

\vskip.5cm

\includegraphics[scale=.6]{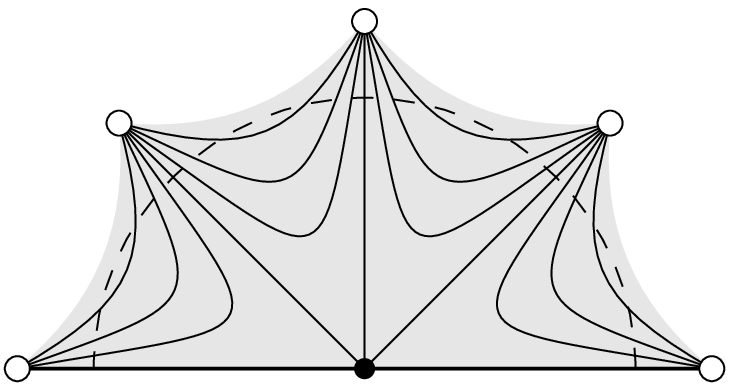}\put(-16,45){$\varphi_2$}\put(-137,1){$\varphi_1$}
\hskip1cm
\raisebox{25pt}{$\longrightarrow$}
\hskip1cm
\includegraphics[scale=.6]{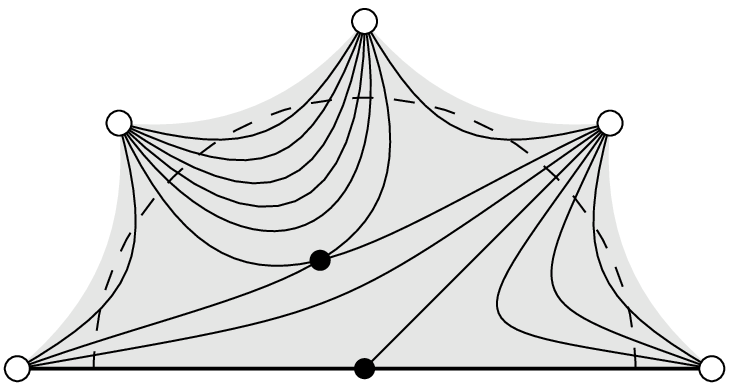}\put(-16,45){$\varphi_2$}\put(-137,1){$\varphi_1$}
\caption{Change of the foliation~$\mathscr F_\Psi$ under a splitting of an event}\label{sing-split-fig}
\end{figure}
\end{defi}

\begin{defi}\label{special-com-def}
Let~$\Psi$ and~$\Psi'$ be two movie diagrams such that, for some~$\theta_1,\theta_2,\theta_3\in\mathbb S^1$,
$\theta_2\in(\theta_1;\theta_3)$,
the non-crossing chord diagrams~$\Psi(\theta)$ and~$\Psi'(\theta)$ coincide for all~$\theta\in(\theta_3;\theta_1]$,
and no event of~$\Psi$ or~$\Psi'$ occur in~$(\theta_1;\theta_2)\cup(\theta_2;\theta_3)$.
Suppose also that there are pairwise distinct~$\varphi_1,\varphi_2,\varphi_3,\varphi_6,\varphi_5,\varphi_4\in\mathbb S^1$
following on~$\mathbb S^1$ in the same or opposite circular order as listed
(observe that it is not the natural one)
such that
the following events occur at the moments~$\theta_1,\theta_2,\theta_3$:\\[2mm]
\centerline{\begin{tabular}{|c|c|c|}
\hline
&$\Psi$&$\Psi'$\\\hline
$\theta_1$&no event&$\{\varphi_3,\varphi_4\},\{\varphi_6,\varphi_5\}\leadsto
\{\varphi_3,\varphi_6\},\{\varphi_5,\varphi_4\}$\\\hline
$\theta_2$&$\{\varphi_1,\varphi_2\},\{\varphi_3,\varphi_4\}\leadsto
\{\varphi_1,\varphi_4\},\{\varphi_2,\varphi_3\}$&$\{\varphi_1,\varphi_2\},\{\varphi_3,\varphi_6\}\leadsto
\{\varphi_1,\varphi_6\},\{\varphi_2,\varphi_3\}$
\\\hline
$\theta_3$&$\{\varphi_1,\varphi_4\},\{\varphi_6,\varphi_5\}\leadsto
\{\varphi_1,\varphi_6\},\{\varphi_5,\varphi_4\}$&no event\\\hline
\end{tabular}}
\\[2mm]Then~$\Psi\mapsto\Psi'$ is called \emph{a special commutation}.

We also use this term for a move whose definition is similar with the only distinction that~$\varphi_2$
is not involved. Thus, the events in~$\Psi$ and~$\Psi'$ at the moments~$\theta_1,\theta_2,\theta_3$ are:
\\[2mm]
\centerline{\begin{tabular}{|c|c|c|}
\hline
&$\Psi$&$\Psi'$\\\hline
$\theta_1$&no event&$\{\varphi_3,\varphi_4\},\{\varphi_6,\varphi_5\}\leadsto
\{\varphi_3,\varphi_6\},\{\varphi_5,\varphi_4\}$\\\hline
$\theta_2$&$\{\varphi_3,\varphi_4\}\leadsto
\{\varphi_1,\varphi_4\}$&$\{\varphi_3,\varphi_6\}\leadsto
\{\varphi_1,\varphi_6\}$
\\\hline
$\theta_3$&$\{\varphi_1,\varphi_4\},\{\varphi_6,\varphi_5\}\leadsto
\{\varphi_1,\varphi_6\},\{\varphi_5,\varphi_4\}$&no event\\\hline
\end{tabular}}
\\[2mm]
The morphism assigned to the special commutation is induced by an elementary isotopy 
supported on a $3$-ball~$B$ similarly to
Definition~\ref{splitting-def}. Figure~\ref{spec-com-fig}
illustrates the difference between the foliations~$\mathscr F_\Psi$
and~$\mathscr F_\Psi'$, where the intersections of the surfaces with~$B$
are encircled in dashed line.
\end{defi}
\begin{figure}[ht]
\includegraphics[scale=.6]{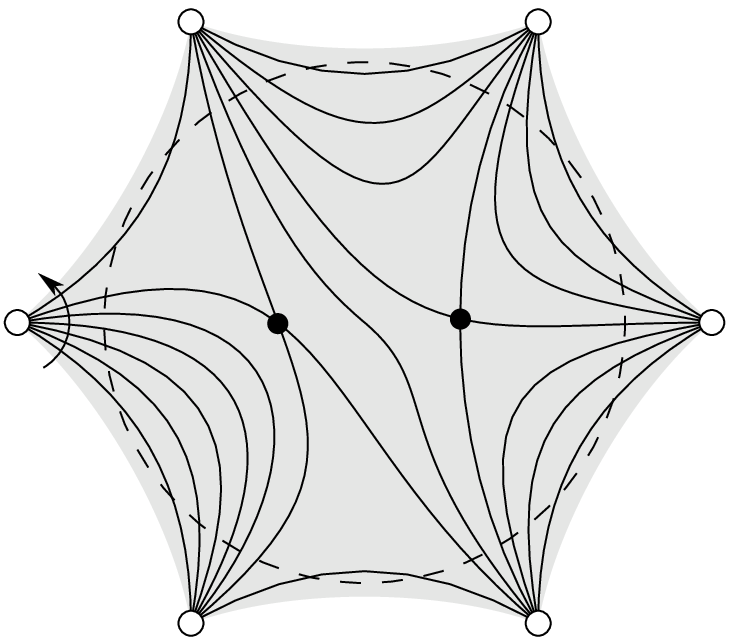}\put(-35,-6){$\varphi_1$}
\put(-98,-6){$\varphi_6$}\put(1,54){$\varphi_2$}\put(-35,114){$\varphi_3$}
\put(-98,114){$\varphi_4$}\put(-138,54){$\varphi_5$}
\hskip1cm
\raisebox{53pt}{$\longrightarrow$}
\hskip1cm
\includegraphics[scale=.6]{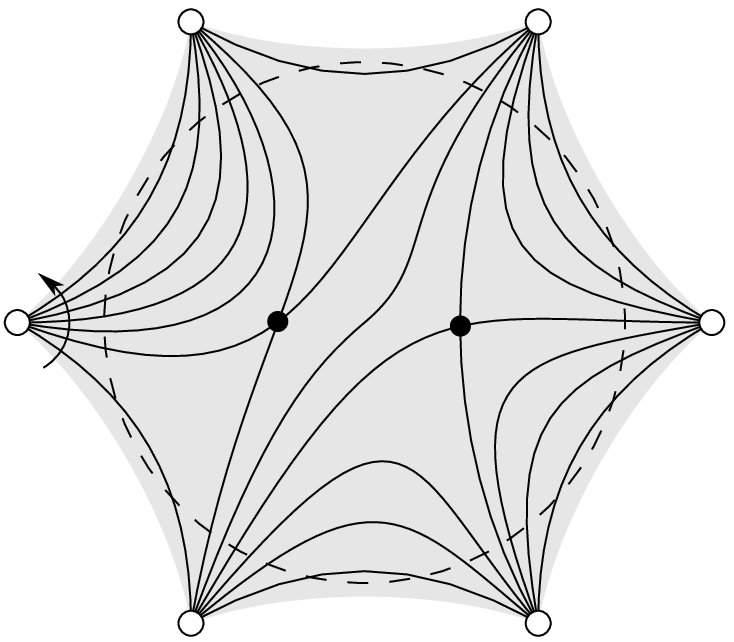}\put(-35,-6){$\varphi_1$}
\put(-98,-6){$\varphi_6$}\put(1,54){$\varphi_2$}\put(-35,114){$\varphi_3$}
\put(-98,114){$\varphi_4$}\put(-138,54){$\varphi_5$}

\vskip.5cm
\includegraphics[scale=.6]{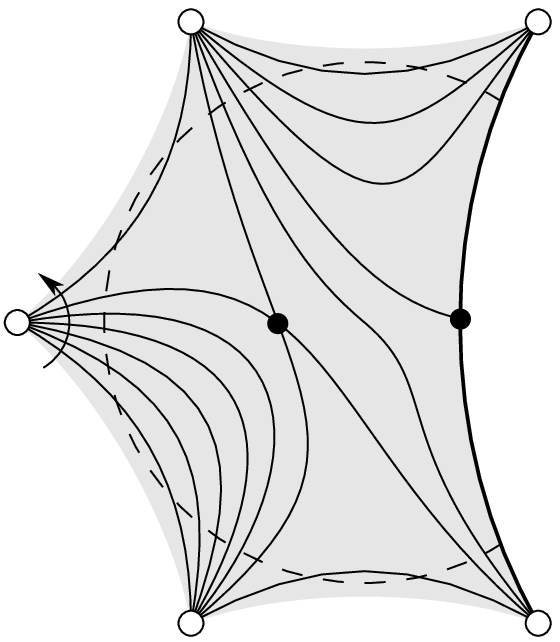}\put(-5,-6){$\varphi_1$}
\put(-68,-6){$\varphi_6$}\put(-5,114){$\varphi_3$}
\put(-68,114){$\varphi_4$}\put(-108,54){$\varphi_5$}
\hskip1cm
\raisebox{53pt}{$\longrightarrow$}
\hskip1cm
\includegraphics[scale=.6]{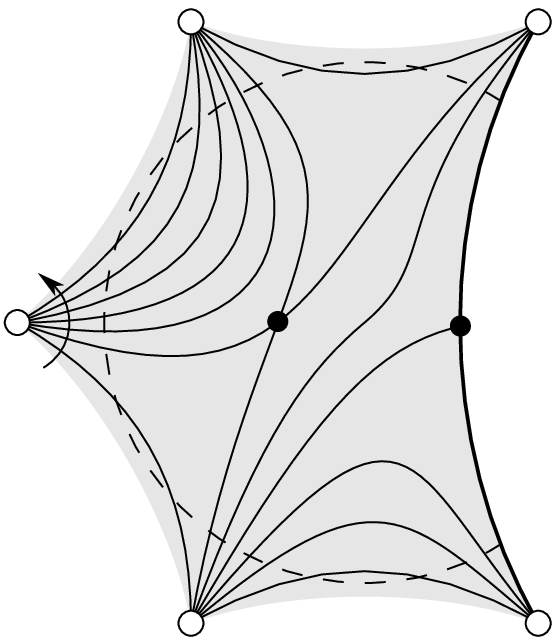}\put(-5,-6){$\varphi_1$}
\put(-68,-6){$\varphi_6$}\put(-5,114){$\varphi_3$}
\put(-68,114){$\varphi_4$}\put(-108,54){$\varphi_5$}
\caption{Change of the foliation under a special commutation}\label{spec-com-fig}
\end{figure}

Two admissible events~$C_1\mapsto C_2$ and~$C_1'\mapsto C_2'$ are considered equal
if they `do the same', that is, remove the same chords, and add the same chords:
$$C_1\setminus C_2=C_1'\setminus C_2',\quad
C_2\setminus C_1=C_2'\setminus C_1'.$$

\begin{defi}\label{coummutation-def}
Let~$\Psi$ and~$\Psi'$ be two movie diagrams such that, for some distinct~$\theta_1,\theta_3\in\mathbb S^1$,
the following holds:
\begin{enumerate}
\item
$\Psi(\theta)=\Psi'(\theta)$ for all~$\theta\in[\theta_3;\theta_1]$;
\item
each of~$\Psi$ and~$\Psi'$ has exactly two events in the interval~$(\theta_1;\theta_3)$,
these events for~$\Psi'$ are the same as for~$\Psi$, but follow in the opposite order.
\end{enumerate}
Then we say that~$\Psi\mapsto\Psi'$ is \emph{a \emph(non-special\emph) commutation}.

The morphism assigned to a commutation~$\Psi\mapsto\Psi'$ is a composition of two
morphisms that can be represented by elementary isotopies supported
on two $3$-balls which are disjoint from one another and from~$\mathbb S^1_{\tau=0}$
such that each of them
contains exactly one singularity of each of~$\mathscr F_\Psi$ and~$\mathscr F_\Psi'$.
\end{defi}

\begin{defi}\label{rescaling-def}
Let~$\Psi$ be a movie diagram, and let~$F\subset\mathbb S^3$ be a surface representing~$\Psi$.
Let also~$f,g$ be orientation preserving homeomorphisms of~$\mathbb S^1$. There is a unique
movie diagram~$\Psi'$ represented by~$(f*g)(F)$.
This diagram can be formally written as
$$\Psi'=(S^2f)\circ\Psi\circ g^{-1},$$
where~$S^2f$ is the self-homeomorphism of the symmetric square~$S^2\mathbb S^1$ induced by~$f$.

In this situation, the passage~$\Psi\xmapsto{[f*g]}\Psi'$ is called \emph{a rescaling}.
\end{defi}

\begin{lemm}\label{decomp-of-movie-moves-lem}
Let~$\Psi\xmapsto\chi\Psi'$ be one of the moves
introduced by any of the Definitions~\ref{finger-def}--\ref{rescaling-def}.
Then there exist rectangular diagrams of surfaces~$\Pi$ and~$\Pi'$
representing~$\Psi$ and~$\Psi'$, respectively, and a sequence of basic moves not including half-wrinkle moves
\begin{equation}\label{pi-seq-eq}
\Pi=\Pi_0\mapsto\Pi_1\mapsto\ldots\mapsto\Pi_k=\Pi'
\end{equation}
such that the move~$\Pi_{i-1}\mapsto\Pi_i$ is fixed on~$\partial\Pi\cap\partial\Pi'$
for all~$i=1,\ldots,k$,
and the morphism~$\widehat\Pi\rightarrow\widehat\Pi'$ obtained
by composing the moves~\eqref{pi-seq-eq} represents~$\chi$.
\end{lemm}

\begin{proof}
We consider all types of moves one by one. The fact that the morphism
obtained from the constructed sequence of basic moves is~$\chi$,
is pretty obvious in each case. So, we concentrate on the description of
the decomposition.

\medskip\noindent{\it Case 1\emph: $\Psi\xmapsto\chi\Psi'$ is a finger move}.\\
We use the notation from Definition~\ref{finger-def}. Since~$\Psi$ has no
events in the interval~$[\theta_1;\theta_2]$, we may choose a rectangular
diagram of a surface~$\Pi$ so that~$\Psi=\Psi_\Pi$ and $\Pi$ has no occupied
meridians in~$(\theta_1;\theta_2)\times\mathbb S^1$.
Since~$\{\varphi_3,\varphi_4\}\in\Psi(\theta)$ for all~$\theta\in[\theta_1;\theta_2]$
we may additionally ensure that~$\Pi$ contains the rectangle~$r=[\theta_1;\theta_2]\times[\varphi_4;\varphi_3]$.

Let~$\Pi'$ be the rectangular diagram of a surface obtained from~$\Pi$ by replacing~$r$
with the rectangles
$$[\theta_1;\theta_2]\times[\varphi_4;\varphi_1],\quad[\theta_1;\theta_2]\times[\varphi_2;\varphi_3],\quad
[\theta_2;\theta_1]\times[\varphi_1;\varphi_2].$$
One can see that~$\Psi'=\Psi_{\Pi'}$, and $\Pi\mapsto\Pi'$ is a horizontal bubble creation move,
which completes the proof in this case.

\medskip\noindent{\it Case 2\emph: $\Psi\xmapsto\chi\Psi'$ is a splitting of an event}.\\
We use the notation from Definition~\ref{splitting-def}.
A rectangular diagram of a surface~$\Pi'$ such that~$\Psi'=\Psi_{\Pi'}$
can be chosen so that the following two rectangles belong to~$\Pi'$:
$$[\theta_1;\theta_3]\times[\varphi_1;\varphi_2],\quad
[\theta_3;\theta_2]\times[\varphi_2;\varphi_1],$$
where~$\theta_3$ is the midpoint of the interval~$[\theta_1;\theta_2]$.
We may also ensure that~$\Pi'$ has no occupied meridians in the domain~$(\theta_1;\theta_2)\times\mathbb
S^1$ except~$m_{\theta_3}$.

There is then a vertical wrinkle reduction move~$\Pi'\mapsto\Pi$ reducing
these two rectangles and producing a diagram~$\Pi$ such that~$\Psi=\Psi_\Pi$.
This case is also done.

\medskip\noindent{\it Case 3\emph: $\Psi\xmapsto\chi\Psi'$ is a special commutation}.\\
We use the notation from Definition~\ref{special-com-def}. Without loss of generality,
we may assume that~$\varphi_1,\varphi_2,\varphi_3,\varphi_6,\varphi_5,\varphi_4$ follow on~$\mathbb S^1$
in the same cyclic order as listed. We consider the first version of the move (with~$\varphi_2$ involved).

We may choose a rectangular diagram of a surface~$\Pi$ such that~$\Psi=\Psi_\Pi$,
and~$\Pi$ has no occupied meridians in the domain~$\bigl([\theta_1;\theta_2)\cup(\theta_2;\theta_3)\bigr)\times\mathbb S^1$.
By applying a vertical bubble creation move, we may achieve that~$\Pi$
has rectangles of the form
$$r_0=[\theta_4;\theta_1]\times[\varphi_4;\varphi_3],\quad
r_1=[\theta_1;\theta_2]\times[\varphi_3;\varphi_4],\quad\theta_4\in(\theta_3;\theta_1).$$
It follows from Definition~\ref{special-com-def}
that there are also the following rectangles in~$\Pi$:
$$r_2=[\theta_2;\theta_3]\times[\varphi_4;\varphi_1],\quad r_3=[\theta_3;\theta_5]\times[\varphi_5;\varphi_4],\quad
r_4=[\theta_6;\theta_3]\times[\varphi_6;\varphi_5],$$
$$r_5=[\theta_3;\theta_7]\times[\varphi_1;\varphi_6],\quad
r_6=[\theta_8;\theta_2]\times[\varphi_1;\varphi_2],\quad r_7=[\theta_2;\theta_9]\times[\varphi_2;\varphi_3]$$
for some~$\theta_5,\ldots,\theta_9\in\mathbb S^1$ (see Figure~\ref{comm-flype-fig}).
\begin{figure}[ht]
\begin{tabular}{ccc}
\includegraphics{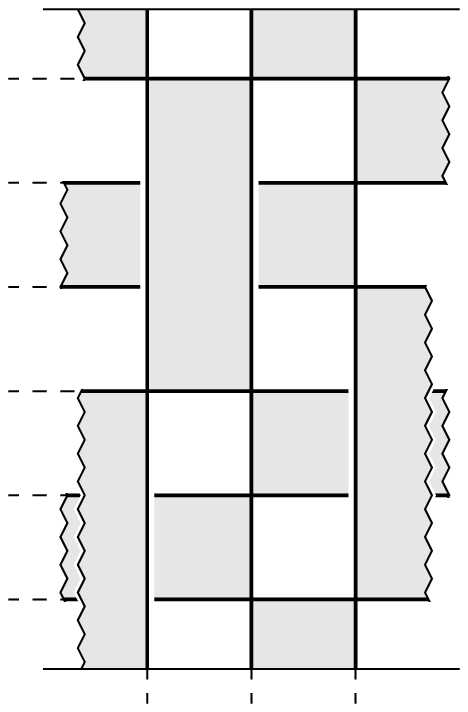}\put(-103,-5){$\theta_1$}\put(-73,-5){$\theta_2$}\put(-43,-5){$\theta_3$}
\put(-152,34){$\varphi_1$}\put(-152,64){$\varphi_2$}\put(-152,94){$\varphi_3$}
\put(-152,124){$\varphi_6$}\put(-152,154){$\varphi_5$}\put(-152,184){$\varphi_4$}
\put(-113,75){$r_0$}\put(-88,48){$r_6$}\put(-58,78){$r_7$}\put(-58,138){$r_4$}
\put(-88,163){$r_1$}\put(-113,193){$r_0$}\put(-58,23){$r_2$}\put(-58,193){$r_2$}
\put(-33,168){$r_3$}\put(-33,105){$r_5$}
&\hbox to1cm{\hss}&
\includegraphics{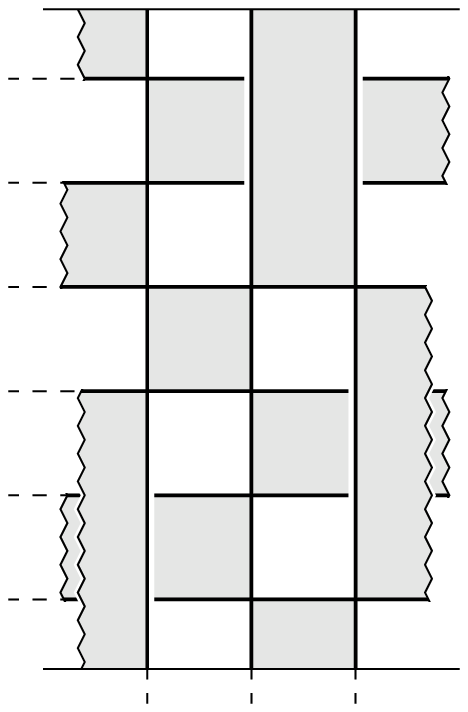}\put(-103,-5){$\theta_1$}\put(-73,-5){$\theta_2$}\put(-43,-5){$\theta_3$}
\put(-152,34){$\varphi_1$}\put(-152,64){$\varphi_2$}\put(-152,94){$\varphi_3$}
\put(-152,124){$\varphi_6$}\put(-152,154){$\varphi_5$}\put(-152,184){$\varphi_4$}
\put(-116,138){$r_4'$}\put(-58,23){$r_2'$}\put(-58,143){$r_2'$}
\put(-88,108){$r_1'$}\put(-88,168){$r_3'$}
\\[3mm]
$\Pi$&&$\Pi'$
\end{tabular}
\caption{Realizing a special commutation by means of a flype}\label{comm-flype-fig}
\end{figure}

Let~$\Pi'$ be a rectangular diagram of a surface obtained by replacing the rectangles~$r_1,r_2,r_3,r_4$ with
$$r_1'=[\theta_1;\theta_2]\times[\varphi_3;\varphi_6],\quad r_2'=[\theta_2;\theta_3]\times[\varphi_6;\varphi_1],\quad
r_3'=[\theta_1;\theta_5]\times[\varphi_5;\varphi_4],\quad r_4'=[\theta_6;\theta_1]\times[\varphi_6;\varphi_5].$$
Then~$\Pi\mapsto\Pi'$ is a flype (the roles of~$r_i,r_i'$, $i=1,2,3,4$, here are the same
as in Definition~\ref{flype-def}) and~$\Pi'$ represents~$\Psi'$.

For the `simplified' version of the move the argument is exactly the same, one only has
to ignore rectangles~$r_6$ and~$r_7$, which are now not present in the diagrams.

\medskip\noindent{\it Case 4\emph: $\Psi\xmapsto\chi\Psi'$ is a non-special commutation}.\\
We use the notation from Definition~\ref{coummutation-def} and the construction
from the proof of Proposition~\ref{representing-movie-prop}.

Let~$\varphi_1',\varphi_2',\ldots,\varphi_k'\in\mathbb S^1$ (respectively,
$\varphi_1'',\varphi_2'',\ldots,\varphi_l''\in\mathbb S^1$) be the points
involved in the first (respectively, the second) event of~$\Psi$ occurring in
the interval~$[\theta_1,\theta_3]$, and let~$\theta'$ (respectively, $\theta''$)
be the exact moment of this event. Exchangeability of the events
means that there are a region~$R_2$
of~$\Psi(\theta'')=\Psi(\theta'+0)$ and points~$\varphi_1,\varphi_2\in R_2$
such that
$$\varphi_1',\varphi_2',\ldots,\varphi_k'\in(\varphi_2;\varphi_1),\quad
\varphi_1'',\varphi_2'',\ldots,\varphi_l''\in(\varphi_1;\varphi_2).$$
Moreover, there are regions~$R_1$ of~$\Psi(\theta')=\Psi(\theta_1)$ and~$R_3$
of~$\Psi(\theta''+0)=\Psi(\theta_3)$ also containing~$\varphi_1,\varphi_2$.

Choose~$\theta_2\in(\theta';\theta'')$
and a rectangular diagram of a surface~$\Pi$ representing~$\Psi$ such that:
\begin{enumerate}
\item
$R_\Pi(\theta_i)=R_i$, $i=1,2,3$;
\item
on the intervals~$(\theta_1;\theta')$, $(\theta';\theta_2)$, $(\theta_2;\theta'')$, $(\theta'';\theta_3)$,
the diagram~$\Pi$ realizes the shortest paths:
\begin{itemize}
\item
from~$R_1$ to~$R_\Pi(\theta')$ on the tree~$T(\Psi(\theta'))$,
\item
from~$R_\Pi(\theta'+0)$ to~$R_2$ on the tree~$T(\Psi(\theta''))$,
\item
from~$R_2$ to~$R_\Pi(\theta'')$ on the tree~$T(\Psi(\theta''))$, and
\item
from~$R_\Pi(\theta''+0)$ to~$R_3$ on the tree~$T(\Psi(\theta_3))$,
\end{itemize}
respectively.
\end{enumerate}

Then~$\Pi$ satisfies all the conditions of Definition~\ref{exchange-def} (with~$\theta_1,\theta_2,
\theta_3,\varphi_1,\varphi_2$ playing the same role), so, an exchange move
can be applied to~$\Pi$, which will exchange the events that occur at~$\theta_1$ and $\theta_2$
in~$\Psi$. This exchange move can be chosen so as to obtain~$\Pi'$ representing~$\Psi'$.

\medskip\noindent{\it Case 5\emph: $\Psi\xmapsto\chi\Psi'$ is a rescaling}.\\
We use the notation from Definition~\ref{rescaling-def}.
Choose any rectangular diagram~$\Pi$ with~$\Psi_\Pi=\Psi$. It suffices to consider the case
when the homeomorphism $f\times g:\mathbb T^2\rightarrow\mathbb T^2$ is identical outside
an annulus containing the whole of exactly one occupied level~$x$ of~$\Pi$, and such that~$x$
is the midline of the annulus. Indeed, a general rescaling
can be decomposed into a sequence of such, `elementary', rescalings.

Let~$\Pi'=(f\times g)(\Pi)$. The transformation~$\Pi\mapsto\Pi'$ can be decomposed
into two basic moves: first, a wrinkle creation, and second, a wrinkle reduction. The
arbitrariness in the definition of these moves allows the second one to be
inverse to the first one in the combinatorial, but not in the geometrical, sense.
\end{proof}

\section{The fixed boundary case}\label{fixed-b-sec}

\begin{prop}\label{fixed-b-prop}
Let~$\Pi$ and~$\Pi'$ be rectangular diagrams of surfaces having common boundary,
$\partial\Pi=\partial\Pi'$. Suppose that there are
a self-homeomorphism~$\phi_1$ of~$\mathbb S^3$ that takes~$\widehat\Pi$ to~$\widehat\Pi'$,
and an isotopy~$(\phi_t)_{t\in[0;1]}$ from~$\phi_0=\mathrm{id}|_{\mathbb S^3}$ to~$\phi_1$
fixed on the boundary~$\partial\widehat\Pi$ and keeping fixed
the tangent plane to the surface at every point of~$\partial\widehat\Pi$.

Then there exists a sequence of basic moves producing~$\Pi'$ from~$\Pi$ such that
all the moves in it are fixed on~$\partial\Pi$, and the composition of all the moves
induces the morphism~$[\phi_1]$.
\end{prop}

\begin{proof}
For brevity, we denote~$\phi_t(\widehat\Pi)$ by~$F_t$, $t\in[0;1]$.
In particular, $F_0=\widehat\Pi$. The singular foliation induced by the open book
decomposition~$\{\mathscr P_\theta\}_{\theta\in\mathbb S^1}$ on the surface~$F_t\setminus\mathbb S^1_{\tau=0}$
will be denoted by~$\mathscr F_t$.

We are going to convert the isotopy~$\phi$
into a sequence of basic moves. This is done in the following four steps.

\medskip\noindent\emph{Step 1}: Alter~$\phi$ slightly to make it `as generic as possible'
with respect to the induced family of foliations~$\mathscr F_t$.

Namely, we disturb~$\phi$ slightly so that,
for all but finitely many values of~$t$, the following holds (in which case the foliation~$\mathscr F_t$
is said to be \emph{generic}):
\begin{enumerate}
\item
the surface~$F_t$ is transverse to~$\mathbb S^1_{\tau=0}$;
\item
the foliation~$\mathscr F_t$, geometrically, has only Morse type singularities in the interior of~$F_t$;
\item
each page contains at most one of the following:
\begin{itemize}
\item
a singularity of~$\mathscr F_t$ at an interior point of~$F_t$;
\item
an arc contained in~$\partial F_t=\partial F_0$;
\end{itemize}
\item
at each point in~$\partial F_t\cap\mathbb S^1_{\tau=1}$ (where the boundary
of the surface usually has a singularity, and the surface is tangent to
a page~$\mathscr P_\theta$), the foliation~$\mathscr F_t$ either has no topological singularity
or has a `half-saddle' singularity (illustrated by the central picture in
the bottom row in Figure~\ref{sing-pic}).
\end{enumerate}
Moreover, we assume that the violations of these rules
that occur at exceptional moments~$t$ cannot be resolved in a one-parametric
family by a small perturbation. Here is the full list of what can happen at such moments:
\begin{enumerate}
\item
there is a tangency point of~$F_t$ and~$\mathbb S^1_{\tau=0}$;
\item
$\mathscr F_t$ has multiple saddle singularities, $t=0$ or~$1$;
\item
there is a page~$\mathscr P_\theta$ containing either two singularities of~$\mathscr F_t$ or an arc
of~$\partial F_t$ and a singularity of~$\mathscr F_t$ outside of this arc;
\item
the foliation~$\mathscr F_t$ has an index zero non-Morse-type geometric singularity in an interior point.
\end{enumerate}
Consider these situations one by one in more detail.

\medskip\noindent\emph{Tangency with~$\mathbb S^1_{\tau=0}$}.
A tangency of~$F_t$ and~$\mathbb S^1_{\tau=0}$ is persistent under small perturbations
if two vertices of~$\mathscr F_t$ are created or
disappear at this moment. We may assume that the second fundamental form
of the surface is non-degenerate at the tangency point. If it is sign-definite, then
a creation of two vertices of~$\mathscr F_t$ is accompanied by
an annihilation of two center singularities. If the second fundamental form
is indefinite, then a creation of two vertices is accompanied by
a creation of two saddle singularities. This is illustrated in Figure~\ref{vertex-creation-fig}.
\begin{figure}[ht]
\includegraphics{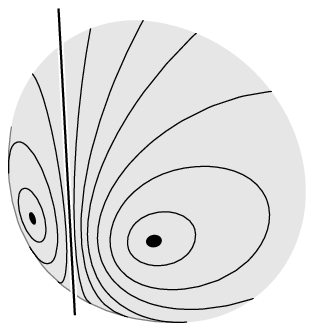}\put(-93,84){$\mathbb S^1_{\tau=0}$}
\hskip.1cm\raisebox{40pt}{$\longleftrightarrow$}\hskip.1cm
\includegraphics{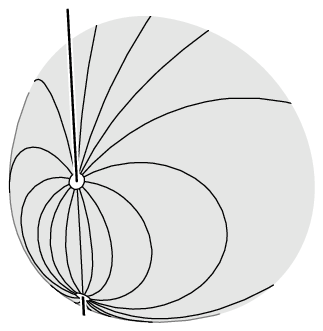}\put(-93,84){$\mathbb S^1_{\tau=0}$}
\hskip1cm
\includegraphics{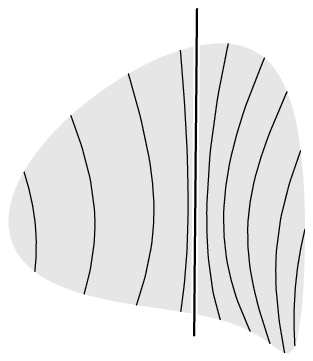}\put(-54,93){$\mathbb S^1_{\tau=0}$}
\hskip.1cm\raisebox{40pt}{$\longleftrightarrow$}\hskip.1cm
\includegraphics{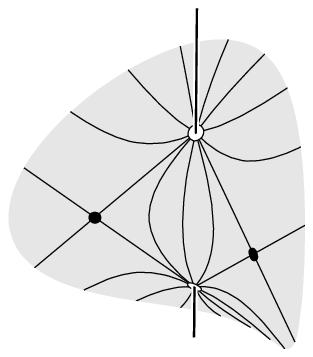}\put(-54,93){$\mathbb S^1_{\tau=0}$}
\caption{Creation and cancellation of two vertices of~$\mathscr F_t$}\label{vertex-creation-fig}
\end{figure}

\medskip\noindent\emph{Multiple saddles}.
By definition, the surfaces~$F_0$ and~$F_1$ are tangent
to the respective pages at all points of their intersections with~$\mathbb S^1_{\tau=1}$.
At all such points, the foliation~$\mathscr F_t$, $t=0,1$, has geometrical singularities,
which are, in general, multiple saddles. The multiplicity of a saddle~$s$
is defined as~$(m-2)/2$ if~$s$ is an interior point of the surface and otherwise as~$(m-2)$,
where~$m$ is the number of separatrices approaching~$s$.
Zero multiplicity saddles are regular points from the topological point of view.

By a small perturbation of the surface all multiple saddles of multiplicities greater than one
can be resolved into ordinary (multiplicity one) saddles, and zero multiplicity saddles at interior points can be smoothed out
to become geometrically regular points.

Note that to specify a resolution it suffices to show how the foliation~$\mathscr F_t$ changes near the singularity.
Indeed, for a fixed~$t$, the embedding~$\phi_t|_{F_0}$ can be viewed as three functions, $\theta$, $\varphi$,
and~$\tau$ defined on~$F_0$ (with~$\theta$ and~$\varphi$ taking values in~$\mathbb S^1$ and
not everywhere defined). At a singularity, the surface is tangent to a page~$\mathscr P_\theta$,
so, small deformations in the $\theta$-direction are safe in the sense that they will
not prevent the respective map~$F_0\rightarrow\mathbb S^3$ from being an embedding.
Deformation in the $\theta$-direction means keeping the functions~$\varphi$ and~$\tau$ fixed
while altering~$\theta$. And to specify the alteration of the latter it suffices
to specify how the foliation~$\mathscr F_t$ changes.

All multiple saddles
must be resolved
at the beginning
of the isotopy and turned into multiple saddles back (probably in a different way) at the last moment.

\medskip\noindent\emph{Two singularities in a single page}.
At each moment~$t$, the values of~$\theta$ at singularities of~$\mathscr F_t$ and on~$\partial F_t$
are called \emph{critical}.
Critical values of~$\theta$ may eventually coincide. In a one-parametric
family of surfaces this situation is persisting if two critical values are going to be exchanged.
Such an unavoidable coincidence of two critical values of~$\theta$
will be referred to as \emph{a collision}.

\medskip\noindent\emph{A non-Morse-type geometrical singularity}.
Finally, there could be moments at which a saddle--center pair of singularities is born or cancelled. At
such moments the foliation~$\mathscr F_t$ has an index zero non-Morse-type singularity
at an interior point.

\medskip\noindent\emph{Step 2}: Turn the isotopy into a sequence
of moves of movie diagrams.

Let~$t_0=0<t_1<t_2<\ldots<t_m=1$ be the moments at which~$\mathscr F_t$ is not generic,
and let~$\varepsilon>0$ be smaller than half the distance between~$t_{i-1}$ and~$t_i$ for any~$i=1,2,\ldots,m$.

By taking~$\varepsilon$ sufficiently small and modifying~$\phi$ slightly if necessary, we may
assume that~$\phi_{t_i+\varepsilon}\circ\phi_{t_i-\varepsilon}^{-1}$
is an elementary isotopy from~$\phi_{t_i-\varepsilon}(F_0)$ to~$\phi_{t_i+\varepsilon}(F_0)$, $i=1,\ldots,m-1$,
supported on
a small ball containing only those singularities and vertices of~$\mathscr F_{t_i\pm\varepsilon}$
that are directly involved in the combinatorial change of the foliation (in the case of a collision, just one of the two collided
singularities).

For~$t\in(t_1;t_{m-1})$ the foliation~$\mathscr F_t$ might have closed leaves, in which
case it cannot be encoded by a movie diagram. We overcome this as follows.

In the proof of~\cite[Proposition 5]{dp17} we described a procedure allowing to
deform an individual generic surface so that all intersections of
the obtained surface with every page become simply connected. It is based
on an idea borrowed from~\cite{IK} and consists of operations that are called
finger moves in~\cite{dp17}. The operations we define below are slightly more general than those in~\cite{dp17}.
We call them geometric finger moves
to distinguish from finger moves of movie diagrams (which are a combinatorial version of
a particular case of geometric finger moves).

\begin{defi}\label{geom-finger-def}
Let~$F$ and~$F'$ be two surfaces related by an elementary isotopy~$\zeta$ supported
on an open $3$-ball~$B$. Suppose that there is a closed $3$-ball~$D\subset B$ with the following
properties:
\begin{enumerate}
\item
$D$ intersects~$\mathbb S^1_{\tau=0}$ in an arc, and each page~$\mathscr P_\theta$ in a $2$-disc;
\item
$D\cap F=\partial D\cap F=d$ is a $2$-disc disjoint from~$\mathbb S^1_{\tau=0}$
and from some page~$\mathscr P_{\theta_0}$;
\item
the boundary of~$d$ consists of two smooth arcs transverse to~$\mathscr F_F$;
\item
there is a self-homeomorphism of~$\mathbb S^3$ identical outside~$B$ which preserve each
page~$\mathscr P_\theta$ and takes~$\overline{(\partial D)\triangle F}$ to~$F'$,
where~$\triangle$ stands for the symmetric difference.
\end{enumerate}
Then the passage~$F\mapsto F'$ assigned with the morphism~$[\zeta]$ is called \emph{a geometric finger move}.
The leaves of~$\mathscr F_F$ intersecting the interior of~$d$ are said to be \emph{broken up}
by this finger move.

If~$F$ and~$F'$ represent movie diagrams~$\Psi$ and~$\Psi'$, respectively,
then the transformation~$\Psi\xmapsto{[\zeta]}\Psi'$ is also called a geometric finger move.
\end{defi}

The effect of a finger move~$F\mapsto F'$ on the foliation~$\mathscr F_F$ depends on
the number of center singularities contained in~$\partial d$ (we use the notation from
Definition~\ref{geom-finger-def}). This can be equal to zero, one, or two.
The corresponding changes of the foliation are shown in Figure~\ref{one-center-fig},
where~$\partial d$ is shown in dashed line.
\begin{figure}[ht]
\begin{tabular}{ccc}
\includegraphics[scale=.6]{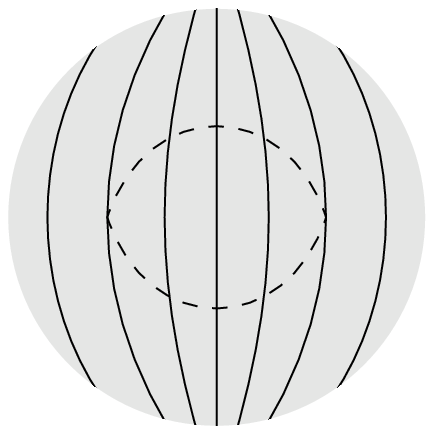}
&\raisebox{38pt}{$\longrightarrow$}&
\includegraphics[scale=.6]{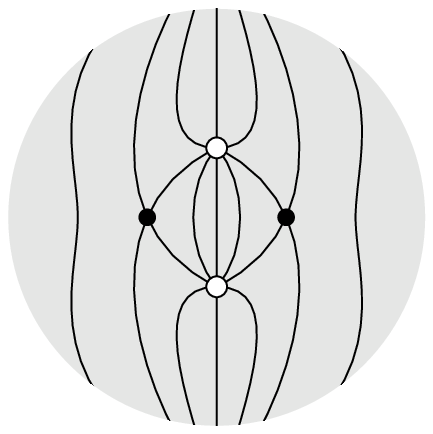}
\\
\includegraphics[scale=.6]{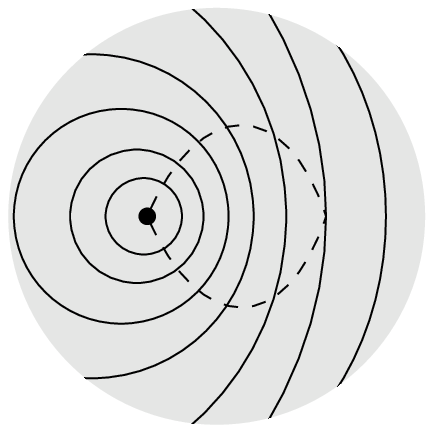}
&\raisebox{38pt}{$\longrightarrow$}&
\includegraphics[scale=.6]{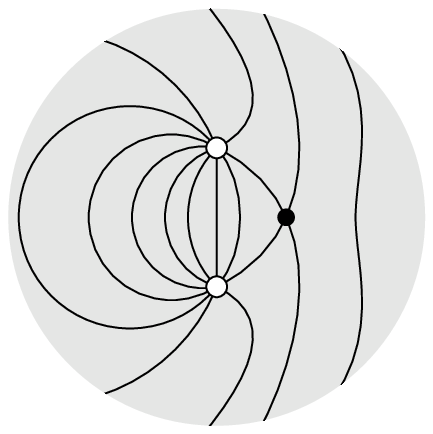}
\\
\includegraphics[scale=.6]{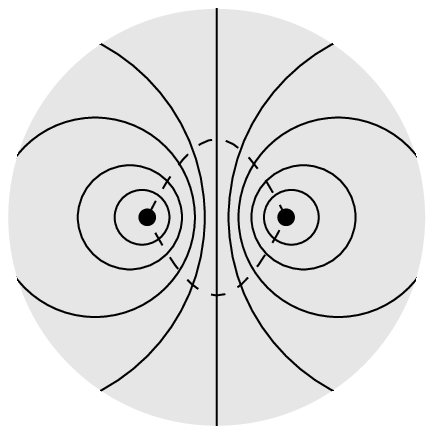}
&\raisebox{38pt}{$\longrightarrow$}&
\includegraphics[scale=.6]{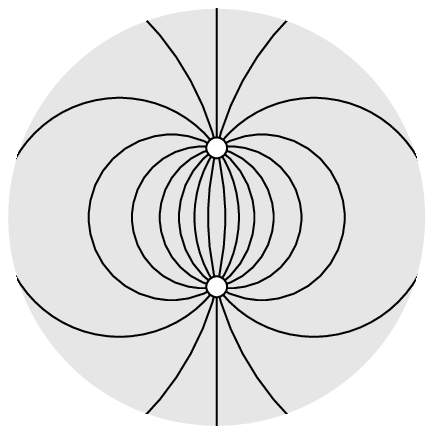}
\end{tabular}
\caption{Change of the foliation under a geometric finger move}\label{one-center-fig}
\end{figure}

In all three cases, the leaves that are said to be broken up by the finger move
are, indeed, modified in a way justifying this terminology. If such a leaf
is a closed curve it is turned into an arc, and if it is an arc it is split into two arcs.
The leaves that do not intersect~$d$ are untouched or just deformed.
The leaves passing through the endpoints of the two smooth arcs transverse to~$\mathscr F_F$
that form~$\partial d$ are modified by a homotopy equivalence.

For~$t\in(0,1)$, denote by~$\mathscr Z_t$ the set of all homeomorphisms~$\iota:\mathbb
S^3\hookrightarrow\mathbb S^3$
that:
\begin{enumerate}
\item
the transition from~$F_t$ to~$\iota(F_0)$ assigned with the morphism
$[\iota\circ\phi_t^{-1}]$ is the composition of a sequence of geometric finger moves
and ambient isotopies preserving the open book decomposition
and fixed on~$\partial F_0$;
\item
$\mathscr F_{\iota(F_0)}$
is generic and has no closed leaves.
\end{enumerate}
As follows from the construction in the proof of~\cite[Proposition 5]{dp17}
the set~$\mathscr Z_t$ is non-empty for all~$t\notin\{t_1,\ldots,t_m\}$.
One can also show that the large arbitrariness in the choice of
a way of breaking up closed leaves of~$\mathscr F_t$ allows one
to proceed from one choice to another by means of geometric finger moves,
their inverses, and rescalings without letting any of the closed leaves to recover.
In other words, for any two elements~$\iota,\iota'$ from~$\mathscr Z_t$
the transition from~$\iota(F_0)$ to~$\iota'(F_0)$ assigned with
the morphism~$[\iota'\circ\iota^{-1}]$ can be decomposed
into geometric finger moves, their inverses, and rescalings performed
\emph{within}~$\mathscr Z_t$.

For each element~$\iota\in\mathscr Z_t$, the surface~$\iota(F_0)$
represents a unique movie diagram,
which we denote by~$\Psi_\iota$.

For any~$i=1,\ldots,m$, small enough~$\varepsilon$,
and~$t\in(t_{i-1}+\varepsilon;t_i-\varepsilon)$, the foliation~$\mathscr F_t$
is generic, which implies that~$\phi_{t_i-\varepsilon}$ can be obtained from~$\phi_{t_{i-1}+\varepsilon}$
by an ambient isotopy preserving the open book decomposition of~$\mathbb S^3$
and fixed on~$\partial F_0$. Therefore, for any~$\iota\in\mathscr Z_{t_{i-1}+\varepsilon}$
there exists~$\iota'\in\mathscr Z_{t_i-\varepsilon}$ such that~$\Psi_\iota\xmapsto{\iota'\circ\iota^{-1}}\Psi_{\iota'}$
is a rescaling, which means
$\mathscr Z_{t_{i-1}+\varepsilon}=\mathscr Z_{t_i-\varepsilon}$.

Suppose that, for each~$i=1,\ldots,m$, two elements~$\iota_i',\iota_i''\in\mathscr Z_{t_{i-1}+\varepsilon}=\mathscr
Z_{t_i-\varepsilon}$ has been chosen. Then the transformation~$\Psi_\Pi\xmapsto{[\phi_1]}\Psi_{\Pi'}$
is the composition of the following ones:
$$\Psi_\Pi\xmapsto{[\iota_1']}\Psi_{\iota_1'}\xmapsto{[\iota_1''\circ{\iota_1'}^{-1}]}\Psi_{\iota_1''}
\xmapsto{[\iota_2'\circ{\iota_1''}^{-1}]}\Psi_{\iota_2'}\xmapsto{[\iota_2''\circ{\iota_2'}^{-1}]}\Psi_{\iota_2''}
\xmapsto{[\iota_3'\circ{\iota_2''}^{-1}]}\ldots\xmapsto{[\iota_m''\circ{\iota_m'}^{-1}]}\Psi_{\iota_m''}
\xmapsto{[\phi_1\circ{\iota_m''}^{-1}]}\Psi_{\Pi'}.$$
It follows from what was just said that each transformation~$\Psi_{\iota_i'}\xmapsto{[\iota_i''\circ{\iota_i'}^{-1}]}\Psi_{\iota_i''}$,
$i=1,\ldots,m$, admits a decomposition into geometric finger moves and rescalings.

Observe that~$\mathscr F_t$ has no closed leaves for~$t=\varepsilon$ or~$t=1-\varepsilon$, hence,
$\phi_\varepsilon\in\mathscr Z_\varepsilon$ and~$\phi_{1-\varepsilon}\in\mathscr Z_{1-\varepsilon}$.
We take~$\phi_\varepsilon$ and~$\phi_{1-\varepsilon}$
for~$\iota_1'$ and~$\iota_m''$, respectively.
We will have that
the passage~$\Psi_\Pi\xmapsto{[\iota_1']}\Psi_{\iota_1'}$ can be decomposed
into a sequence of splittings of an event,
and the passage~$\Psi_{\iota_m''}\xmapsto{[\phi_1\circ{\iota_m''}^{-1}]}\Psi_{\Pi'}$ into
a sequence of mergings of events.

Thus, it remains to do the following two things:
\begin{enumerate}
\item
show how to implement a geometric finger move within a class~$\mathscr Z_t$ by means of allowed moves of movie diagrams;
\item
show how to choose, for each~$i=1,\ldots,m-1$, the elements~$\iota_i'',\iota_{i+1}'$ so that
the transformation~$\Psi_{\iota_i''}\xmapsto{[\iota_{i+1}'\circ{\iota_i''}^{-1}]}\Psi_{\iota_{i+1}'}$
can be decomposed into a sequence of allowed moves of movie diagrams.
\end{enumerate}

We start from the latter.

Suppose that a collision occurs at a moment~$t=t_i$, and~$s_1$, $s_2$ are the two singularities of~$\mathscr F_t$
that belong to the same page. There are the following five cases of collisions treated differently.

\smallskip\noindent\emph{Case 1}: one of~$s_1$ and~$s_2$ is a center singularity of~$\mathscr F_{t_i}$.

Without loss of generality we may assume that~$s_1$ is a center, and~$s_2$ is unaltered
when~$t$ goes from~$t_i-\varepsilon$ to~$t_i+\varepsilon$. We may also assume that~$s_1$
is a local maximum of~$\theta|_{F_t}$, and it moves `forward' when~$t$ goes from~$t_i-\varepsilon$
to~$t_i+\varepsilon$, from a page~$\mathscr P_{\theta_0}$ to a page~$\mathscr P_{\theta_0+\delta}$,
where~$\delta>0$ is small
(the other cases are symmetric to this one).

Clearly, there is an element~$\iota\in\mathscr Z_{t_i+\varepsilon}$ such that~$\iota(F_0)$
is obtained from~$\phi_{t_i+\varepsilon}(F_0)$
by a sequence of geometric finger moves in which the move eliminating~$s_1$  breaks
up a leaf in each page~$\mathscr P_\theta$ with~$\theta\in[\theta_0;\theta_0+\delta)$.
One can see that such~$\iota$ belongs also to~$Z_{t_i-\varepsilon}$.

Thus, in this case the sets~$\mathscr Z_{t_i\pm\varepsilon}$ have a non-empty intersection,
and we choose any~$\iota_i''=\iota_{i+1}'\in\mathscr Z_{t_i-\varepsilon}\cap\mathscr Z_{t_i+\varepsilon}$.

\smallskip\noindent\emph{Case 2}: $s_1$ and~$s_2$ are saddles (one of them may be
at the boundary and have multiplicity~zero or~one) not
connected by a separatrix.

Recall that the surfaces~$\phi_{t_i-\varepsilon}(F_0)$
and~$\phi_{t_i+\varepsilon}(F_0)$ coincide outside of a small ball~$B$ containing one of the
saddles~$s_1,s_2$. This implies that elements~$\iota_0\in\mathscr Z_{t_i-\varepsilon}$
and~$\iota_1\in\mathscr Z_{t_i+\varepsilon}$ can be chosen so that:
\begin{enumerate}
\item
$\phi_{t_i-\varepsilon}\circ\iota_0^{-1}$
and~$\phi_{t_i+\varepsilon}\circ\iota_1^{-1}$ are identical on~$B$;
\item
$\iota_0\circ\iota_1^{-1}$ is identical outside~$B$;
\item
there is an isotopy~$\{\iota_t\}_{t\in[0;1]}$ from~$\iota_0$ to~$\iota_1$
such that for all~$t\in[0;1]\setminus\{1/2\}$ the foliation~$\mathscr F_{\iota_t(F_0)}$ is generic,
and at~$t=1/2$ a collision of~$s_1$ and~$s_2$ occurs.
\end{enumerate}
We let~$\iota_i''$ and~$\iota_{i+1}'$ be~$\iota_0$ and~$\iota_1$, respectively.
One can see
that~$\Psi_{\iota_i''}\xmapsto{[\iota_{i+1}'\circ{\iota_i''}^{-1}]}\Psi_{\iota_{i+1}'}$ is a non-special commutation.

\smallskip\noindent\emph{Special subcase of Case~2}: the foliations~$\mathscr F_{t_i\pm\varepsilon}$ have no closed leaves.

We simply put~$\iota_i''=\phi_{t_i-\varepsilon}$
and~$\iota_{i+1}'=\phi_{t_i+\varepsilon}$.

\smallskip\noindent\emph{Cases 3 and 4}: $s_1$ and~$s_2$ are saddles
which are connected by a separatrix at the moment of collision,
and neither of~$\mathscr F_{\phi_{t_i-\varepsilon}(F_0)}$, $\mathscr F_{\phi_{t_i+\varepsilon}(F_0)}$ has closed leaves,
so, $\phi_{t_i\pm\varepsilon}\in\mathscr Z_{t_i\pm\varepsilon}$.
In these cases, we again put~$\iota_i''=\phi_{t_i-\varepsilon}$
and~$\iota_{i+1}'=\phi_{t_i+\varepsilon}$.
Denote by~$\alpha$ the separatrix connecting~$s_1$ and~$s_2$ at the collision moment.

Assume for the moment that none of~$s_1$ and~$s_2$ lies at the boundary, which means that they are
ordinary saddles of~$\mathscr F_{\phi_{t_i}}$.

Let~$\beta_i$, $i=1,2,3,4,5,6$, be the six separatrices other than~$\alpha$ approaching~$s_1$ and~$s_2$,
numbered in the order in which they intersect the boundary of a small regular neighborhood of~$\alpha$,
and so that~$\beta_1,\beta_2,\beta_3$ approach~$s_1$.
Since~$\mathscr F_{\phi_{t_i-\varepsilon}(F_0)}$ and~$\mathscr F_{\phi_{t_i+\varepsilon}(F_0)}$
have no closed leaves, the foliation~$\mathscr F_{\phi_{t_i}(F_0)}$ has no saddle connection cycles.
Therefore, the separatrices~$\beta_i$ are pairwise distinct, and each~$\beta_i$ approaches
a vertex of~$\mathscr F_{\phi_{t_i}(F_0)}$, which we denote by~$v_i$.

There are two cases here depending on whether the tangent plane to~$\phi_{t_i}(F_0)$ at~$p\in\alpha$
makes a half-twist around~$\alpha$ relative to~$\mathscr P_{\theta_0}$ when~$p$
proceeds from one endpoint of~$\alpha$ to the other.

\smallskip\noindent\emph{Case 3}: No half-twist.

The vertices~$v_1,\ldots,v_6$ follow on~$\mathbb S^1_{\tau=0}$ in the same or opposite cyclic order
as their numeration suggests.
In this case, by an isotopy
fixed outside a small neighborhood of~$\alpha$, the separatrix~$\alpha$ can be collapsed to a point
producing a double saddle out of two ordinary ones.
This means that the transition~$\Psi_{\iota_i''}\xmapsto{[\iota_{i+1}'\circ{\iota_i''}^{-1}]}\Psi_{\iota_{i+1}'}$ can be decomposed into two moves: a merging of events,
and a splitting of an event.

\smallskip\noindent\emph{Case 4}: There is a half-twist.

The cyclic order of~$v_1,\ldots,v_6$ in~$\mathbb S^1_{\tau=0}$ is either
$$v_1,v_2,v_3,v_6,v_5,v_4$$
or the opposite one. In this case, the passage~$\Psi_{\iota_i''}\xmapsto{[\iota_{i+1}'\circ{\iota_i''}^{-1}]}\Psi_{\iota_{i+1}'}$ is a special commutation.

In Cases~3 and~4,
if~$s_1$ lies at the boundary of the surface, the reasoning is exactly the same with five
separatrices instead of six ones ($\beta_2$ and~$v_2$ should be omitted).

\smallskip\noindent\emph{Case 5}: $s_1$ and~$s_2$ are saddles connected by
a separatrix, and at least one of~$\mathscr F_{\phi_{t_i-\varepsilon}}$ and~$\mathscr F_{\phi_{t_i+\varepsilon}}$
has closed leaves. Geometric finger moves allow to break up not only
closed leaves of the foliations~$\mathscr F_t$
but also saddle connections. So, we can proceed as in Case~2 by choosing
the elements~$\iota_0\in\mathscr Z_{t_i-\varepsilon}$
and~$\iota_1\in\mathscr Z_{t_i+\varepsilon}$ satisfying the additional requirement
that, at the moment of the collision that occurs during the isotopy between
them, no saddle connection between~$s_1$ and~$s_2$ is present.

We are done with collisions.

\smallskip

Now suppose that a tangency of~$F_t$ and~$\mathbb S^1_{\tau=0}$ occurs at a moment~$t=t_i$
and two vertices of~$\mathscr F_t$ are created.
In this case, the passage from~$\phi_{t_i-\varepsilon}$ to~$\phi_{t_i+\varepsilon}$
is a geometric finger move, therefore~$\mathscr Z_{t_i+\varepsilon}\subset\mathscr Z_{t_i-\varepsilon}$.
Similarly, If a tangency of~$F_t$ and~$\mathbb S^1_{\tau=0}$ occurs at a moment~$t=t_i$ and
two vertices of~$\mathscr F_t$ disappear, we have~$\mathscr Z_{t_i+\varepsilon}\supset\mathscr Z_{t_i-\varepsilon}$.
In these cases, we again pick any~$\iota_i''=\iota_{i+1}'\in\mathscr Z_{t_i-\varepsilon}\cap\mathscr Z_{t_i+\varepsilon}$.

\smallskip

Suppose that at a moment~$t=t_i$ a saddle--center pair of~$\mathscr F_t$ is being born.
We may assume that~$\phi_{t_i+\varepsilon}\circ\phi_{t_i-\varepsilon}^{-1}$ is identical
outside a small ball~$B$ in which the change of the foliation occurs.

We can find an element~$\iota\in\mathscr Z_{t_i-\varepsilon}$ obtained from~$\phi_{t_i-\varepsilon}$
by a sequence of geometric finger moves such that~$\iota\circ\phi_{t_i-\varepsilon}^{-1}$
is identical on~$B$. Then the same finger moves can be applied
to~$\phi_{t_i+\varepsilon}$ resulting in a self-homeomorphism~$\iota'$ of~$\mathbb S^3$
such that~$\iota'\circ\iota^{-1}$ is identical outside~$B$, and~$\iota'\circ\phi_{t_i+\varepsilon}^{-1}$
is identical on~$B$.

The foliation~$\mathscr F_{\iota'(F_0)}$
has closed leaves, and all them are the result of the creation of the saddle--center pair.
Let~$\iota'(F_0)\mapsto\widetilde F$ be a geometric finger move that breaks up all these leaves,
and~$\widetilde\phi$ be the corresponding elementary isotopy. Put~$\iota''=\widetilde\phi\circ\iota'$.
One can see that~$\iota(F_0)\mapsto\iota''(F_0)$
is also a geometric finger move. This means~$\iota''\in\mathscr Z_{t_i-\varepsilon}\cap\mathscr Z_{t_i+\varepsilon}$,
so, we put $\iota_i''=\iota_{i+1}'=\iota''$.

The situation when a saddle--center pair of~$\mathscr F_t$ is being cancelled at~$t=t_i$ is symmetric to this one.

Now let~$\iota,\iota_1\in\mathscr Z_t$, $t\notin\{t_i\}_{i=0}^m$, be such
that~$\Psi_\iota\xmapsto{\iota_1\circ\iota^{-1}}\Psi_{\iota_1}$ is a geometric finger move.
Let~$\gamma\subset\iota(F_0)$ be an arc transverse to~$\mathscr F_{\iota(F_0)}$ intersecting (except
at the endpoints) exactly those
leaves that are broken up by this move, and let~$[\theta_1;\theta_2]$ be the interval
in which~$\theta|_\gamma$ takes values.

If there are no singularities of~$\mathscr F_{\iota(F_0)}$ in the domain~$\bigcup_{\theta\in[\theta_1;\theta_2]}\mathscr P_\theta$,
then the passage~$\Psi_\iota\xmapsto{\iota_1\circ\iota^{-1}}\Psi_{\iota_1}$ fits into the definition of
a finger move of movie diagrams.

In the general case, one can always find a geometric finger move~$\iota(F_0)\mapsto\iota_0(F_0)$
with~$\iota_0\in\mathscr Z_t$ such that the leaves of~$\mathscr F_{\iota(F_0)}$ broken up
by this move intersect only a small portion of~$\gamma$,
and~$\Psi_\iota\xmapsto{\iota_0\circ\iota^{-1}}\Psi_{\iota_0}$ is a finger move of movie diagrams. There is then an isotopy~$\{\iota_u\}_{u\in[0;1]}$
from~$\iota_0$ to~$\iota_1$ such that, for all~$u$, $\iota(F_0)\mapsto\iota_u(F_0)$ is a geometric finger move
such that the leaves of~$\mathscr F_{\iota(F_0)}$ broken up by this move intersect
a portion of~$\gamma$ that is lengthening when~$u$ grows. This is illustrated in Figure~\ref{moving-saddles-fig}.
\begin{figure}[ht]
\includegraphics[scale=1.4]{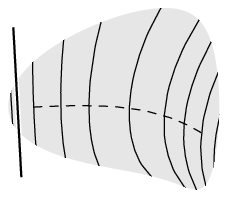}\put(-105,60){$\mathbb S^1_{\tau=0}$}
\put(-60,0){$\iota(F_0)$}\put(-48,38){$\gamma$}
\hskip1cm\raisebox{35pt}{$\longrightarrow$}\hskip1cm
\includegraphics[scale=1.4]{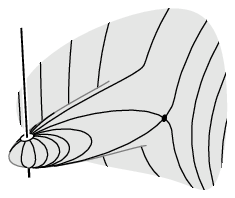}\put(-105,60){$\mathbb S^1_{\tau=0}$}
\put(-60,0){$\iota_0(F_0)$}
\hskip1cm\raisebox{35pt}{$\longrightarrow$}\hskip1cm
\includegraphics[scale=1.4]{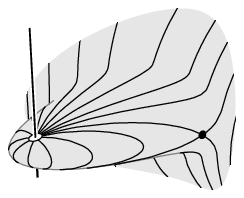}\put(-106,60){$\mathbb S^1_{\tau=0}$}
\put(-60,0){$\iota_1(F_0)$}
\caption{Realizing a geometric finger move}\label{moving-saddles-fig}
\end{figure}

During this isotopy, the foliation~$\mathscr F_{\iota_u}$ will remain generic
except at finitely many moments when a collision of one of the two saddles
created by the geometric finger move~$\iota\mapsto\iota_u$
with another singularity of~$\mathscr F_{\iota_u}$ occurs.
For each of these collision, Case~3, Case~4, or Special subcase of Case~2 considered above
takes place, whereas the effect of the isotopy between the collisions
amounts to a rescaling. Hence, there is a decomposition of
the transformation~$\Psi_{\iota_0}\xmapsto{\iota_1\circ\iota_0^{-1}}\Psi_{\iota_1}$
into a sequence of allowed moves of movie
diagrams.

\medskip
Thus, we have shown that there is a decomposition of~$\Psi_{\Pi}\xmapsto{[\phi_1]}\Psi_{\Pi'}$ into
a sequence of moves of movie diagrams introduced by Definitions~\ref{finger-def}--\ref{rescaling-def},
which completes Step~2.

\medskip\noindent\emph{Step 3}: Represent the evolution of the movie diagram by
basic moves of rectangular diagrams of surfaces.

It follows from Proposition~\ref{representing-movie-prop} and Lemma~\ref{decomp-of-movie-moves-lem} that a sequence of allowed moves of movie diagrams
can be converted into a sequence of basic moves of rectangular diagrams of surfaces
not including stabilizations and horizontal half-wrinkle moves and preserving~$\partial\Pi$.
So, a sequence of such moves produces~$\Pi'$ from~$\Pi$
and induces the morphism~$[\phi_1]$ from~$\widehat\Pi$ to~$\widehat\Pi'$.

\medskip\noindent\emph{Step 4}: Exclude vertical half-wrinkle moves from the constructed sequence of basic moves.

Let
$$\Pi=\Pi_0\mapsto\Pi_1\mapsto\Pi_2\mapsto\ldots\mapsto\Pi_N=\Pi'$$
be the sequence of basic moves obtained at the previous step. By construction, the boundaries of all~$\Pi_j$
are the same, but the type of each boundary vertex (`$\diagup$' or~`$\diagdown$') may vary.
Observe that the type of a boundary vertex changes under the move~$\Pi_j\mapsto\Pi_{j+1}$
only if this move is a vertical half-wrinkle move (since horizontal half-wrinkle moves are not involved),
which changes simultaneously the types of two boundary vertices forming a vertical edge.

Pick~$\delta>0$ smaller than half the distance between any two distinct~$\theta'$, $\theta''$ such that~$m_{\theta'}$
and~$m_{\theta''}$ are occupied meridians of~$\Pi_{j'}$, $\Pi_{j''}$, respectively,
for some~$j',j''\in\{0,1,\ldots,N\}$.

For each~$j=0,1,2,\ldots,N$ we construct a rectangular diagram of a surface~$\Pi_j'$ as follows.
If the types of any~$v\in\partial\Pi$ is the same in~$\Pi_j$ and~$\Pi$, we put~$\Pi_j'=\Pi_j$.
In particular, $\Pi_0'=\Pi_0=\Pi$ and~$\Pi_N'=\Pi_N=\Pi'$.

Suppose there is a vertical edge~$\{v_1,v_2\}\in\partial\Pi$ such that the types of~$v_1,v_2$
in~$\Pi$ and~$\Pi_j$ disagree. Apply a vertical half-wrinkle creation move to~$\Pi_j$
so that:
\begin{enumerate}
\item
the boundary of the diagram is preserved;
\item
the types of~$v_1$, $v_2$ change to the opposite;
\item
the new occupied meridian is at distance~$\delta$ from the meridian containing~$v_1$ and~$v_2$.
\end{enumerate}

Repeat this procedure with the new diagram in place of~$\Pi_j$ until all boundary vertices have
the same type in it as they have in~$\Pi$. Let~$\Pi_j'$ be the obtained diagram.

One can see the following:
\begin{enumerate}
\item
if~$\Pi_j\mapsto\Pi_{j+1}$ is a vertical half-wrinkle creation (respectively, reduction) move, then~$\Pi_j'\mapsto\Pi_{j+1}'$
is a vertical wrinkle creation (respectively, reduction) move or a rescaling;
\item
if~$\Pi_j\mapsto\Pi_{j+1}$ is not a vertical half-wrinkle move, then~$\Pi_j'\mapsto\Pi_{j+1}'$
is a basic move of the same kind as~$\Pi_j\mapsto\Pi_{j+1}$;
\item
all the moves~$\Pi_j'\mapsto\Pi_{j+1}'$, $j=0,1,\ldots,N-1$, are fixed on~$\partial\Pi$;
\item
for all~$j=0,1,\ldots,N$, we have~$\Psi_{\Pi_j'}=\Psi_{\Pi_j}$, which implies that
the morphism from~$\Pi$ to~$\Pi'$ induced by the sequence of moves
$$\Pi=\Pi_0'\mapsto\Pi_1'\mapsto\Pi_2'\mapsto\ldots\mapsto\Pi_N'=\Pi'$$
is still~$[\phi_1]$.\qedhere
\end{enumerate}
\end{proof}

\section{Modifying the boundary}\label{boundary-sec}

\begin{lemm}\label{delete-rect-lem}
Let~$\Pi$ be a rectangular diagram of a surface, and let~$r\in\Pi$ be a rectangle having
exactly one, two, or three consecutive
vertices at~$\partial\Pi$ and such that~$\Pi'=\Pi\setminus\{r\}$ is also a rectangular diagram of a surface.
Let also~$\phi:(\mathbb S^3,\widehat\Pi')\rightarrow(\mathbb S^3,\widehat\Pi)$ be a homeomorphism isotopic
to the identity \emph(in the class of maps~$(\mathbb S^3,\widehat\Pi')\rightarrow(\mathbb S^3,\widehat\Pi)$\emph). Then the transformation~$\Pi\mapsto\Pi'$
assigned with the morphism~$[\phi]$
can be decomposed into basic moves fixed on the boundary components disjoint from~$V(r)$.
\end{lemm}

\begin{proof}
There are the following three cases to consider.

\medskip\noindent\emph{Case 1}: $r$ shares two common vertices with the other rectangles in~$\Pi$.

The sought for decomposition is shown in Figure~\ref{remove-rect-2-fig}. First, we apply an exchange move
that replaces~$r$ with a `thin' rectangle, and then remove it by a half-wrinkle reduction move.
\begin{figure}[ht]
\includegraphics{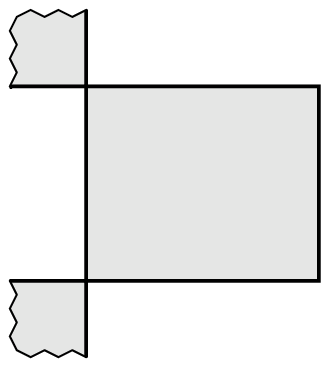}
\hskip.5cm\raisebox{50pt}{$\longrightarrow$}\hskip.5cm
\includegraphics{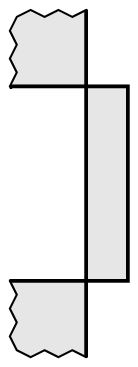}
\hskip.5cm\raisebox{50pt}{$\longrightarrow$}\hskip.5cm
\includegraphics{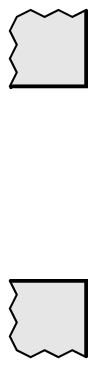}
\caption{Removing a rectangle having two common
vertices with other rectangles}\label{remove-rect-2-fig}
\end{figure}

\medskip\noindent\emph{Case 2}: $r$ shares a single common vertex with the other rectangles in~$\Pi$.

The decomposition is shown in Figure~\ref{remove-rect-1-fig}. First, we apply
an exchange move, then a half-wrinkle creation move, and finally, a destabilization move.
\begin{figure}[ht]
\includegraphics{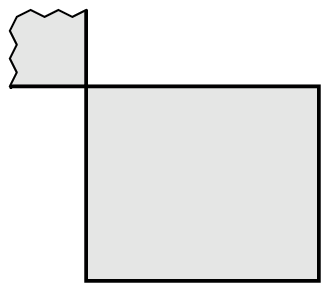}
\hskip.5cm\raisebox{50pt}{$\longrightarrow$}\hskip.5cm
\includegraphics{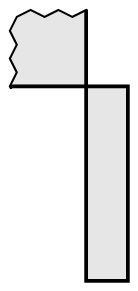}
\hskip.5cm\raisebox{50pt}{$\longrightarrow$}\hskip.5cm
\includegraphics{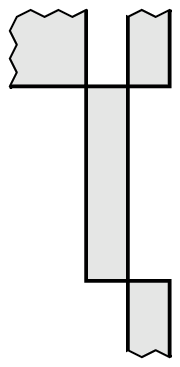}
\hskip.5cm\raisebox{50pt}{$\longrightarrow$}\hskip.5cm
\includegraphics{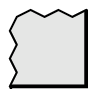}
\caption{Removing a rectangle having a single common
vertex with other rectangles}\label{remove-rect-1-fig}
\end{figure}

\medskip\noindent\emph{Case 3}: $r$ shares three common vertices with the other rectangles in~$\Pi$.

We use the two previously resolved cases. First, we add, one by one, four new rectangles as
shown in Figure~\ref{remove-rect-3-fig}, each sharing a single vertex with others
at the moment of the addition. Then we apply a flype, and then remove five rectangles,
each sharing either one or two vertices with others at the moment of the removal.
\begin{figure}[ht]
\begin{tabular}{ccc}
\includegraphics{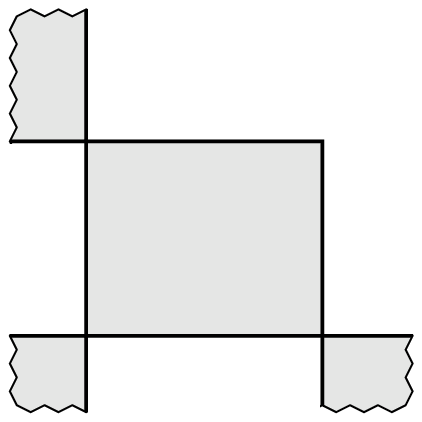}
&\hskip.5cm\raisebox{50pt}{$\longrightarrow$}\hskip.5cm&
\includegraphics{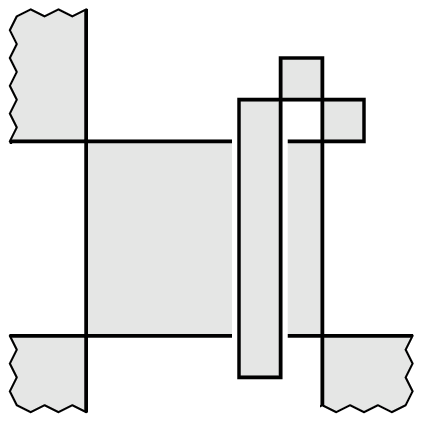}\\[3mm]
&&$\bigl\downarrow$\\[3mm]
\includegraphics{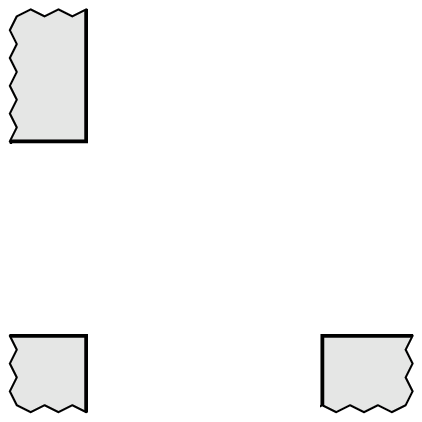}
&\hskip.5cm\raisebox{50pt}{$\longleftarrow$}\hskip.5cm&
\includegraphics{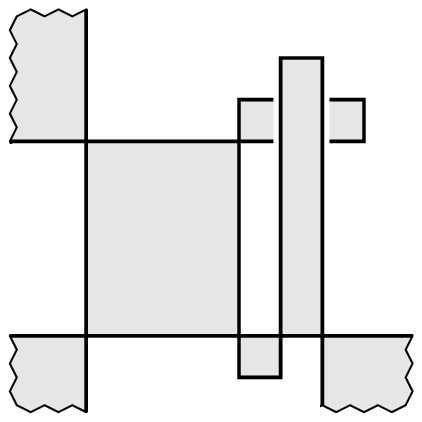}
\end{tabular}
\caption{Removing a rectangle having three common
vertices with other rectangles}\label{remove-rect-3-fig}
\end{figure}
\end{proof}

\begin{lemm}\label{homot-eq-lem}
Let~$\Pi$ and~$\Pi'$ be rectangular diagrams of surfaces such that~$\Pi\subset\Pi'$
and the inclusion~$\widehat\Pi\hookrightarrow\widehat\Pi'$ is a homotopy equivalence.
Then there exists a sequence of basic moves
\begin{equation}\label{bbm-seq-eq}
\Pi=\Pi_0\mapsto\Pi_1\mapsto\Pi_2\mapsto\ldots\mapsto\Pi_N=\Pi'
\end{equation}
such that:
\begin{enumerate}
\item
each move~$\Pi_i\mapsto\Pi_{i+1}$ is fixed on the common components of~$\partial\Pi$ and~$\partial\Pi'$;
\item
the composition of moves~\eqref{bbm-seq-eq} induces the morphism~$\widehat\Pi\rightarrow\widehat\Pi'$
that is represented by any homeomorphism~$\phi:(\mathbb S^3,\widehat\Pi)\rightarrow(\mathbb S^3,\widehat\Pi')$
isotopic to the identity in the class of maps~$(\mathbb S^3,\widehat\Pi)\rightarrow(\mathbb S^3,\widehat\Pi')$.
\end{enumerate}
\end{lemm}

\begin{proof}
The claim follows from Lemma~\ref{delete-rect-lem} by induction in the number of rectangles in~$\Pi'\setminus\Pi$.
\end{proof}

\begin{proof}[Proof of Theorem~\ref{main-theo}]
The implication~$\text{(ii)}\Rightarrow\text{(i)}$ is easy and left to the reader. We will prove the
inverse one.
So, suppose that~(i) holds true.

It suffices to prove the theorem in the case when~$\partial\widehat\Pi\cap\partial\widehat\Pi'=L$. Indeed,
by a small perturbation of~$\Pi'$ we can always obtain a diagram~$\Pi''$ such
that~$\partial\widehat\Pi\cap\partial\widehat\Pi''=\partial\widehat\Pi'\cap\partial\widehat\Pi''=L$.
Then the passage~$\Pi\mapsto\Pi'$
will decompose in the following two:~$\Pi\mapsto\Pi''\mapsto\Pi'$
with the former assigned a morphism~$[\phi']$ with~$\phi'$ close to~$\phi$,
and the latter the morphism~$[\phi\circ{\phi'}^{-1}]$. An application of the theorem
to these transformations yields the general case.

So, from now on, we assume~$\partial\widehat\Pi\cap\partial\widehat\Pi'=L$.

Let~$\{\phi_t\}_{t\in[0;1]}$ be an isotopy from~$\mathrm{id}|_{\mathbb S^3}$ to~$\phi$
fixed on~$L$ and preserving the tangent plane to~$\widehat\Pi$ along~$L$.
For selected moments~$t_0,t_1,\ldots,t_N\in[0,1]$ we denote~$\phi_{t_i}(\widehat\Pi)$ by~$F_i$.

The isotopy~$\{\phi_t\}_{t\in[0;1]}$ can be perturbed slightly, and moments~$t_0=0<t_1<t_2<\ldots<t_{N-1}<t_N=1$
can be chosen so that, for all~$i=1,\ldots,N$, one of the following holds:
\begin{enumerate}
\item
we have~$\partial F_{i-1}=\partial F_i$, and the tangent plane to~$F_{i-1}$ at any point~$p\in\partial F_{i-1}$
is preserved during the isotopy~$\{\phi_t\}_{t\in[t_{i-1};t_i]}$;
\item
we have~$\phi_t(F_0)\subset F_i$ for all~$t\in[t_{i-1};t_i]$, and each connected component of~$\partial F_i$ is either
contained in~$\partial F_{i-1}$ or disjoint from~$\bigcup_{j=0}^{i-1}\partial F_j$;
\item
we have~$\phi_t(F_0)\subset F_{i-1}$ for all~$t\in[t_{i-1};t_i]$, and each connected component of~$\partial F_i$ is either
contained in~$\partial F_{i-1}$ or disjoint from~$\bigcup_{j=0}^{i-1}\partial F_j$.
\end{enumerate}

Denote by~$L_*$ the link~$\bigcup_{i=0}^N\partial F_i$. For each~$i=0,\ldots,N$ denote also
by~$L_i$ the union of connected components of~$L_*$ that are contained by whole in~$F_i$.

It follows from~\cite[Lemma~2]{dp17} and a slight generalization of~\cite[Theorem~1]{dp17} that
there exist a tubular neighborhood~$U$ of~$L$, a homeomorphism~$\psi:\mathbb S^3\rightarrow\mathbb S^3$
identical on~$U$, a rectangular diagram of a link~$R_*$, and rectangular diagrams of surfaces~$\Pi_0,\Pi_1,\ldots\Pi_{N-1},\Pi_N$ such that the following holds:
\begin{enumerate}
\item
$\psi(L_*)=\widehat{R_*}$;
\item
for each~$i=0,1,\ldots,N$, there is an isotopy from~$\psi(F_i)$ to~$\widehat\Pi_i$
fixed on~$\psi(L_i)$ and preserving the tangent plane to~$\psi(F_i)$ at
any~$p\in\psi(L_i)$.
\end{enumerate}

The generalization of~\cite[Theorem~1]{dp17} needed here consists in requesting that a given link
contained in the surface~$F$ also becomes `rectangular' in a rectangular presentation
of~$F$. This does not require any essential change of the proof. Indeed, Proposition~5 in~\cite{dp17},
which is the key ingredient, treats even more general case which allows
to make `rectangular' any graph embedded in~$F$.

Now each passage~$\Pi_{i-1}\mapsto\Pi_i$, $i=1,\ldots,N$, as well as~$\Pi\mapsto\Pi_0$ and~$\Pi_N\mapsto\Pi'$
(assigned with the morphism obtained naturally from the construction) decompose into basic moves
either by Proposition~\ref{fixed-b-prop} or Lemma~\ref{homot-eq-lem}.
This completes the proof of the theorem.
\end{proof}

\end{document}